\documentclass[11pt,leqno]{article}
\usepackage{microtype} % typographical improvements
\usepackage[margin=2.5cm]{geometry} % page geometry

\usepackage{url}
\usepackage[pagebackref]{hyperref}
\usepackage{amsfonts}
\usepackage{amsthm}
\usepackage{amsmath}
\usepackage{amscd}
\usepackage{amssymb}
\usepackage{caption,subcaption}
\usepackage[shortlabels]{enumitem}
\usepackage{mathtools}
\usepackage{tikz}
\usepackage{tikz-cd}
\usetikzlibrary{arrows,arrows.meta}
\tikzset{>=latex}
\tikzcdset{arrow style=tikz, diagrams={>=latex}}
\usepackage{xcolor}

\def\CC {{\mathbb C}}     %% complex numbers
\def\GG {{\mathbb G}}     %% group
\def\PP {{\mathbb P}}     %% projective
\def\QQ {{\mathbb Q}}     %% rationals
\def\RR {{\mathbb R}}     %% real numbers
\def\ZZ {{\mathbb Z}}     %% integers

\def\ring#1{\ifmmode \mathaccent'027 #1\else \rm\accent'027 #1\fi}

\newcommand{\mR}{{\mathbf R}}
\newcommand{\mK}{{\mathbf K}}
\newcommand{\mbk}{{\mathbf k}}

\newcommand{\Ee}{{\mathcal E}}
\newcommand{\Ff}{{\mathcal F}}
\newcommand{\Pp}{{\mathcal P}}

\newcommand{\zz}{{\mathbb Z}}

\def\ol  {\overline}

\def\im {\mathrm{Im}}

\def\mc {\mathcal}
\def\mk {\mathfrak}

\def\Ob {\mathrm{Ob}}
\def\Hom {\mathrm{Hom}}

\mathchardef\mhyphen="2D

\theoremstyle{plain}
\newtheorem{theorem}{Theorem}[section]

\newtheorem{lemma}[theorem]{Lemma}
\newtheorem{prop}[theorem]{Proposition}
\newtheorem{coro}[theorem]{Corollary}
\theoremstyle{definition}
\newtheorem{df}[theorem]{Definition}

\newtheorem{ex}[theorem]{Example}
\newtheorem{remark}[theorem]{Remark}
\theoremstyle{plain}
\newtheorem{prob}[theorem]{Problem}

\newtheorem{q}[theorem]{Question}

\newtheorem{conj}[theorem]{Conjecture}

\numberwithin{equation}{section}

\newlength{\miniwidth}

\definecolor{boundaryclr}{HTML}{FFB380}

\newcommand\hexbg{
\draw[draw=boundaryclr,thick] (0,0) circle [radius=2]; 
\draw[draw=black!30!white,thick] (1,1.73) to (-1,1.73) to (-2,0) to (-1,-1.73) to (1,-1.73) to (2,0);
}

\newcommand{\Addresses}{{% additional braces for segregating \footnotesize
  \bigskip
  \footnotesize

	(F.~Haiden) \textsc{University of Oxford, Mathematical Institute, Andrew Wiles Building, Woodstock Road, Oxford OX2 6GG, UK} \par\nopagebreak
	\textit{E-mail:} \texttt{Fabian.Haiden@maths.ox.ac.uk}
	\medskip
	
	(L.~Katzarkov) \textsc{University of Miami, Coral Gables, FL} and \textsc{Institute of Mathematics and Informatics, Bulgarian Academy of Sciences, Acad. G. Bonchev Str. bl. 8, 1113, Sofia, Bulgaria} and \textsc{HSE University, Russian Federation}\par\nopagebreak
	\textit{E-mail:} \texttt{lkatzarkov@gmail.com}
	\medskip
	
	(C.~Simpson) \textsc{Universit\'e Côte d’Azur, CNRS, LJAD, France}\par\nopagebreak
	\textit{E-mail:} \texttt{carlos.simpson@univ-cotedazur.fr}
}}

\title{Spectral networks and stability conditions for Fukaya categories with coefficients}
\author{F. Haiden, L. Katzarkov, C. Simpson}
\date{}

\begin{document}

\maketitle

\begin{abstract}
    Given a holomorphic family of Bridgeland stability conditions over a surface, we define a notion of spectral network which is an object in a Fukaya category of the surface with coefficients in a triangulated DG-category.
    These spectral networks are analogs of special Lagrangian submanifolds, combining a graph with additional algebraic data, and conjecturally correspond to semistable objects of a suitable stability condition on the Fukaya category with coefficients.
    They are closely related to the spectral networks of Gaiotto--Moore--Neitzke.
    One novelty of our approach is that we establish a general uniqueness results for spectral network representatives. We also verify the conjecture in the case when the surface is disk with six marked points on the boundary and the coefficients category is the derived category of representations of an $A_2$ quiver.
    This example is related, via homological mirror symmetry, to the stacky quotient of an elliptic curve by the cyclic group of order six.
\end{abstract}

\tableofcontents

\section{Introduction}

Homological mirror symmetry (HMS), proposed by Kontsevich~\cite{kontsevich_hms}, is a conjectural equivalence between the Fukaya category of a Calabi--Yau manifold $M$ and the derived category of coherent sheaves on its mirror dual $W$.
The Fukaya category depends only on the symplectic topology of $M$ while the category of coherent sheaves depends only on the complex structure on $W$.
The full geometry of the Calabi--Yau space is conjecturally encoded in terms of a triangulated category with stability condition in the sense of Bridgeland~\cite{bridgeland07}.
Constructing stability conditions on general derived categories of geometric origin is a major open problem~\cite{bmt1,joyce_conj}.

More recently, an extension of HMS involving families of categories with stability condition, initially suggested by Kontsevich, is starting to emerge.
On the symplectic side of mirror symmetry this involves locally constant families of categories (more generally: perverse schobers~\cite{kapranov_schechtman}) with holomorphically varying stability conditions, while on the algebro-geometric side one considers flat families of categories equipped with locally constant stability conditions.
In this paper we are concerned with the symplectic side of this correspondence. 

\vspace{\baselineskip}

\begin{tabular}{c|c}
     A-side & B-side \\
     \hline
     symplectic manifold & complex manifold \\
     compatible complex structure & K\"ahler form \\
     schober & flat family of categories \\
     Fukaya category with coefficients & category of global sections \\
     holomorphic family of stability conditions & locally constant  family of stability conditions
\end{tabular}

\vspace{\baselineskip}

An important class of examples of schobers with holomorphic family of stability conditions comes from spectral curves $\widetilde{C}\subset T^*C$ where $C$ is a Riemann surface. These in turn arise as spectra of Higgs bundles and are parameterized by the base of the $SL_n$ Hitchin system.
While studying the geometry of the Hitchin system, physicists Gaiotto, Moore, and Neitzke introduced the notion of \textit{spectral networks}~\cite{gmn} which are, roughly speaking, graphs on $C$ with edges labelled by pairs of sheets of $\widetilde{C}$ and so that each edge is a leaf of the foliation determined by the difference of tautological 1-forms on the two sheets. (See~\cite{gmn_snakes,glm,hollands_neitzke, williams, longhi_park,gabella,longhi_park2,glpy,longhi} for some further developments.)
In the special case of $SL_2$, spectral networks are just geodesics with respect to the flat metric coming from a quadratic differential, and it was shown by Bridgeland--Smith~\cite{bs} that there is a 3-d Calabi--Yau (3CY) category with stability condition (depending on a quadratic differential) so that stable objects correspond to finite length geodesics. This was extended to the case of quadratic differentials without higher order poles in~\cite{h_3cyteich}.
The case of $SL_n$, $n\geq 3$ is much harder, but one expects a similar statement: A 3CY category $\mc C$ such that the Hitchin base embeds in the space of stability conditions $\mathrm{Stab}(\mc C)$ and so that semistable objects correspond to (finite length) spectral networks.

Here, we propose a new definition of spectral networks adapted to that case of schobers on Riemann surfaces with holomorphic family of stability conditions. 
More precisely, we restrict to the simplest case when the schober has no singular points or monodromy and is thus given by a single triangulated DG-category $\mc E$.
The first task is to define a Fukaya category of a surface $S$ with marked points $M\subset S$ and coefficients in $\mc E$, denoted $\mc F(S,M;\mc E)$.
This is a category of the Novikov ring $\mathbf R$, and one can think of it as providing a ``non-archimedean K\"ahler metric'' on the category $\mc F(S,M,\mc E)_{\mathbf K}:=\mc F(S,M,\mc E)\otimes_{\mathbf R}\mathbf K$ over the Novikov field $\mathbf K$.
Objects of $\mc F(S,M;\mc E)$ are, roughly speaking, graphs in $S$ with edges labelled by objects of $\mc E$, fitting together in an exact sequence at the vertices. The precise definition takes into account certain disk corrections coming from regions of $S$ cut out by the graph.

Spectral networks, in our sense, are objects of $\mc F(S,M;\mc E)$ satisfying a certain constant phase condition, see Subsection~\ref{subsec_spectral}, similar to special Lagrangian submanifolds.
With this definition we can make the following conjecture precise (Conjecture~\ref{main_conj} in the main text, which includes the formula for the central charge).

\begin{conj}
\label{main_conj_intro}
There is a stability condition on the Fukaya category with coefficients $\mc F(S,M;\mc E)_{\mathbf K}$ with semistable objects of phase $\phi$ those objects $X$ which have a spectral network representative of phase $\phi$, i.e. an object in $\mc F(S,M;\mc E)$ which satisfies the spectral network condition and becomes isomorphic to $X$ after base-change to $\mathbf K$.
\end{conj}

We establish the following general uniqueness result for spectral network representatives, assuming the family of stability conditions is locally constant, see Theorem~\ref{thm_specnet_unique} in the main text.

\begin{theorem}
Let $X,X'\in \mc F(S,M;\mc E)$ be spectral networks  representing the same object in the Fukaya category $\mc F(S,M;\mc E)_{\mathbf K}$. 
Then $X$ and $X'$ have the same support graph $G\subset S$ and  are S-equivalent in a certain sense (made precise in the main text).
\end{theorem}

Our main result is to verify the above conjecture --- the existence of spectral network representatives --- in the following example.

\begin{theorem} 
The main conjecture is true in the particular case when $S$ is a disk with six marked points on the boundary which are the vertices of a regular hexagon, and $\mc E$ is the category $\mc A_2$ of representations of the $A_2$-type quiver with a certain (most symmetric) stability condition.
\end{theorem}

This is Theorem~\ref{thm_a5a2} in the main text.
It is an example of a (constant) 3-differential related to a stability condition.
The usual Fukaya category of the disk with $n+1$ marked points on the boundary is equivalent to the category $\mc A_n$ of representations of an $A_n$-type quiver over a field.
The category $\mc F(S,M;\mc E)_{\mathbf K}$ is category of representations of the $A_n$-type quiver in the category $\mc E$.
The case $n=5$, $\mc E=\mc A_2$ is special in that $\mc F(S,M;\mc E)_{\mathbf K}$ turns out to be equivalent to $D^b(E/(\ZZ/6))$ where $E$ is an elliptic curve over $\mathbf K$ with $\ZZ/6$ action and the stacky (orbifold) quotient is taken.
This equivalence, which can be seen as an instance of homological mirror symmetry and was proven in one form by Ueda~\cite{ueda}, plays a central part in our proof.
Since we can transfer the stability condition from the B-side, it only remains to construct spectral network representatives of all semistable objects.
A typical such spectral network is shown in Figure~\ref{fig_specnet}.
They are constructed by a recursive procedure starting from a basic 1-parameter family, see Section~\ref{sec_a5a2} for details.

\begin{figure}
\begin{center}
\includegraphics[scale=.5]{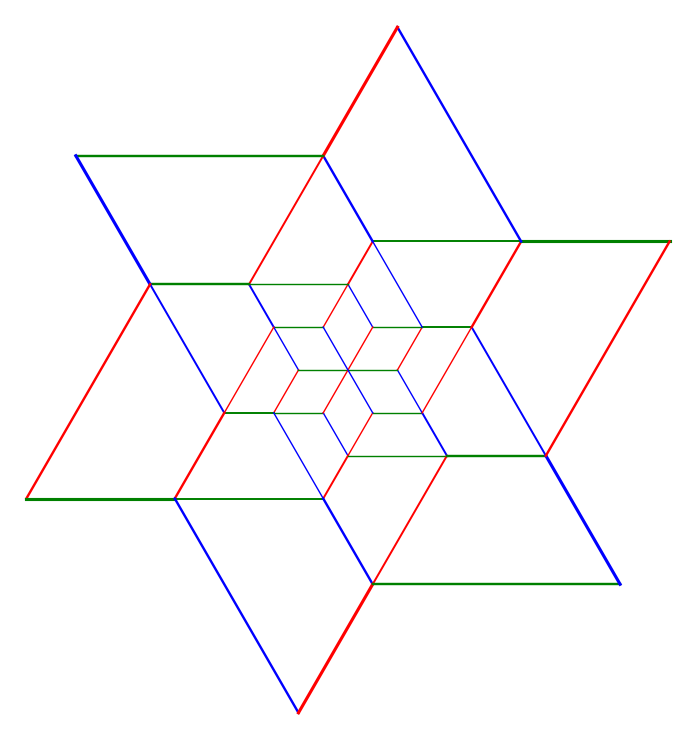}
\end{center}
\caption{Example of a spectral network for $A_5\otimes A_2$.}
\label{fig_specnet}
\end{figure}

Let us comment on the relation between Conjecture~\ref{main_conj_intro} and the Thomas--Yau conjecture~\cite{thomas01, thomas_yau02,joyce_conj} on special Lagrangian submanifolds and stability conditions.
Kontsevich has suggested a hypothetical generalization of Conjecture~\ref{main_conj_intro} to symplectic manifolds $M$ of arbitrary dimension and holomorphic families of stability conditions on any perverse schober on $M$.
This would present a generalization of the Thomas--Yau conjecture to Fukaya categories with coefficients.
On the other hand, such a conjecture can be heuristically interpreted as a degenerate limiting case of the Thomas--Yau conjecture as follows.
Suppose a Calabi--Yau manifold $X$ is fibered over a base $M$. Taking the fiberwise Fukaya category gives rise, conjecturally, to a schober $\mc E$ on $M$ equipped with a holomorphic family of stability conditions such that the Fukaya category of $M$ with coefficients in $\mc E$ recovers the Fukaya category of $X$, at least near the limit where the volume of the fibers is very small. 
At the limit where the volume of the fibers vanishes, special Lagrangian submanifolds in $X$ should project to singular Lagrangians in $X$.
If $\dim_{\RR}X=2$, these are (generalized) spectral networks, while for higher dimensional $X$ they are singular lagrangians which are a generalization of spectral networks to higher dimensions.

The text is organized as follows.
In Section~\ref{sec_alg} we discuss various aspects of curved $A_\infty$-categories and lay the algebraic foundation for the subsequent Section~\ref{sec_aside}, where the Fukaya category of a marked surface $(S,M)$ with coefficients in a triangulated DG-category $\mc E$, denoted $\mc F(S,M;\mc E)$, is constructed.
Spectral networks are defined in Section~\ref{sec_general} which also contains the statement of the main conjecture and a uniqueness result.
In Section~\ref{sec_ana2} we specialize to the case where either the base surface is a disk, the fiber category is $\mc A_2$, or both. 
The main conjecture is verified in the simplest non-trivial case: $S$ is a disk with three marked points on the boundary and $\mc E=\mc A_2$.
Finally, the example where $S$ is a disk with six marked points on the boundary and $\mc E=\mc A_2$ is discussed in Section~\ref{sec_a5a2}.

\textbf{Acknowledgements}: Pranav Pandit was initially part of this project and we thank him for his input during our discussions. 
We also thank Maxim Kontsevich for stimulating discussions.
F.~H. was partially supported by a Titchmarsh Research Fellowship.
L.~K. was partially supported by an NSF Grant, a Simons Investigator Award HMS, the National Science Fund of Bulgaria, National Scientific Program ``Excellent Research and People for the Development of European Science'' (VIHREN), Project No. KP-06-DV-7, and an HSE University Basic Research Program.
C.~S. was supported by the 3IA Côte d’Azur (ANR-19-P3IA-0002), DuaLL (ERC Horizon 2020 number 670624), the program ``Moduli of bundles and related structures'' (ICTS/mbrs2020/02), by a grant from the Institute for Advanced Study, and by the University of Miami. 
L.~K. and C.~S. were supported by a Simons Investigator Award HMS.

%%%%%%%%%%%%%%%%%%%%%%%%%%%%%%%%%%%%%%%%%%%%%%%%%%%%%%%%%%%%%%%%%%%%%%%%%%

\section{Algebraic preliminaries}
\label{sec_alg}

This section collects some definitions and results related to $A_\infty$-categories in preparation for Section~\ref{sec_aside}.
In particular, we emphasize working with \textit{curved} $A_\infty$-categories throughout.
We recall their definition and some basic constructions in the first subsection.
The short Subsection~\ref{subsec_tensor} contains a definition of the tensor product of a DG-category with a curved $A_\infty$-category.
A type of bar construction which starts from an $A_\infty$-algebra and yields a curved DG-algebra is discussed in Subsection~\ref{subsec_bar}.
Subsection~\ref{subsec_fibdg} contains a construction of fibrations between $A_\infty$-categories.
Finally, we discuss curved $A_\infty$-categories over a Novikov ring in Subsection~\ref{subsec_catnov}.

\subsection{Curved $A_\infty$-categories}
\label{subsec_curved}

$A_\infty$-categories are a generalization of DG-categories where associativity holds only up to a homotopy which is part of an infinite tower of higher coherences.
See \cite{keller_a_infinity} for an introduction to the subject and \cite{lefevre_hasegawa,seidel08,ks_ainfty} for more in-depth treatments.
Curved $A_\infty$-categories are yet a further generalization where each object $X$ has a \textit{curvature} $\mk m_0\in\Hom^2(X,X)$. 
Such structures naturally appear in symplectic topology~\cite{FOOO} and deformation theory.
Typically, one requires $\mk m_0$ to be small in some sense, and one way to make this precise is to work with filtered categories, as will be described below.
We refer the reader to~\cite{positselski,armstrong_clarke,deken_lowen} for some related works.

Fix a ground field $\mathbf k$.
We begin with some preliminaries on filtrations.
An \textbf{$\RR$-filtration} on a vector space $V$ over $\mathbf k$ is given by subspaces $V_{\geq \lambda}\subseteq V$, $\lambda\in\RR$, with $V_{\geq\lambda}\subseteq V_{\geq\mu}$ for $\lambda\geq\mu$.
When dealing with $\ZZ$-graded vector spaces, the filtration is required to be compatible with grading in the sense that each $V_{\geq\lambda}$ is a direct sum of its homogeneous components.
Furthermore, we usually require filtrations to be complete and separated, i.e. $\widehat{V}:=\varprojlim V/V_{\geq \lambda}=V$ and $\bigcap_{\lambda} V_{\geq\lambda}=0$, respectively.
Write $V_{>\lambda}:=\cup_{\mu>\lambda}V_{\geq\mu}$.
The tensor product $V\otimes W$ of $\RR$-filtered vector spaces is given the $\RR$-filtration where $(V\otimes W)_{\geq\lambda}$ is spanned by $V_{\geq \mu}\otimes W_{\geq \lambda-\mu}$, $\mu\in\RR$.

\begin{df}
\label{def_curved_ainfty}
A \textbf{(weakly) curved $A_\infty$-category}, $\mc C$, over $\mathbf k$ is given by 
\begin{enumerate}
\item a set $\Ob(\mc C)$ of objects,
\item for each pair $X,Y\in\Ob(\mc C)$ a $\ZZ$-graded vector space $\Hom(X,Y)$ over $\mathbf k$ with $\RR$-filtration $\Hom(X,Y)_{\geq\lambda}$ (complete and separated), and
\item for $X_0,\ldots,X_n\in\Ob(\mc C)$, $n\geq 0$, filtration-preserving \textit{structure maps}
\[
\mk m_n:\Hom(X_{n-1},X_n)\otimes\cdots\otimes\Hom(X_0,X_1)\to\Hom(X_0,X_n)
\]
of degree $2-n$,
\end{enumerate}
such that:
\begin{enumerate}
\item The \textit{$A_\infty$-relations with curvature}
\begin{equation} \label{A_infty_rels_m0}
\sum_{i+j+k=n}(-1)^{\Vert a_k\Vert+\ldots+\Vert a_1\Vert}\mk m_{i+1+k}(a_n,\ldots,a_{n-i+1},\mk m_j(a_{n-i},\ldots,a_{k+1}),a_k,\ldots,a_1)=0
\end{equation}
hold, where $i,j,k,n\geq 0$ and $\Vert a\Vert:=|a|-1$.
\item
\textit{Small curvature}: $\mk m_0(X)\in\Hom^2(X,X)_{>0}$ where, as shorthand, we write $\mk m_0$ or $\mk m_0(X)$ for the image of $1$ under $\mk m_0:\mathbf k\to\Hom^2(X,X)$.
\item \textit{Strict unitality}: For every $X\in\Ob(\mc C)$ there is  $1_X\in\Hom^0(X,X)$, such that
\begin{gather*}
\mk m_2(a,1_X)=(-1)^{|a|}a,\qquad \mk m_2(1_X,a)=a, \\
\mk m_k(\ldots,1_X,\ldots)=0\text{ for }k\neq 2.
\end{gather*}
\end{enumerate}
\end{df}

\begin{remark}
Instead of strict units one could, more generally, consider units up to homotopy as in~\cite{FOOO}. However, for our purposes the more restrictive notion will suffice. 
\end{remark}

The first two $A_\infty$-relations with curvature are 
\begin{gather*}
\mk m_1(\mk m_0)=0 \\
\mk m_1(\mk m_1(a))+\mk m_2(a,\mk m_0)+(-1)^{\|a\|}\mk m_2(\mk m_0,a)=0 
\end{gather*}
so $\mk m_1$ is not a differential in general.
The next two are
\begin{gather*}
\begin{split}
& \mk m_1(\mk m_2(a,b))+\mk m_2(a,\mk m_1(b))+(-1)^{\|b\|}\mk m_2(\mk m_1(a),b) \\
&+\mk m_3(a,b,\mk m_0)+(-1)^{\|b\|}\mk m_3(a,\mk m_0,b)+(-1)^{\|a\|+\|b\|}\mk m_3(\mk m_0,a,b)=0 
\end{split} \\
\mk m_2(a,\mk m_2(b,c))+(-1)^{\|c\|}\mk m_2(\mk m_2(a,b),c)+(\text{terms with }\mk m_3\text{ or }\mk m_4)=0
\end{gather*}
from which one sees that curved $A_\infty$-categories with $\mk m_k=0$ for $k\geq 3$ correspond to curved DG-categories (in the sense of~\cite{positselski}) via
\[
h=-\mk m_0,\qquad da=(-1)^{|a|}\mk m_1(a),\qquad ab=(-1)^{|b|}\mk m_2(a,b).
\]
Indeed, the first four $A_\infty$-relations correspond to $dh=0$, $d^2a=[h,a]$, the graded Leibniz rule, and associativity of the product.

The definition of an \textbf{(uncurved) $A_\infty$-category} differs from that of a curved one in that $\mk m_0=0$ and the $\Hom$'s are not equipped with filtrations. 
Giving each $\Hom$ the discrete filtration with $\Hom(X,Y)_{>0}=0$ turns such a category into a curved $A_\infty$-category in a canonical way.
In the other direction, there are two most basic ways in which to pass from a curved $A_\infty$-category to an uncurved one.

\begin{df}
Let $\mc C$ be a curved $A_\infty$-category. 
An object $X\in\Ob(\mc C)$ is \textit{flat} if $\mk m_0(X)=0$.
Denote by $V(\mc C)$ the full \textbf{subcategory of flat objects}.
\end{df}
\begin{df}
Denote by $\mc C_0$ the uncurved $A_\infty$-category with $\Ob(\mc C_0)=\Ob(\mc C)$ and
\[
\Hom_{\mc C_0}(X,Y):=\Hom_{\mc C}(X,Y)_{\geq 0}/\Hom_{\mc C}(X,Y)_{>0}
\]
and the induced structure maps. 
\end{df}

Since we have defined flat objects, let us also mention the following related definition.

\begin{df}
The \textbf{homotopy category}, $[\mc C]$, of a curved $A_\infty$-category $\mc C$ is the $\mathbf k$-linear category $H^0(V(\mc C))$, i.e. objects of $[\mc C]$ are flat objects of $\mc C$ and morphisms are
\[
\Hom_{[\mc C]}(X,Y):=H^0\left(\Hom_{\mc C}(X,Y),\mk m_1\right)
\]
where $\mk m_1^2=0$ since both $X$ and $Y$ are flat.
\end{df}

Two curved $A_\infty$-categories $\mc C$ and $\mc D$ over $\mathbf k$ are by said to be \textit{quasi-equivalent}, if the subcategories of flat objects, $V(\mc C)$ and $V(\mc D)$, are.
In particular, quasi-equivalence implies equivalence of the homotopy categories.

\begin{remark}
The relation between $\mc C$ and the uncurved category $V(\mc C)$ is analogous to the relation between a chain complex and its homology. 
Even though we may ultimately be interested in $V(\mc C)$, there are often intermediate constructions for which one needs to keep track of curved objects also.
As a basic example, $V(\mc C)_0$ is usually strictly smaller than $V(\mc C_0)=\mc C_0$.
\end{remark}

\begin{df}
An \textbf{curved $A_\infty$-functor} $F$ between curved $A_\infty$-categories $\mc A$ and $\mc B$  is given by a map $F:\Ob(\mc A)\to\Ob(\mc B)$ and a collection of filtration preserving maps
\[
F_n:\Hom_{\mc A}(X_{n-1},X_n)\otimes\cdots\otimes \Hom_{\mc A}(X_0,X_1)\to\Hom_{\mc B}(FX_0,FX_n)
\]
of degree $1-n$, $n\geq 0$, such that $F_0\in\Hom^1(FX,FX)_{>0}$ and the $A_\infty$-equations
\begin{equation}\label{ainfty_mor}
\begin{gathered}
\sum_{k\geq 0} \sum_{i_1+\ldots+i_k=n}\mk m_k(F_{i_k}(a_n,\ldots,a_{n-i_k+1}),\ldots,F_{i_1}(a_{i_1},\ldots,a_1)) \\
=\sum_{i+j+k=n}(-1)^{\Vert a_k\Vert+\ldots+\Vert a_1\Vert}F_{i+1+k}(a_n,\ldots,a_{n-i+1},\mk m_j(a_{n-i},\ldots,a_{k+1}),a_k,\ldots,a_1).
\end{gathered}
\end{equation}
hold.
We further require $A_\infty$-functors to be \textit{strictly unital}, i.e.
\[
F_1(1_X)=1_{FX},\qquad F_n(\ldots,1_X,\ldots)=0\qquad\text{for }n>1.
\]
\end{df}

Note that because of the $F_0$-term, there are infinitely many summands on the left-hand-side of the above equation. 
Convergence to a unique limit is guaranteed by the smallness requirement on $F_0$ together with completeness and separatedness of the filtrations.

Let $F:\mc A\to \mc B$ be a curved $A_\infty$-functor between curved $A_\infty$-categories.
The case $n=0$ of \eqref{ainfty_mor} is
\[
\mk m_0(F(X))+\sum_{k\geq 1} \mk m_k(F_0(X),\ldots,F_0(X))=F_1(\mk m_0(X))
\]
so, in particular, if $X$ is flat, then $F(X)$ need not be, which is why we refer to $F$ as \textit{curved}.
However, it turns out that a curved functor may be replaced by one without $F_0$-term at the expense of enlarging the target category.

\begin{df}
Let $\mc C$ be a curved $A_\infty$-category, then $\widetilde{\mc C}$ is the curved $A_\infty$-category whose objects are pairs $(X,\delta)$ with $\delta\in\Hom^1(X,X)_{>0}$, morphisms
\[
\Hom_{\widetilde{\mc C}}((X,\delta),(Y,\gamma)):=\Hom_{\mc C}(X,Y)
\]
(with the same filtration and grading), and structure maps $\widetilde{\mk m}_k$ obtained by ``inserting $\delta$'s everywhere'':
\begin{equation}\label{twisted_structure_maps}
\widetilde{\mk m}_k(a_k,\ldots,a_1):=\sum_{n_0,\ldots,n_k\ge 0}\mk m_{k+n_0+\ldots+n_k}(\underbrace{\delta_k,\ldots,\delta_k}_{n_k\text{ times}},a_k,\ldots,a_1,\underbrace{\delta_0,\ldots,\delta_0}_{n_0\text{ times}})
\end{equation}
This converges by our smallness assumption on $\delta_i$'s as well as completeness of $\Hom$-spaces.
\end{df}

According to the case $k=0$ of \eqref{twisted_structure_maps}, an object $(X,\delta)$ is flat if and only if
\begin{equation}\label{MC_eqn}
\sum_{k=0}^{\infty}\mk m_k(\underbrace{\delta,\ldots,\delta}_{k\text{ times}})=0
\end{equation}
which is known as the \textit{$A_\infty$-Maurer--Cartan equation}, and its solutions the \textit{Maurer--Cartan elements} or \textit{bounding cochains}.
Let $\mc{MC}(X)\subseteq\Hom^1(X,X)_{>0}$ denote the set of Maurer--Cartan elements.

A useful technical tool is the ability to ``transport'' Maurer--Cartan elements along equivalences, see~\cite{h_skein} for a proof.

\begin{lemma}
\label{lem_mctransport}
Let $\mc C$ be a curved $A_\infty$-category, $X,Y\in \mc C$, $\delta\in\mc{MC}(X)$, and $f\in\Hom(X,Y)_{\geq 0}/\Hom(X,Y)_{>0}$ and equivalence in $\mc C_0$.
Then there is a $\gamma\in\mc{MC}(Y)$ and $\tilde{f}\in\Hom(X,Y)_{\geq 0}$ lifting $f$ which is an equivalence from $(X,\delta)$ to $(Y,\gamma)$.
\end{lemma}

Given a curved $A_\infty$-category $\mc C$ one forms its additive closure $\mathrm{add}(\mc C)$, i.e. closure under finite direct sums and shifts, essentially as in the case of uncurved $A_\infty$-categories, see for example~\cite{seidel08}.
It is useful to combine this with the twisting construction $\mc C\mapsto \widetilde{\mc C}$ above.

\begin{df}
For a curved $A_\infty$-category $\mc C$ let $\mathrm{Tw}(\mc C):=\widetilde{\mathrm{add}(\mc C)}$.
\end{df}

There is another construction of \textit{one-sided twisted complexes} which is defined commonly applied to (uncurved) DG- and $A_\infty$-categories and for which the same notation $\mathrm{Tw}(\mc C)$ is often used. 
Instead of considering $\delta\in\Hom^1(X,X)_{>0}$ one requires $\delta$ to be strictly upper triangular with respect to some direct sum decomposition and that $\delta$ satisfies the Maurer--Cartan equation. 
The first condition ensures that in the formula for $\widetilde{\mk m_n}$ there are only finitely many non-zero terms.
Using this construction, one gets a (pre-)triangulated DG- or $A_\infty$-category. 
We refer to~\cite{seidel08} for further details.

\subsection{Tensor product of categories}
\label{subsec_tensor}

Let $\mc A$ be a DG-category over $\mathbf k$ and $\mc B$ a curved $A_\infty$-category over $\mathbf k$.
In this subsection we define the tensor product $\mc A\otimes \mc B$ as a curved $A_\infty$-category.
This is fairly straightforward except perhaps for the sign conventions.

Let $\mathrm{Ob}(\mc A\otimes\mc B):=\mathrm{Ob}(\mc A)\times \mathrm{Ob}(\mc B)$ and write $A\otimes B$ for the pair $(A,B)$.
Morphisms are given by
\[
\Hom_{\mc A\otimes\mc B}(A_1\otimes B_1,A_2\otimes B_2):=\Hom_{\mc A}(A_1,A_2)\otimes \Hom_{\mc B}(B_1,B_2)
\]
with filtration
\[
\Hom_{\mc A\otimes\mc B}(A_1\otimes B_1,A_2\otimes B_2)_{\geq\lambda}:=\Hom_{\mc A}(A_1,A_2)\otimes \Hom_{\mc B}(B_1,B_2)_{\geq\lambda}.
\]
The structure maps are
\begin{align*}
    \mk m_0(A\otimes B)&:=1_A\otimes \mk m_0(B) \\
    \mk m_1(a\otimes b)&:=(-1)^{|a|+\|b\|}da\otimes b+a\otimes \mk m_1(b) \\
    \mk m_n(a_n\otimes b_n,\ldots,a_1\otimes b_1)&:=(-1)^*a_n\cdots a_1\otimes \mk m_n(b_n,\ldots,b_1),\qquad n\geq 2
\end{align*}
where $*=\sum_{i<j}|a_i|\cdot\|b_j\|$.
The proof of the following proposition is straightforward and will be omitted.

\begin{prop}
Given a DG-category $\mc A$ and a curved $A_\infty$-category $\mc B$, $\mc A\otimes \mc B$ as defined above is a curved $A_\infty$-category.
\end{prop}

\subsection{Bar construction with curvature}
\label{subsec_bar}

The data of an augmented $A_\infty$-algebra $A=\mathbf k1\oplus\overline{A}$ over a field $\mathbf k$ is equivalently encoded in terms of the tensor DG-coalgebra generated by $\overline{A}[1]$ with the differential whose components are $\mk m_n$. 
This is known as the \textit{bar construction}.
Here we discuss a generalization where the augmentation is required to be only linear, not compatible with the $A_\infty$-algebra structure. The result is a \textit{curved} DG-coalgebra.
This is a generalization of parts of~\cite{positselski} to $A_\infty$-algebras. See also~\cite{hirsh_milles} for a version of this for operads and properads.

Instead of working over a fixed field, it will be useful to consider, more generally, $A_\infty$-algebras over the semisimple base ring $R:=\mathbf k^X$, where $X$ is a finite set. These are the same as $A_\infty$-categories over $\mathbf k$ with set of objects $X$.

\begin{df}
Let $R$ be a $\mathbf k$-algebra, then an \textbf{$A_\infty$-algebra over $R$} is an $A_\infty$-algebra $A$ over $\mathbf k$ such that $A$ has the structure of an $R$-bimodule and such that the structure maps induce maps of $R$-bimodules $A\otimes_R\cdots\otimes_R A\to A$.
\end{df}

If $A$ is an $A_\infty$-algebra over $R$, then define $R\to A$ by $r\mapsto r1_A=1_Ar$.
We will only consider the case $R=\mathbf k^X$, so in particular $R$ is semisimple.

\begin{df}
A \textbf{pre-augmentation} of an $A_\infty$-algebra $A$ over $R$ is a splitting $A=\ol{A}\oplus R$, where $R$ is identified with its image under the map $R\to A$, which is assumed to be injective.
An \textbf{augmentation} is a pre-augmentation such that $\mk m_n(\ol A,\ldots,\ol A)\subset \ol A$, equivalently so that the projection $A\to R$ is an $A_\infty$-morphism.
\end{df}

Our main example of interest is the following.
Let $n\geq 0$ and $\mc C_n$ be the $A_\infty$-category with objects $X_i$, $i\in \ZZ/(n+1)$, non-zero morphism spaces 
\[
\Hom(X_i,X_i)=\mathbf k1_{X_i},\qquad \Hom(X_{i-1},X_i)=\mathbf k\alpha_i
\]
with degrees of $\alpha_i$'s chosen such that 
\[
\sum_i|\alpha_i|=n-1.
\]
The only non-zero structure maps are $\mk m_2$, which is fixed by the strict unitality requirement, and
\[
\mk m_{n+1}(\alpha_{i},\alpha_{i+n},\ldots,\alpha_{i+1})=1_{X_i}.
\]
This $A_\infty$-category is independent, up to (derived) Morita equivalence, of the choice of $|\alpha_i|$ and is Morita equivalent to the path algebra of the $A_n$ quiver (with any orientation).
Viewing $\mc C_n$ as an $A_\infty$-algebra over $\mathbf k^{n+1}$, it has a unique pre-augmentation which is however not an augmentation.

Next, we turn to the discussion of the ``dual'' objects. 
The definition of a curved DG-coalgebra is obtained from that of a curved DG-algebra (as in~\cite{positselski}) by reversing all the arrows.
Let us spell this out explicitly.

\begin{df}
A \textbf{DG-coalgebra} is a $\ZZ$-graded $R$-module $A$ together with a comultiplication $\Delta:A\to A\otimes A$ of degree 0, a differential $d:A\to A$ of degree $1$, and a curvature functional $h:A\to R$ of degree 2.
These maps should satisfy associativity, the graded Leibniz rule, $h\circ d=0$, and 
\begin{equation}\label{cdga1}
d^2=(h\otimes \mathrm{id}-\mathrm{id}\otimes h)\circ\Delta.
\end{equation}
\end{df}

If $C$ is a curved DG-coalgebra, then $C^\vee$ (dual $\ZZ$-graded vector space) is a curved DG-algebra.
A \textbf{morphism of curved DG-coalgebras} $A\to B$ is given by a coalgebra homomorphism $f:A\to B$ and a functional $\alpha:A\to R$ such that
\begin{align}\label{cdga2}
f\circ d_A=d_B\circ f+(\alpha\otimes f-f\otimes\alpha)\circ\Delta \\ \label{cdga3}
h_A=h_B\circ f+\alpha\circ d_A-(\alpha\otimes\alpha)\circ\Delta
\end{align}
Composition of morphisms is defined by $(f,\alpha)\circ (g,\beta):=(f\circ g,\beta+\alpha\circ g)$.

We are now ready to spell out the definition of the curved DG-coalgebra $B\ol{A}$ of a pre-augmented strictly unital $A_\infty$-algebra $A=\overline{A}\oplus R$ over $R$.
Strict unitality implies that the operations $\mk m_n$ are determined by their restriction to $\overline{A}$.
Using the direct sum decomposition of $A$ we write
\[
\mk m_n(a_n,\ldots,a_1)=\overline{\mk m}_n(a_n,\ldots,a_1)+\mk h_n(a_n,\ldots,a_1)1_A
\]
where $\mk h_n(a_n,\ldots,a_1)\in R$.
Let
\[
B\overline{A}:=\bigoplus_{n\geq 1}(\overline{A}[1])^{\otimes_R n}
\]
be the (reduced) tensor coalgebra generated by $\overline{A}[1]$, then the operations $\overline{\mk m}_n$ induce a degree 1 codifferential, $\mk m$, on $B\overline{A}$. 
The multi-linear maps $\mk h_n:\overline{A}^{\otimes n}\to R$ combine to give a degree 2 coderivation $\mk h:B\overline{A}\to R$.

\begin{prop}
Given an pre-augmented strictly unital $A_\infty$-algebra $A$ over $R$, the tensor coalgebra $B\overline{A}$ becomes a curved DG-coalgebra with differential $\mk m$ and curvature $\mk h$ as defined above.
\end{prop}

\begin{proof}
The $A_\infty$-relations and strict unitality for $A$ imply
\begin{align*}
0&=\sum_{i+j+k=n}(-1)^{\Vert a_k\Vert+\ldots+\Vert a_1\Vert}\overline{\mk m}_{i+1+k}(a_n,\ldots,\mk m_j(a_{n-i},\ldots,a_{k+1}),a_k,\ldots,a_1) \\
&=\sum_{i+j+k=n}(-1)^{\Vert a_k\Vert+\ldots+\Vert a_1\Vert}\overline{\mk m}_{i+1+k}(a_n,\ldots,\overline{\mk m}_j(a_{n-i},\ldots,a_{k+1}),a_k,\ldots,a_1) \\
&\qquad+\overline{\mk m}_2(a_1,\mk h_{n-1}(a_2,\ldots,a_n)1_A)+(-1)^{\|a_n\|}\overline{\mk m}_2(\mk h_{n-1}(a_1,\ldots,a_{n-1})1_A,a_n)
\end{align*}
for $a_1,\ldots,a_n\in \overline{A}$.
Rearranging and using unitality again we get
\begin{equation}\label{reduced_m_rel}
\begin{aligned}
&\sum_{i+j+k=n}(-1)^{\Vert a_k\Vert+\ldots+\Vert a_1\Vert}\overline{\mk m}_{i+1+k}(a_n,\ldots,\overline{\mk m}_j(a_{n-i},\ldots,a_{k+1}),a_k,\ldots,a_1)=\\
&\qquad =a_n\mk h_{n-1}(a_{n-1},\ldots,a_1)-\mk h_{n-1}(a_n,\ldots,a_2)a_1
\end{aligned}
\end{equation}
which is equivalent to \eqref{cdga1}.
Taking the $R1_A$ component instead of the $\overline{A}$ component of the $A_\infty$-relations in $A$ we get
\begin{equation}\label{lambda_rel}
0=\sum_{i+j+k=n}(-1)^{\Vert a_k\Vert+\ldots+\Vert a_1\Vert}\mk h_{i+1+k}(a_n,\ldots,\overline{\mk m}_j(a_{n-i},\ldots,a_{k+1}),a_k,\ldots,a_1).
\end{equation}
for $a_1,\ldots,a_n\in \overline{A}$.
This is equivalent to $\mk h\circ \mk m=0$.
\end{proof}

Having constructed a curved DG-coalgebra $B\overline{A}$ we can take the $R$-dual $(B\overline{A})^\vee$ to get a curved DG-algebra. 
It is non-unital, which is however easily fixed by passing to the augmented algebra $(B\overline{A})^{\vee+}=(B\overline{A})^\vee\oplus R$.
Moreover, $(B\overline{A})^{\vee+}$ is $\ZZ_{\geq 0}$-filtered with $\left((B\overline{A})^{\vee+}\right)_{\geq j}$ by definition those functionals vanishing on tensors of length $<j$, where the added copy of $R$ is considered as tensors of length 0. 

Let us return to the example $A=\mc C_n$ with its unique pre-augmentation.
Taking the dual of $B\overline{A}$, the curved DG-algebra $(B\overline{A})^{\vee+}$ has underlying graded algebra the path algebra of the quiver arrows $\beta_i$, $i\in\ZZ/(n+1)$, forming a single cycle $\beta_n\cdots\beta_1\beta_0$ of length $n+1$ and grading such that $|\beta_i|=1-|\alpha_i|$ thus $|\beta_n|+\ldots+|\beta_0|=2$. 
The curvature is the central element
\[
h=\sum_{i=0}^n \beta_{n+i}\beta_{n+i-1}\cdots\beta_i.
\]

The construction above is functorial with respect to $A_\infty$-morphisms.

\begin{prop}
A strictly unital $A_\infty$-morphism $f:C\to D$ of strictly unital, augmented $A_\infty$-algebras induces a morphism of curved DG-coalgebras $B\overline{C}\to B\overline{D}$.
\end{prop}

\begin{proof}
By assumption, we have splittings $C=\overline{C}\oplus R$ and $D=\overline{D}\oplus R$. 
Using the second one, we write $f_n=:\overline{f}_n+\phi_n$.
The sequence of maps $\overline{f}_n$ define a morphism of coalgebras $f:B\overline{C}\to B\overline{D}$ and the sequence of maps $-\phi_n$ define a degree 1 functional $\phi:B\overline{C}\to R$.
We need to check that these satisfy \eqref{cdga2} and \eqref{cdga3}.

Taking the $\overline{D}$-component of the $A_\infty$-morphism equations~\ref{ainfty_mor} yields
\begin{equation*}
\begin{gathered}
\sum_k \sum_{i_1+\ldots+i_k=n}\overline{\mk m}_k(\overline{f}_{i_k}(a_n,\ldots,a_{n-i_k+1}),\ldots,\overline{f}_{i_1}(a_{i_1},\ldots,a_1)) \\
+\sum_k\left( \overline{f}(a_n,\ldots,a_{k+1})\phi(a_k,\ldots,a_1)-(-1)^{\|a_1\|+\ldots+\|a_k\|}\phi(a_n,\ldots,a_{k+1})\overline{f}(a_k,\ldots,a_1)\right) \\
=\sum_{i+j+k=n}(-1)^{\Vert a_1\Vert+\ldots+\Vert a_k\Vert}\overline{f}_{i+1+k}(a_n,\ldots,a_{n-i+1},\overline{\mk m}_j(a_{n-i},\ldots,a_{k+1}),a_k,\ldots,a_1).
\end{gathered}
\end{equation*}
which is \eqref{cdga2}, while taking the $R$-component yields
\begin{equation*}
\begin{gathered}
\sum_k \sum_{i_1+\ldots+i_k=n}\lambda_k(\overline{f}_{i_k}(a_n,\ldots,a_{n-i_k+1}),\ldots,\overline{f}_{i_1}(a_{i_1},\ldots,a_1)) \\
-\sum_k(-1)^{\|a_1\|+\ldots+\|a_k\|}\phi(a_n,\ldots,a_{k+1})\phi(a_k,\ldots,a_1) \\
=\sum_{i+j+k=n}(-1)^{\Vert a_1\Vert+\ldots+\Vert a_k\Vert}\phi_{i+1+k}(a_n,\ldots,a_{n-i+1},\overline{\mk m}_j(a_{n-i},\ldots,a_{k+1}),a_k,\ldots,a_1) \\
+\lambda_n(a_n,\ldots,a_1)
\end{gathered}
\end{equation*}
which is \eqref{cdga3}.
\end{proof}

\begin{prop}
\label{prop_funmc}
Let $A$ be a pre-augmented $A_\infty$-algebra over $R=\mathbf k^n$ and $\mc A$ the corresponding $A_\infty$-category with $n$ objects, and let $\mc E$ be a DG-category over $\mathbf k$ with finite direct sums.
Assume that $\Hom_{\mc E}(X,Y)$ is finite-dimensional for each $X,Y\in\mc E$.
Then the following DG-categories are quasi-equivalent.
\begin{enumerate}
\item 
The category of $A_\infty$-functors $\mc A\to \mc E$.
\item
The category of twisted complexes $\mathrm{Tw}\left(\mc E\otimes (B\overline{A})^{\vee+}\right)$, where $(B\overline A)^{\vee+}$ is viewed as a curved DG-category with $n$ objects.
\end{enumerate}
\end{prop}

\begin{proof}
An $A_\infty$-functor $F:\mc A\to\mc E$ corresponds to an $A_\infty$-morphism $f:A\to E$ of $A_\infty$-algebras over $R$ where
\[
E:=\mathrm{End}_{\mc E}\left(\bigoplus_{i=1}^nF(X_i)\right)
\]
is a DG-algebra and $X_i$ is the $i$-th object in $\mc A$.
The components $f_k$ of $f$ need to satisfy the quadratic relations
\begin{gather*}
    \mk m_1\left(f_k(a_k,\ldots,a_1)\right)+\sum_{i=1}^{k-1} \mk m_2\left(f_{k-i}(a_k,\ldots,a_{i+1}),f_i(a_i,\ldots,a_1)\right) \\
    =\sum_{i=1}^k\sum_{j=0}^{k-i}(-1)^{\|a_1\|+\ldots+\|a_j\|} f_{k-i+1}\left(a_k,\ldots,\overline{\mk m}_i(a_{i+j},\ldots,a_{j+1}),\ldots,a_1\right)+\mk h_k(a_k,\ldots,a_1)1_E.
\end{gather*}
By strict unitality, $f$ is completely determined by an $R\otimes R$-linear map $f:B\overline{A}\to E$ and the above equations for $f$ then translate into the Maurer--Cartan type equation
\begin{equation}\label{twisting_cochain}
\mk h1_E+\mk m\circ f-f\circ \mk m+\mu\circ(f\otimes f)\circ\Delta=0
\end{equation}
for \textit{twisting cochains} from a curved DG-coalgebra with comultiplication $\Delta$ to a unital DG-algebra with multiplication $\mu$.
Finally, since $\dim E<\infty$  by assumption, there is a canonical isomorphism
\[
\Hom_{R\otimes R}(B\overline{A},E)\cong E\otimes_{R\otimes R}(B\overline{A})^\vee\cong\mathrm{End}\left(\bigoplus_{i=1}^n F(X_i)\otimes X_i\right)_{>0}
\]
under which \eqref{twisting_cochain} becomes the Maurer--Cartain equation in $\mc E\otimes (B\overline{A})^{\vee+}$.
Since $\mc E$ has finite direct sums, any twisted complex is of the above form for some functor $F$.

To define the correspondence on morphisms, suppose $F,G:\mc A\to \mc E$ is a pair of $A_\infty$-functors.
Then we consider, instead of $E$ above,
\[
H:=\mathrm{Hom}_{\mc E}\left(\bigoplus_{i=1}^nF(X_i),\bigoplus_{i=1}^nG(X_i)\right)
\]
so that an element in the complex of natural transformations from $F$ to $G$ is given by a linear map $BA^+\to H$.
On the other hand,
\[
\Hom_{R\otimes R}(B\overline{A}^+,H)\cong H\otimes_{R\otimes R}(B\overline{A})^{\vee+}\cong\mathrm{Hom}\left(\bigoplus_{i=1}^n F(X_i)\otimes X_i,\bigoplus_{i=1}^n G(X_i)\otimes X_i\right)
\]
which has $B\overline{A}^+$ instead of $BA^+$.
However, the two complexes are quasi-isomorphic by the standard argument for normalized bar complexes.
We omit the calculation to show that this is compatible with differentials and composition.
\end{proof}

\subsection{Fibrations of DG-categories}
\label{subsec_fibdg}

We consider a class of $A_\infty$-functors which enjoy good lifting properties.

\begin{df}
An $A_\infty$-functor $F:\mc C\to \mc D$ between (uncurved) $A_\infty$-categories over $\mathbf k$ is a \textbf{fibration} if
\begin{enumerate}
    \item the maps $\Hom_{\mc C}^\bullet(X,Y)\to\Hom_{\mc D}^\bullet(X,Y)$ are surjective, and
    \item if $f:F(X)\to Y$ is an isomorphism in the homotopy category $H^0\mc D$ then there is an isomorphism $\widetilde{f}:X\to \widetilde{Y}$ in $H^0\mc C$ with $F(\widetilde{f})=f$. 
\end{enumerate}
\end{df}

According to \cite{tabuada} there is a model structure on the category of (small) DG-categories over a commutative ring $\mathbf k$ where weak equivalences are quasi-equivalences of DG-categories and fibrations are DG-functors which are fibrations in the sense of the above definition.  
In the general case of $A_\infty$-functors between $A_\infty$-categories, these fibrations are not part of a model structure however.
A source of fibrations in this sense is obtained from filtered $A_\infty$-categories, i.e. categories as in Definition~\ref{def_curved_ainfty} but with $\mk m_0=0$. 

Let $\mc C$ be a filtered $A_\infty$-category, then the (unfiltered) $A_\infty$-category $\mc C_0$ has  the same objects as $\mc C$, morphisms 
\[
\Hom_{\mc C_0}(X,Y):=\Hom_{\mc C}(X,Y)/\Hom_{\mc C}(X,Y)_{>0}
\] 
and the induced structure maps.
There is a tautological strict functor $\mc C\to \mc C_0$ of $A_\infty$-categories.

\begin{lemma}
Let $\mc C$ be a filtered $A_\infty$-category, then the functor $F:\mc C\to \mc C_0$ is a fibration.
\end{lemma}

\begin{proof}
It is clear that $F$ is surjective on $\Hom$.
Furthermore, $F$ reflects isomorphisms, that is if $f\in\Hom_{\mc C}(X,Y)$ such that $F(f)\in\Hom_{\mc C_0}(X,Y)$ has an inverse up to homotopy, then $f$ has an inverse up to homotopy, see e.g.~\cite{h_skein}.
In particular, since $F$ is an isomorphism on objects, this implies that $F$ has the isomorphism lifting property.
\end{proof}

We want to apply this to the following situation: $\mc B$ is an $A_\infty$-category, $\mc A\subset \mc B$ is a subcategory in the sense that $\Ob(\mc A)=\Ob(\mc B)$, $\Hom_{\mc A}(X,Y)\subset \Hom_{\mc B}(X,Y)$, and the structure maps of $\mc A$ are obtained by restricting the ones from $\mc B$.
If $\mc C$ is a DG-category, then there is a DG-category of $A_\infty$-functors, $\mathrm{Fun}(\mc B,\mc C)$.
A morphism in this category is given by a natural transformation between functors $F$ and $G$ which has components
\[
\lambda_d:\Hom_{\mc B}(X_{d-1},X_d)\otimes\cdots\otimes\Hom_{\mc B}(X_0,X_1)\to \Hom_{\mc C}(FX_0,GX_d)
\]
for $d\geq 0$ and any sequence of objects $X_0,\ldots,X_d$ in $\mc B$.
The subcategory $\mc A$ can be used to define a filtration on this space $\Hom(F,G)$ of natural transformations by declaring $\Hom(F,G)_{\geq k}$, $k\geq 1$, to consist of those sequences of multilinear maps $\lambda_d$ which vanish if all but possibly $k-1$ inputs are in $\mc A$.
By construction, the category $\mathrm{Fun}(\mc B,\mc C)_0$ is then equivalent to $\mathrm{Fun}(\mc A,\mc C)$. 
Applying the previous lemma, we obtain:

\begin{coro}
\label{cor_funfib}
With $\mc A\to\mc B$ an inclusion of $A_\infty$-categories as above and $\mc C$ a DG-category, the induced functor of DG-categories
\[
\mathrm{Fun}(\mc B,\mc C)\to\mathrm{Fun}(\mc A,\mc C)
\]
is a fibration.
\end{coro}

\subsection{Categories over the Novikov ring}
\label{subsec_catnov}

Fix a field $\mathbf k$, playing the role of \textit{residue field}.
The \textbf{Novikov ring}, $\mathbf R$, is the ring of formal power series of the form
\[
\sum_{\lambda\in[0,\infty)}a_\lambda t^{\lambda}
\]
where $a_\lambda\in\mathbf{k}$ and $a_\lambda=0$ except for $\lambda$ in some discrete subset of $[0,\infty)$.
This is a non-Noetherian local ring with field of fractions, $\mathbf K$, consisting of formal series of the form
\[
\sum_{\lambda\in\RR}a_\lambda t^{\lambda}
\]
where $a_\lambda=0$ except for $\lambda$ in some discrete, bounded below subset of $\RR$.
If $\mathbf k$ is algebraically closed, then so is $\mathbf K$.

Linear maps $\mathbf R^m\to \mathbf R^n$ have a kind of singular value decomposition.
\begin{prop}
\label{prop_RmapSVD}
Let $f:V\to W$ be an $\mathbf R$-linear map between finite rank free modules, then there are bases $b_1,\ldots,b_m$ of $V$, $c_1,\ldots,c_n$ of $W$, and $0\leq \lambda_1\leq\ldots\leq \lambda_r$, $r\leq \min(m,n)$ such that $f(b_i)=t^{\lambda_i}c_i$ for $1\leq i\leq r$, and $f(b_i)=0$, $r<i\leq n$. 
\end{prop}

The proof is a basic row and column reduction argument.
As a corollary, any finitely presented $\mathbf R$-module is a finite direct sum of modules of the form $\mathbf R$ or $\mathbf R/t^\lambda \mathbf R$, $t>0$.

Next, we turn to $A_\infty$-categories over $\mathbf R$.

\begin{df}
A \textbf{curved $A_\infty$-category over $\mathbf R$} is a curved $A_\infty$-category, $\mc C$, over $\mathbf k$ such that each $\Hom_{\mc C}(X,Y)$ has the structure of a free $\mathbf R$-module, the structure maps $\mk m_n$ are $\mathbf R$-linear, and the filtrations satisfy
\begin{equation}
t^\lambda\Hom(X,Y)_{\geq \alpha}\subseteq \Hom(X,Y)_{\geq \alpha+\lambda}
\end{equation}
for $\alpha\in\RR$ and $\lambda\geq 0$.
\end{df}

Geometrically, we think of an $A_\infty$-category $\mc C$ over $\mathbf R$ as a family of categories over the formal ``disk'' $\mathrm{Spec}(\mathbf R)$.
The \textbf{central fiber} of this family is the curved $A_\infty$-category $\mc C_{\mathbf k}:=\mc C\otimes_{\mathbf R}\mathbf k$ over $\mathbf k$.
This category has by definition the same objects and 
\[
\Hom_{\mc C_{\mathbf k}}(X,Y):=\Hom_{\mc C}(X,Y)\otimes_{\mathbf R} \mathbf k=\Hom_{\mc C}(X,Y)/t^{>0}\Hom_{\mc C}(X,Y)
\]
with the induced filtrations and structure maps.
The \textbf{general fiber} is the curved $A_\infty$-category $\mc C_{\mathbf K}:=\mc C\otimes_{\mathbf R}\mathbf K$ over $\mathbf K$.
There are natural $A_\infty$-functors fitting into a diagram
\[
\begin{tikzcd}
 & \mc C \arrow[dl]\arrow[dr] & \\
\mc C_{\mathbf k} & & \mc C_{\mathbf K} 
\end{tikzcd}
\]
The interplay between these three categories is of fundamental importance in this paper.

Note that the category $\mc C_0$, defined for any curved $A_\infty$-category $\mc C$, is in general different from $\mc C_{\mathbf k}$, defined for a curved $A_\infty$-category over the Novikov ring. 
On the other hand, the relation between $\mc C_0$ and $\mc C$ is in many ways similar to the relation between $\mc C_{\mathbf k}$ and $\mc C$.
As an example, we have the following variant of Lemma~\ref{lem_mctransport}.

\begin{lemma}
\label{lem_mctransport_nov}
Let $\mc C$ be a curved $A_\infty$-category over $\mathbf R$, $(X,\delta)\in V(\mathrm{Tw}(\mc C))$ and $f_0:(X,\delta_0)\to (Y,\gamma_0)$ an equivalence in $V(\mathrm{Tw}(\mc C)_{\mathbf k})$.
Then there is a Maurer--Cartan element $\gamma$ for $Y$ with constant term $\gamma_0$ and a lift of $f_0$ to an equivalence $f:(X,\delta)\to (Y,\gamma)$ in $V(\mathrm{Tw}(\mc C))$.
\end{lemma}

\begin{proof}
This follows from Lemma~\ref{lem_mctransport} after changing the filtrations on $V(\mathrm{Tw}(\mc C))$ so that 
\[
\Hom(X,Y)_{\geq \lambda}:=t^\lambda \Hom(X,Y)
\]
which is still preserved by the structure maps thanks to $\mathbf R$-linearity.
\end{proof}

\begin{prop}
\label{prop_quasiequiv}
Let $F:\mc C\to\mc D$ be an $A_\infty$-functor between curved $A_\infty$-categories such that $\mathrm{Tw}(F)_{\mathbf k}: \mathrm{Tw}(\mc C)_{\mathbf k}\to \mathrm{Tw}(\mc D)_{\mathbf k}$ is a quasi-equivalence of curved $A_\infty$-categories over $\mathbf k$.
Assume that each $\Hom^n_{\mc C}(X,Y)$ and $\Hom^n_{\mc D}(X,Y)$ is a finite-rank free module.
Then $\mathrm{Tw}(F):\mathrm{Tw}(\mc C)\to \mathrm{Tw}(\mc D)$ is a quasi-equivalence.
\end{prop}

\begin{proof}
Using Proposition~\ref{prop_RmapSVD} one sees that if $(C,d)$ is a chain complex of finite-rank free $\mathbf R$-modules such that $C\otimes_{\mathbf R}\mathbf k$ is acyclic, then $C$ is itself acyclic.
Thus, if $C,D$ are chain complexes satisfying the same condition and $f:C\to D$ is a chain map inducing a quasi-isomorphism after base-change to $\mathbf k$, then $f$ was already a quasi-isomorphism.
This shows that $\mathrm{Tw}(F)$ is quasi fully-faithful.

It remains to check essential surjectivity of $\mathrm{Tw}(F)$ on the level of homotopy categories.
Let $(Y,\gamma)\in V(\mathrm{Tw}(\mc D))$, then by assumption there is an $(X,\delta)\in\mathrm{Tw}(\mc C)$ and an equivalence $F(X)\to (Y,\gamma)$ in $\mathrm{Tw}(\mc C)_{\mathbf k}$. 
Transport of $\gamma$ along this equivalence (Lemma~\ref{lem_mctransport_nov}) gives a Maurer--Cartan element $\delta'\in\mc{MC}(F(X))$, differing from $\delta$ in higher order terms, such that $(F(X),\delta')$ is equivalent to $(Y,\gamma)$ in $\mathrm{Tw}(\mc D)$.
Finally, since $F$ induces a quasi-equivalence of curved $A_\infty$-algebras $\Hom(X,X)\to\Hom(F(X),F(X))$ there is a $\tilde{\delta}\in\mc{MC}(X)$ such that $F(\tilde{\delta})$ is gauge equivalent to $\delta'$, in particular $F(X,\tilde{\delta})=(F(X),F(\tilde{\delta}))$ is equivalent to $(F(X),\delta')$.
\end{proof}

%%%%%%%%%%%%%%%%%%%%%%%%%%%%%%%%%%%%%%%%%%%%%%%%%%%%%%%%

\section{Fukaya categories of surfaces with coefficients}
\label{sec_aside}

Given a symplectic manifold $M$ with some additional structure one considers its Fukaya category $\mc F(M)_{\mathbf K}$ which is an $A_\infty$-category over the Novikov field $\mathbf K$.
Choosing a triangulated DG-category $\mc E$ over $\mathbf k$, there is the \textit{Fukaya category with coefficients in $\mc E$} which can be defined simply as $\mc F(M;\mc E)_{\mathbf K}:=\mc E\otimes_{\mathbf k}\mc F(M)_{\mathbf K}$ (triangulated completion of the tensor product) by a universal coefficient principle.
The purpose of this section is to define a certain \textit{model} of this category in the case $\dim_{\RR}M=2$, i.e. a category $\mc F(M;\mc E)$ over the Novikov ring $\mathbf R$ which gives $\mc F(M;\mc E)_{\mathbf K}$ upon base-change to $\mathbf K$.
Objects in the central fiber $\mc F(M;\mc E)_{\mathbf k}=\mc F(M;\mc E)\otimes_{\mathbf R}\mathbf k$ correspond, roughly, to graphs in $M$ with edges labeled by objects of $\mc E$.

\subsection{Grading conventions}
\label{subsec_grading}

The definition of $\ZZ$-graded Floer homologies and Fukaya categories, as opposed to $\ZZ/2$-graded ones, requires Lagrangian submanifolds to be equipped with certain additional grading structure~\cite{kontsevich_hms,seidel_graded}.
We recall the relevant definitions in the special case of surfaces.
This subsection also serves to fix our conventions.

Suppose $S$ is a Riemann surface and consider its holomorphic cotangent bundle $T^*S$ as a complex line bundle.
A \textbf{grading structure} on $S$ is a smooth, nowhere vanishing section $\nu$ of $(T^*S)^{\otimes 2}$, i.e. a (smooth) quadratic differential on $S$.
A grading of a curve in $S$ is essentially a continuous choice of $\mathrm{Arg}(\nu)$ along the curve. 
More precisely, a \textbf{grading} of an immersed curve $c:I\to S$ is a smooth function $\alpha:I\to\RR$ such that if $x\in I$ and $v\in T_{c(x)}S\setminus\{0\}$ is tangent to $c$, then 
\[
\nu(v,v)=re^{2\pi i \alpha(x)}
\]
for some $r>0$. 
Given $n\in\ZZ$, the \textbf{shift} of the grading $\alpha$ by $n$ is $\alpha+n$.

\begin{remark}
For \textit{cochain} complexes (differential has degree $+1$), the functor $[1]$ shifts the entire complex \textit{down} in degree.
Thus the $\alpha$ above is in a sense opposite to cohomological degree.
Still, we choose this conventions because 1) the phase of a special Lagrangian submanifold is usually defined as a choice of argument of the holomorphic volume form (not $-\mathrm{Arg}$) and 2) in Bridgeland's definition of a stability condition, phase is opposite to cohomological degree.
\end{remark}

\begin{remark}
The choice of complex structure on $S$ is auxiliary and can be avoided entirely, see for instance~\cite{hkk}.
\end{remark}

Suppose $L_0=(I_0,c_0,\alpha_0)$ and $L_1=(I_1,c_1,\alpha_1)$ are graded curves with transverse intersection at $p=c_0(t_0)=c_1(t_1)$. 
Define the \textbf{intersection index}
\begin{equation}
\begin{aligned}
i(L_0,t_0,L_1,t_1) & :=\lceil \alpha_0(t_0)-\alpha_1(t_1)\rceil\in\ZZ.
\end{aligned}
\end{equation}
If $t_0$ and $t_1$ are uniquely determined by $p\in S$, then we may write $i_p(L_0,L_1)$ for $i(L_0,t_0,L_1,t_1)$.

\subsection{Category with support on a graph}
\label{subsec_graph}

Fix throughout this subsection a compact Riemann surface $S$ with boundary, equipped with a grading structure $\nu$ and an area 2-form $\omega$ compatible with the orientation coming from the complex structure.
Furthermore, fix an \textbf{embedded graph} $G\subset S$ given by a finite set $G_0$ of vertices and a finite set $G_1$ of embedded intervals ending at the vertices. 
The edges meeting at a given vertex are required have distinct tangent rays.
We allow 1-valent and 2-valent vertices, multiple edges between vertices, and vertices to lie on the boundary of $S$.
Our goal is to construct a curved $A_\infty$-category $\mc F_G(S;\mc E)$ over the Novikov ring, which can be thought of as a Fukaya category of $S$ with objects supported on $G$ and coefficients in a triangulated DG-category $\mc E$ over $\mathbf k$.

It will be convenient to pass to the \textit{real blow-up} of $S$ in $G_0$, denoted $\pi:\widehat{S}\to S$ which replaces each vertex of $G$ in the interior of $S$ with a boundary circle and each vertex of $G$ on $\partial S$ by an interval with each endpoint a corner of $\widehat{S}$.
These together form the \textit{exceptional boundary} $\pi^{-1}(G_0)\subset\partial \widehat{S}$.
The edges of $G$ lift to a disjoint union of embedded intervals (arcs), $\widehat{G}$, with endpoints in the exceptional boundary, see Figure~\ref{fig_blowup}.

\begin{figure}
\centering
\includegraphics[scale=.9]{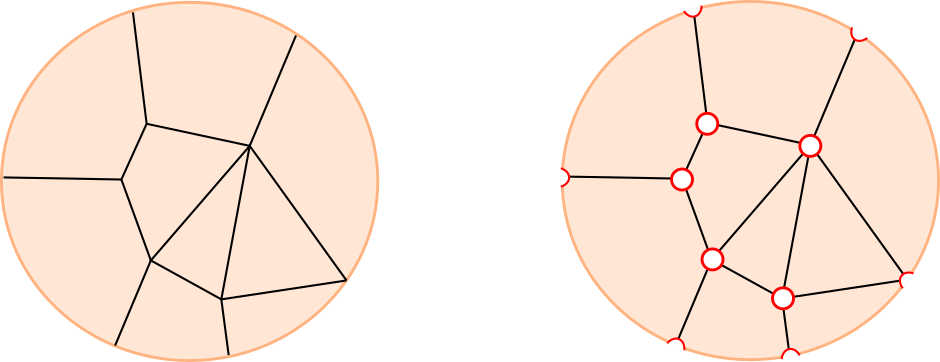}
\caption{Real blow-up $\widehat{S}$ (shown right) of a surface $S$ in the vertices of a graph $G$ (shown left). The exceptional boundary is depicted in red.}
\label{fig_blowup}.
\end{figure}

As a first application of this construction we make the following definition. 

\begin{df}
Let $X,Y\in \widehat{G}$ be arcs. 
A \textbf{boundary path} from $X$ to $Y$ is an immersed path, up to reparameterization, starting at an endpoint of $X$, ending at an endpoint of $Y$, and following the exceptional boundary of $\widehat{S}$ in the direction opposite of its natural orientation, i.e. so that the surface is to the right when following the path.
\end{df}

\begin{remark}
This convention for the orientation of boundary paths is the same as in~\cite{hkk}, but opposite to the more common convention in the (wrapped) Fukaya category literature, which is however incompatible with the combination of 1) non-zero degree zero morphisms between semistable objects in \textit{increasing} phase 2) the phase of special Lagrangians is a choice of $\mathrm{Arg}(\Omega|_L)$.
\end{remark}

If a boundary path, $a$, ends at the same endpoint of an arc $Y$ at which another boundary path, $b$, starts, then their \textit{concatenation} is also boundary path and denoted by $ba$. (Note the order!)
Furthermore, a boundary path from $X$ to $Y$ can be assigned an integer degree once gradings have been chosen for $X$ and $Y$.
To see this, assume that the complex structure and $\nu$ have been extended to $\partial\widehat{S}$, which is possible after perturbing them slightly near the boundary. 
The degree, $|a|\in\ZZ$, of a boundary path $a:[0,1]\to\widehat{S}$ from $p=a(0)\in X$ to $q=a(1)\in Y$ is then 
\[
|a|:=i(X,p,a,0)-i(Y,q,a,1)
\]
for arbitrary choice of grading on $a$.

For the purpose of defining suitable $\RR$-filtrations on $\Hom$ it will be necessary to assign weights in $\RR_{\geq 0}$ to all boundary paths such that the following conditions are satisfied.

\begin{df}
A \textbf{compatible choice of weights} is a choice of $w(a)\in\RR_{>0}$ for each boundary path such that:
\begin{enumerate}[(1)]
\item 
$w(ab)=w(a)+w(b)$ if $a,b$ are composable boundary paths.
\item 
Whenever $a_1,a_2,\ldots,a_n$ is a sequence of boundary paths, $a_i:X_{i-1}\to X_i$, $X_n=X_0$, so that the edges $X_1,X_2,\ldots,X_n$ form the boundary of an $n$-gon cut out by $G$ in $S$ with area $A>0$, then $w(a_1)+\ldots+w(a_n)\leq A$.
\end{enumerate}
\end{df}

By the first condition, a compatible choice of weights is determined by its value on \textit{primitive} boundary paths, i.e. those which do not meet any arc except at the endpoints.
Each such boundary path belongs to at most one polygon cut out by $G$, so it is easy to ensure the second condition, which only involves the primitive boundary paths.
Thus a compatible choice of weights always exists.

After these preliminaries on boundary paths, the first stage in the construction of the category $\mc F_G(S;\mc E)$ is the definition of a curved $A_\infty$-category $\mc A_G(S)_{\mathbf k}$ over $\mathbf k$ which does not depend on the coefficient category $\mc E$ or the area form $\omega$.
We fix a compatible choice of weights and choice of grading on each edge of $G$.
\begin{itemize}
\item
\textbf{Objects} of $\mc A_G(S)_{\mathbf k}$ are edges of $G$.
\item
\textbf{Morphisms:} Let $X,Y$ be objects of $\mc A_G(S)_{\mathbf k}$ and $\widehat{X},\widehat{Y}$ the corresponding arcs in $\widehat{S}$. 
A basis of $\Hom(X,Y)$ over $\mathbf k$ is given by boundary paths from $\widehat X$ to $\widehat Y$ and, if $X=Y$ the identity morphism $1_X$.
\item
Define the \textbf{filtrations} on $\Hom(X,Y)$ by
\[
a\in\Hom(X,Y)_{\geq\lambda}\Leftrightarrow w(a)\geq \lambda
\]
where $a$ is a boundary path from $X$ to $Y$.
\item
The \textbf{structure map} $\mk m_2$ is essentially concatenation of paths. 
More precisely, let $X,Y,Z$ be objects of $\mc A_G(S)_{\mathbf k}$, $a$ a boundary path from $X$ to $Y$, and $b$ a boundary path from $Y$ to $Z$. 
Define
\[
\mk m_2(b,a):=\begin{cases} (-1)^{|a|}ba & \text{ if } a,b \text{ composable } \\ 0 & \text{ otherwise}\end{cases}
\]
and $\mk m_2(a,1_X):=a$, $\mk m_2(1_X,a)=(-1)^{|a|}a$.
The structure maps $\mk m_1$ and $\mk m_{\geq 3}$ are zero.
\item
The \textbf{curvature} $\mk m_0(X)\in\Hom^2(X,X)$ of an object $X$ is defined as follows.
If $p$ is an endpoint of $X$ in the interior of $S$, then let $c_{p,X}\in\Hom^2(X,X)$ be the boundary path which starts and ends at $X$ and winds once around the component of $\widehat{S}$ corresponding to $p$.
Otherwise, if $p\in\partial S$, then let $c_{p,X}=0$.
Define
\[
\mk m_0(X):=c_{p,X}+c_{q,X}
\]
where $p,q$ are the two endpoints of $X$.
\end{itemize}

The definition of $\mk m_0$ above shows why it was necessary to assign positive weights to boundary paths, since this ensures that $\mk m_0\in \Hom^2(X,X)_{>0}$.

\begin{prop}
$\mc A_G(S)_{\mathbf k}$ defined above is a curved $A_\infty$-category over $\mathbf k$.
\end{prop}

\begin{proof}
As only $\mk m_0$ and $\mk m_2$ are non-zero, the claim is simply that the product $a\cdot b=(-1)^{|b|}\mk m_2(a,b)$ is associative, and that $\mk m_0$ is central in the sense that $\mk m_0(X)$ are the components of a natural transformation from the identity functor to itself.
Both follow immediately from the definition in terms of boundary paths.
Also compatibility with grading and filtrations is clear. (We do not need the second condition on the weights for now, it will be used later.)
\end{proof}

The next step is to define a certain deformation of $\mc A_G(S)_{\mathbf k}$ over $\mathbf R$, i.e. a category $\mc A_G(S)$ over $\mathbf R$ whose base change to $\mathbf k$ is $\mc A_G(S)_{\mathbf k}$.
The deformation depends on the areas (with respect to $\omega$) of regions cut out by $G\subset S$.
We let $\Ob{\mc A_G(S)}:=\Ob{\mc A_G(S)_{\mathbf k}}$ and 
\[
\Hom_{\mc A_G(S)}(X,Y):=\Hom_{\mc A_G(S)_{\mathbf k}}(X,Y)\otimes_{\mathbf{k}}\mathbf R
\]
with filtration so that $t^{\mu}\Hom(X,Y)_{\geq\lambda}=\Hom(X,Y)_{\geq\lambda+\mu}$.
Higher order ($t^{>0}$) corrections to the structure maps of $\mc A_G(S)$ come from immersed disks as described below.

\begin{df}
A \textbf{marked disk} is a pair $(D,M)$ where $D$ is diffeomorphic to a closed disk and $M\subset D$ is a finite subset of marked points with at least one point on $\partial D$.
Given a marked disk $(D,M)$ we consider its real blow-up, $\widehat{D}$, in $M$ and call and immersion $f:\widehat{D}\to\widehat{S}$ \textbf{compatible} if satisfies the following conditions:
\begin{enumerate}[(1)]
\item
$f$ maps the non-exceptional boundary of $\widehat{D}$ to edges of $G$
\item
$f$ maps the exceptional boundary of $\widehat{D}$ to the exceptional boundary of $\widehat{S}$
\item
$f$ is injective on each component of the exceptional boundary of $\widehat D$ diffeomorphic to $S^1$ 
\end{enumerate}
\end{df}

If $f:\widehat D\to \widehat S$ is a compatible immersion of a marked disk $(D,M)$ then there is an associated cyclic \textbf{disk sequence} $a_n,\ldots,a_1$, $n=|M\cap\partial D|$, of boundary paths of $\widehat{S}$ by following the boundary of $D$ in counter-clockwise direction and mapping each interval belonging to the exceptional boundary of $\widehat{D}$ to $\partial\widehat{S}$ via $f$.
Furthermore, the associated area is
\[
A:=\int_{\widehat{D}}(\pi\circ f)^* \omega
\]
where $\pi:\widehat{S}\to S$ is the real blow-up and $\omega$ is the area form on $S$.
Such an immersed disk, up to reparameterization, contributes a term $t^A1_{X_0}$ to $\mk m_n(a_n,\ldots,a_1)$, a term $t^Ab$ to $\mk m_n(ba_n,\ldots,a_1)$ where $b$ is a boundary path composable with $a_n$, and a term $t^{A}(-1)^{|b|}b$ to $\mk m_n(a_n,\ldots,a_1b)$ where $b$ is a boundary path composable with $a_1$.
As in~\cite{hkk} one finds that 
\begin{equation}
\|a_1\|+\ldots+\|a_n\|=-2
\end{equation}
for any disk sequence $a_n,\ldots,a_1$, so $\mk m_n$ has correct degree $2-n$.

\begin{prop}
$\mc A_G(S)$ defined above is a curved $A_\infty$-category over $\mathbf R$.
\end{prop}

\begin{proof}
The proof is a slight extension of the one in \cite{hkk}, due to the presence of curvature and disks with interior marked points. It is also quite similar to the one in \cite{h_3cyteich} except for grading and signs.
We check that the various new terms in the $A_\infty$-structure equation cancel.
This is essentially a matter of drawing pictures and checking signs.
We refer to terms of the structure maps coming from $\mc A_G(S)_{\mathbf k}$ as \textit{constant terms} and terms coming from disk corrections as \textit{higher order terms}.

Consider first terms in the structure equation which involve a constant $\mk m_0$ and a higher order $\mk m_n$.
There are three possibilities.
The first is that there is a disk sequence $a_n,\ldots,a_{i+1},c,a_i,\ldots, a_1$ with $c=c_{p,X}$, see Figure~\ref{fig_cancel2}.
Then, since $c$ goes once around a component of the exceptional boundary, $a_n,\ldots,a_{i+1}a_i,\ldots,a_1$, where $a_{i+1}a_i$ is the concatenated path, is also a disk sequence with the same area $A>0$.
Thus the following two terms cancel, where $b$ is any boundary path composable with $a_1$.
\begin{gather*}
(-1)^{\|a_1\|+\ldots+\|a_i\|}\mk m_{n+1}(a_n,\ldots,a_{i+1},\mk m_0,a_i,\ldots,a_1b)=(-1)^{|b|+\|a_1\|+\ldots+\|a_i\|}t^Ab \\
=-(-1)^{\|a_1\|+\ldots+\|a_{i-1}\|} \mk m_{n-1}(a_n,\ldots,\mk m_2(a_{i+1},a_i),a_{i-1},\ldots a_1b)
\end{gather*}
and similarly for $b$ composable with $a_n$.
The second possibility is that $a_1,\ldots,a_n$ is a disk sequence and $b$ a boundary path from $X$ to $Y$ with $a_1b=c_{p,X}$ and thus $ba_1=c_{p,Y}$.
This gives the following two terms, which cancel:
\begin{gather*}
(-1)^{\|a_2\|+\ldots+\|a_n\|}\mk m_n(\mk m_0,a_n,\ldots,a_2)=-(-1)^{|a_1|}t^Ab=-\mk m_n(a_n,\ldots,a_2,\mk m_0)
\end{gather*}

\begin{figure}
\centering
\includegraphics[scale=1]{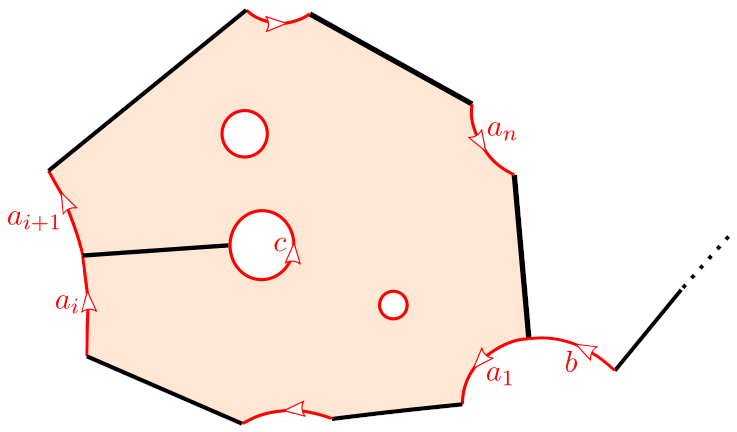}
\caption{Cancellation of terms of the form $\mk m_{n+1}(\ldots,\mk m_0,\ldots)$ and $\mk m_{n-1}(\ldots,\mk m_2(\ldots),\ldots)$}.
\label{fig_cancel2}.
\end{figure}

Next, we look at terms in the structure equation involving higher order corrections to $\mk m_n$ both times.
There are again two possibilities. 
The first is that a disk sequence $a_n,\ldots,a_1$ with area $A$ and a disk sequence $b_k,\ldots,b_1$ with area $B$ combine to a disk sequence 
\[
a_n,\ldots,a_{i+1},a_ib_k,\ldots,b_1a_{i-1},\ldots,a_1
\]
with area $A+B$, see Figure~\ref{fig_cancel1}.
If $c$ is a boundary path which concatenates with $a_1$ then this gives the following two terms, which cancel:
\begin{gather*}
(-1)^{|c|+\|a_1\|+\ldots+\|a_{i-1}\|}\mk m_n(a_n,\ldots,a_{i+1},\mk m_k(a_ib_k,\ldots,b_1),a_{i-1},\ldots,a_1c)=-(-1)^{|a_i|}t^{A+B} \\
=-(-1)^{|c|+\|a_1\|+\ldots+\|a_{i-2}\|}\mk m_{n+k-2}(a_n,\ldots,a_ib_k,\ldots,\mk m_2(b_1,a_{i-1}),\ldots,a_1c)
\end{gather*}
and similarly for a boundary path $c$ which concatenates with $a_n$.
In the second case there are disk sequences $a_n,\ldots,a_1$, $c_k,\ldots,c_1$ with areas $A$ and $C$ respectively, and a boundary paths $b$ such that the concatenation $a_1bc_k$ is defined.
We get terms
\begin{gather*}
\mk m_n(a_n,\ldots,a_2,\mk m_k(a_1bc_k,c_{k-1},\ldots,c_1))=(-1)^{|b|}t^{A+C}b \\
=-(-1)^{\|c_1\|+\ldots+\|c_{k-1}\|}\mk m_k(\mk m_n(a_n,\ldots,a_2,a_1bc_k),c_{k-1},\ldots,c_1)
\end{gather*}
which cancel.

\begin{figure}
\centering
\includegraphics[scale=1]{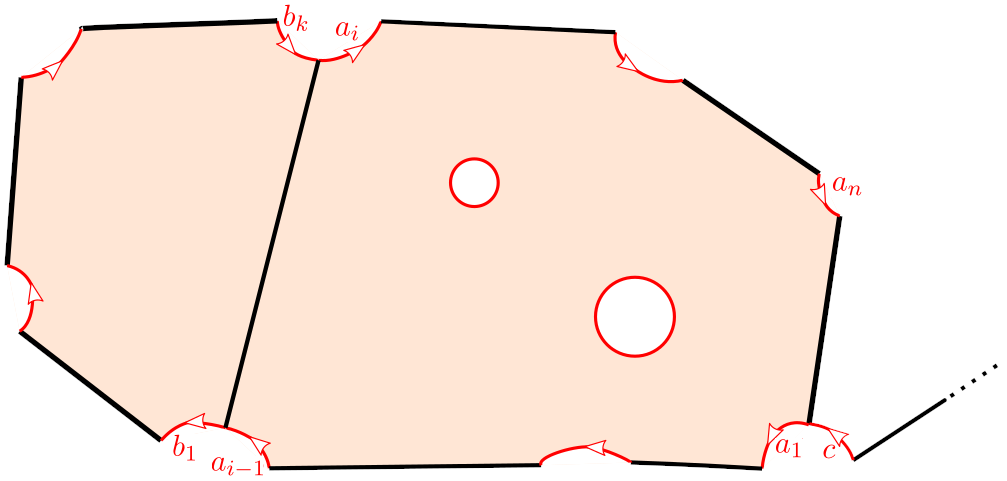}
\caption{Cancellation of terms of the form $\mk m_n(\ldots,\mk m_k(\ldots),\ldots)$ and $\mk m_{n+k-2}(\ldots,\mk m_2(\ldots),\ldots)$}.
\label{fig_cancel1}.
\end{figure}

It remains to consider terms involving a constant $\mk m_2$ and a higher order $\mk m_n$.
We have already encountered those terms where the concatenation produced by $\mk m_2$ happens inside the disk sequence associated with $\mk m_n$.
The remaining possibility looks as follows: $a_n,\ldots,a_1$ is a disk sequence of area $A$, and $b$ and $c$ are boundary path so that concatenation $a_1bc$ is defined.
Then we get terms
\begin{gather*}
\mk m_n(a_n,\ldots,a_2,\mk m_2(a_1b,c))=(-1)^{|b|}t^Abc=-(-1)^{\|c\|}\mk m_2(\mk m_n(a_n,\ldots,a_1b),c)
\end{gather*}
which cancel, and similarly if the concatenation $cba_n$ is defined.
\end{proof}

As a final step we incorporate the fiber category $\mc E$ by tensoring (as defined in Subsection~\ref{subsec_tensor}) and passing to twisted complexes,
\[
\mc F_G(S;\mc E):=\mathrm{Tw}\left(\mc E \otimes\mc A_G(S)\right),
\]
which is a curved $A_\infty$-category over $\mathbf R$.
Note that already $\mc E \otimes\mc A_G(S)$ does not depend on the choice of grading on the edges of $G$ since $\mc E$ is assumed to be triangulated, so in particular closed under shifts.

\subsection{Central fiber as homotopy limit}
\label{subsec_centralfiber}

The \textit{central fiber} of the category $\mc F_G(S;\mc E)$ over $\mathbf R$,
\[
\mc F_G(S;\mc E)_{\mathbf k}:=\mc F_G(S;\mc E)\otimes_{\mathbf R}\mathbf k=\mathrm{Tw}\left(\mc E\otimes \mc A_G(S)_{\mathbf k}\right)
\]
is a curved DG-category over $\mathbf k$.
The goal of this subsection is to show that the subcategory of flat objects, $V\left(\mc F_G(S;\mc E)_{\mathbf k}\right)$ has, up to quasi-equivalence, a more conceptual definition as global sections of a constructible sheaf of DG-categories over $G$.

Defining constructible sheaves of categories on a general space requires a certain amount of machinery, but the case of graphs is much more elementary.
The data of a constructible sheaf of categories on a graph $G$ is:
\begin{enumerate}[(1)]
\item
a category $\mc C_v$ for each vertex $v$ of $G$ (\textit{stalk at $v$}),
\item
a category $\mc C_e$ for each edge $e$ of $G$ (\textit{stalk at any point in the interior of $e$}),
\item
a functor $F_h:\mc C_v\to \mc C_e$ for each half-edge $h$ of $G$ belonging to the vertex $v$ and the edge $e$ (\textit{restriction functor}).
\end{enumerate}
We will simply take this as the definition of a \textit{sheaf of categories on $G$}.
The category of global sections is then the homotopy limit of the diagram with all the $\mc C_v$, $\mc C_e$, and $F_h$.
For the dual notion of a co-sheaf one instead has co-restriction functors $F_h:\mc C_e\to \mc C_v$.

Returning to the case at hand, we will define a sheaf of DG-categories on $G$ so that $\mc C_e$ is $\mc E$ for any edge $e$ and $\mc C_v$ is (non-canonically) quasi-equivalent to the functor category $\mathrm{Fun}(\mc C_n,\mc E)$ where $n+1$ (resp. $n$) is the valency of $v$ if $v\notin \partial S$ (resp. $v\in\partial S$). 
To define the sheaf of categories more precisely, we first define a \textit{co-sheaf} of $A_\infty$-categories $\mathbf k_G$ on $G$ which does not depend on $\mc E$.
This idea is originally due to Kontsevich and implemented e.g. in~\cite{dk_triangulated, hkk}.
Make an auxiliary choice of grading for each edge of $G$ as before.
The stalk of $\mathbf k_G$ over an edge is just the category with a single object $X$ with $\Hom(X,X)=\mathbf k$.
The stalk of $\mathbf k_G$ at a vertex $p\in G_0$ of valency $n+1$ is the category $\mc C_n$ from Subsection~\ref{subsec_bar} where $\Ob(\mc C_n)=\ZZ/(n+1)$ is identified with the set of half-edges meeting at $p$ in their \textit{counter-clockwise} cyclic order, and the degree of $\alpha_i:X_{i-1}\to X_i$ is $1-|a_i|$, where $a_i$ is the primitive boundary path from $X_i$ to $X_{i-1}$ going along the exceptional boundary component of $\widehat{S}$ corresponding to $p$.
The co-restriction functors send the unique object $X$ in the stalk over an edge to the object $X_i$ where $i$ is the corresponding half-edge.
This co-sheaf is, up to derived Morita equivalence of the stalks, independent of the choice of grading on the edges of the graph. 

The \textit{sheaf} $\mc E_G$ is obtained by taking $A_\infty$-functors to $\mc E$.
More precisely, if $\mathbf k_{G,p}$ is the stalk of $\mathbf k_G$ over $p$, then the stalk of $\mc E_G$ over $p$ is 
\[
\mc E_{G,p}=\mathrm{Fun}\left(\mathbf k_{G,p},\mc E\right)
\]
which is the DG-category of strictly unital $A_\infty$-functors.
Restriction functors of $\mc E_G$ are given by pre-composition with co-restriction functors of $\mathbf k_G$.
Up to equivalence, the sheaf of categories $\mc E_G$ does not depend on the choice of grading on the edges of $G$.

\begin{prop}
The category $V\left(\mc F_G(S;\mc E)_{\mathbf k}\right)$ is naturally quasi-equivalent to the triangulated DG-category of global sections of the sheaf $\mc E_G$.
\end{prop}

\begin{proof}
Geometrically, the definition of the category $V\left(\mc F_G(S;\mc E)_{\mathbf k}\right)$ involves only the ``local'' composition of boundary paths, not the ``global'' disk corrections.
As a consequence, the category $V\left(\mc F_G(S;\mc E)_{\mathbf k}\right)$ is the category of global sections, in the naive (1-categorical) sense, of a sheaf of categories on $G$ whose stalks on edges coincide with the ones of $\mc E_G$ and whose stalks at vertices are of the form $\mathrm{Tw}(\mc E\otimes(B\overline{\mc C_n})^{\vee+})$.
This sheaf is thus equivalent to $\mc E_G$ by Proposition~\ref{prop_funmc}.
To complete the proof, we must verify that naive global sections coincide with the global sections in the homotopical sense. This follows for the sheaf $\mc E_G$ from Corollary~\ref{cor_funfib}.
\end{proof}

Let us describe the correspondence of categories attached to a vertex $p$ of $G$ of valency $n+1$ more explicitly. We assume $p$ is an interior vertex --- this is the more complicated case.
First, we have the category $\mathbf k_{G,p}=\mc C_n$ with objects $X_i$, $i\in\ZZ/(n+1)$, corresponding to half-edges attached to $p$. 
In fact, it is useful to think of the $X_i$ as edges of an $n+1$-gon which is dual to the vertex.
A basis of morphisms is given by identity morphisms and $x_i:X_{i-1}\to X_i$ with degrees such that $\sum_i|x_i|=n-2$.
Non-zero $A_\infty$-structure maps are just
\[
\mk m_{n+1}(x_{i},x_{i-1},\ldots,x_{i-n}):=1_{X_i},\qquad i\in\ZZ/(n+1)
\]
as well as those $\mk m_2$ fixed by strict unitality.
A strictly unital $A_\infty$-functor $F:\mc C_n\to \mc E$ is determined by objects $F(X_i)\in\mc E$ and morphisms
\[
F_k(x_{i},x_{i-1},\ldots,x_{i-k+1}):F(X_{i-k})\to F(X_i)
\]
for a linear sequence of morphisms of arbitrary length $k\geq 1$ drawing from the cyclic sequence of morphisms $x_n,x_{n+1},\ldots,x_0$.

The cobar construction yields a curved DG-category with objects $A_i$, $i\in \ZZ/(n+1)$, morphisms freely generated by $a_i:A_{i-1}\to A_i$ of degree $|a_i|:=1-|x_i|$, and curvature
\[
h(A_i):=a_ia_{i-1}\cdots a_{i-n}.
\]
The twisted complex corresponding to $F:\mc A_n\to \mc E$ is of the form
\[
\left(\bigoplus_i F(X_i)\otimes A_i,\delta\right)
\]
where the coefficient of $a_{i}a_{i-1}\cdots a_{i-k+1}$ in $\delta$ is $F_k(x_{i},x_{i-1},\ldots,x_{i-k+1})\otimes e^{i-k}_k$.

\subsection{Refinement}
\label{subsec_refine}

Let $G\subset S$ be a graph as before.
In this subsection we show that the category $\mc F_G(S;\mc E)$ over $\mathbf R$ is unchanged, up to quasi-equivalence, if an edge $X$ of $G$ is split into two, by inserting a new 2-valent vertex.
To fix some notation, let $G'$ be the modified graph with the additional 2-valent vertex $p\in S$ and edges $Y,Z$ of $G'$ replacing the edge $X$ of $G$.
There are primitive boundary paths $a:Y\to Z$ and $b:Z\to Y$ on the boundary component of $\widehat{S}$ lying over $p\in S$, see Figure~\ref{fig_2val}.
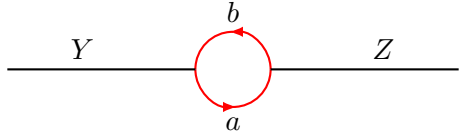
\begin{figure}
\centering
\begin{tikzpicture}
\draw[thick,draw=red] (0,.5) to [out=180,in=90] (-.5,0);
\draw[thick,->,draw=red] (.5,0) to [out=90,in=0] (-.05,.5);
\node[above] at (0,.5) {$b$};
\draw[thick,->,draw=red] (-.5,0) to [out=-90,in=180] (.05,-.5);
\draw[thick,draw=red] (0,-.5) to [out=0,in=-90] (.5,0);
\node[below] at (0,-.5) {$a$};
\draw[thick] (.5,0) to (3,0);
\node[above] at (2,0) {$Z$};
\draw[thick] (-.5,0) to (-3,0);
\node[above] at (-2,0) {$Y$};
\end{tikzpicture}
\caption{Real blow-up of 2-valent vertex.}
\label{fig_2val}
\end{figure}

We will define a canonical uncurved functor 
\begin{equation}\label{subdiv_functor}
F:\mc F_G(S;\mc E)\longrightarrow\mc F_{G'}(S;\mc E)
\end{equation}
of curved $A_\infty$-categories over $\mathbf R$.
The first step is to define a functor of curved $A_\infty$-categories
\[
D:\mc A_G(S)\longrightarrow\mathrm{Tw}\left(\mc A_{G'}(S)\right)
\]
with $F_1$-term only (i.e. uncurved and strict).
To be precise, the construction of $\mc A_G(S)$ involves a compatible choice of weights on boundary paths and a choice of grading of the edges.
For consistency, choose weights for $G'$, then these restrict to $G$, and choose gradings of the edges of $G$ and the grading on the two new edges of $G'$ obtained by restriction.
Then $|a|=|b|=1$ and we map
\[
D(X):=\left(Y\oplus Z,\begin{bmatrix} 0 & b \\ a & 0 \end{bmatrix}\right)
\]
where $X$ is considered as an object of $\mc A_G(S)$ and $Y,Z$ are considered objects of $\mc A_{G'}(S)$.
All other edges and all boundary paths are mapped to themselves by $D$.

\begin{lemma}
$D$ defined above is a functor of curved $A_\infty$-categories over $\mathbf R$.
\end{lemma}

\begin{proof}
In the definition of $\widetilde{\mk m}_0(D(X))$, the term coming from the Maurer--Cartan element is $-ab-ba$, which cancels with the summand of $\mk m_0(Y)$ and $\mk m_0(Z)$ coming from $p$, thus
\[
D(\mk m_0(X))=\widetilde{\mk m}_0(D(X)).
\]
as required. 
Compatibility with terms of $\mk m_2$ coming from composition of paths is clear, since $D$ is just the identity on boundary paths.
It remains to check compatibility of $D$ with disk corrections.
To see this, suppose $a_n,\ldots,a_1$ is a disk sequence for $G$ with $a_i:X_{i-1}\to X_i$.
There is a corresponding disk sequence $b_m,\ldots,b_1$, $m\geq n$ for $G'$ (of the same area) obtained by inserting for each $i$ with $X_i=X$ either $a$ or $b$ into the sequence, depending on which side of $X$ the disk is on.
Then the contribution of $a_n,\ldots,a_1$ to $\mk m_n$ in $\mc A_G(S)$ has a corresponding contribution to $\widetilde{\mk m}_n$ in $\mc A_{G'}(S)$ coming from $b_m,\ldots,b_1$ and the Maurer--Cartan element of $D(X)$.
\end{proof}

The functor $D$ induces a functor $F$ as in~\eqref{subdiv_functor}. 
Concretely, if an object of $\mc F_G(S;\mc E)$ comes with an object $E\in \mc E$ on the edge $X$, then the image of that object under $F$ has $E$ on both edges $Y$ and $Z$ and the Maurer--Cartan element has an additional summand $1_E\otimes a+1_E\otimes b$.

\begin{prop}
The functor $F:\mc F_G(S;\mc E)\longrightarrow\mc F_{G'}(S;\mc E)$ is a quasi-equivalence of $A_\infty$-categories over $\mathbf R$.
\end{prop}

\begin{proof}
The strategy is to first consider the induced functor on central fibers
\[
F_{\mathbf k}:\mc F_G(S;\mc E)_{\mathbf k}\longrightarrow\mc F_{G'}(S;\mc E)_{\mathbf k}
\]
and show that this is a quasi-equivalence, then appeal to Proposition~\ref{prop_quasiequiv} to deduce that $F$ itself is a quasi-equivalence.
Alternatively, a somewhat longer direct proof would also be possible.

In the previous subsection we constructed a quasi-equivalence between $V\left(\mc F_G(S;\mc E)_{\mathbf k}\right)$ and the category of global sections, $\Gamma(G;\mc E_G)$, of the sheaf $\mc E_G$.
The quasi-equivalence of the categories of global sections of $\mc E_G$ and $\mc E_{G'}$ follows from the quasi-equivalence of homotopy limits
\begin{equation*}
\Gamma(G;\mc E_G)=\mathrm{holim}\left(\cdots\longrightarrow\mc E\longleftarrow\cdots\right)\cong\mathrm{holim}\left(\cdots\longrightarrow\mc E\longleftarrow\mc E\longrightarrow \mc E\longleftarrow\cdots\right)=\Gamma(G';\mc E_{G'})
\end{equation*}
where the diagrams on the left and right are identical except for the parts shown and arrows between copies of $\mc E$ are identity functors.
Moreover, there is a commutative (up to natural isomorphism) diagram of functors
\[
\begin{tikzcd}
\mc F_G(S;\mc E)_{\mathbf k} \arrow[r,"F_{\mathbf k}"]\arrow[d] & \mc F_G(S;\mc E)_{\mathbf k} \arrow[d] \\
\Gamma(G;\mc E_G) \arrow[r] & \Gamma(G';\mc E_{G'})
\end{tikzcd}
\]
where the vertical arrows are quasi fully-faithful embeddings.
\end{proof}

\subsection{Edge contraction/expansion}
\label{subsec_edge}

The category $\mc F_G(S;\mc E)$ is invariant, up to quasi-equivalence, under certain basic modifications of $G$ which keep the regions cut out by $G$ and their areas the same.

\begin{df}
Let $X\subset G$ be an embedded graph and $X$ an edge of $G$ connecting vertices $p\neq q$ with $q\notin \partial S$.
A modified graph $G'$ is obtained by \textbf{contracting $X$}, i.e. removing $X$ and $q$ and re-attaching all other edges previously meeting $q$ to $p$.
This should be done so that all areas of regions cut out by $G$ (equivalently $G'$) are unchanged, and is thus well-defined up to Hamiltonian isotopy of the modified edges.
The inverse operation which adds an additional edge is referred to as \textbf{edge expansion}.
\end{df}

\begin{figure}
\centering
\includegraphics[scale=1]{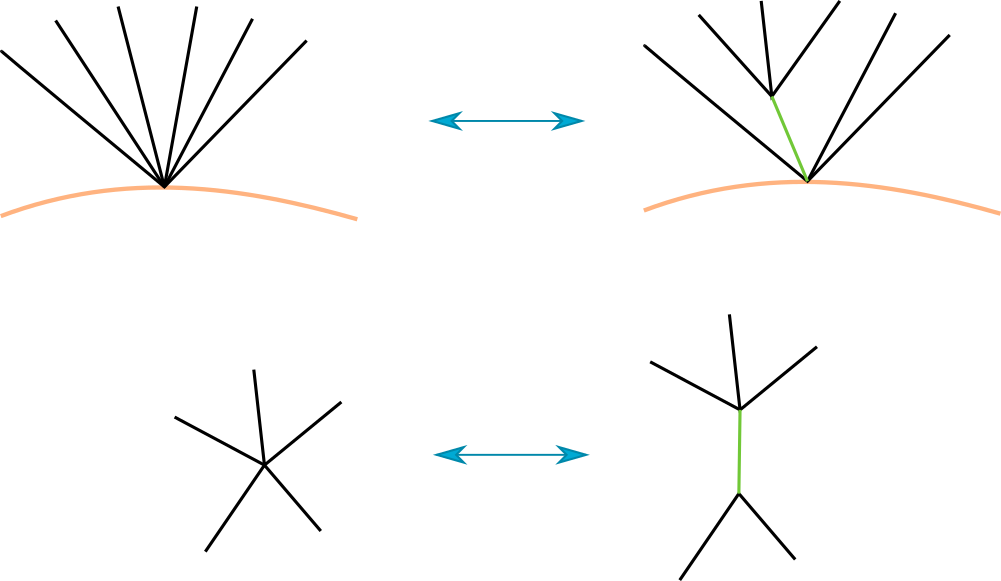}
\caption{Edge contraction/expansion at a vertex on the boundary (top) and a vertex in the interior (bottom).}
\label{fig_edgecontr}
\end{figure}

The main result of this subsection is the following proposition:

\begin{prop}\label{prop_edgecontr}
Suppose $G'$ is obtained from $G$ by edge contraction, then there is a natural quasi-equivalence
\[
F_G^{G'}:\mc F_{G'}(S;\mc E)\to \mc F_G(S;\mc E)
\]
of $A_\infty$-categories over $\mathbf R$.
\end{prop}

For technical reasons, it will be convenient to view edge expansion not as elementary but as the result of subdividing and identifying edges, see Figure~\ref{fig_edgemovie}.
\begin{figure}
\centering
\includegraphics[scale=1]{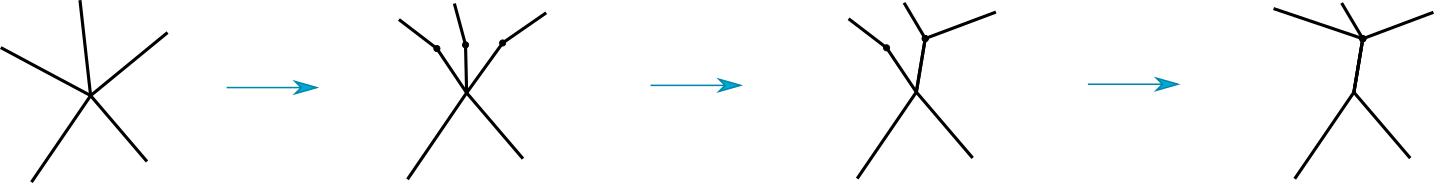}
\caption{Edge expansion realized by subdividing edges, followed by identifying pairs of edges, at least one of which ends in a 2-valent vertex.}
\label{fig_edgemovie}
\end{figure}
Since we already constructed an equivalence relating the category of a graph and its subdivision in the previous subsection, it remains to consider the operation of identifying edges.
Thus, suppose $G$ is a graph in $S$ with distinct vertices $p,q,r$ and with and edge from $p$ to $q$ and and edge from $p$ to $r$ which are attached next to each other at $p$.
We assume that $p,q,r$ are interior vertices, the cases with one or two boundary vertices ($q$ and $r$ cannot both be on the boundary) being similar.
Let $G'$ be the graph obtained by identifying the two edges (and thus also the two vertices $q$ and $r$).
We assume that this is done so that all areas of regions cut out by $G$ (equivalently $G'$) are unchanged.

\begin{figure}
\centering
\includegraphics[scale=1]{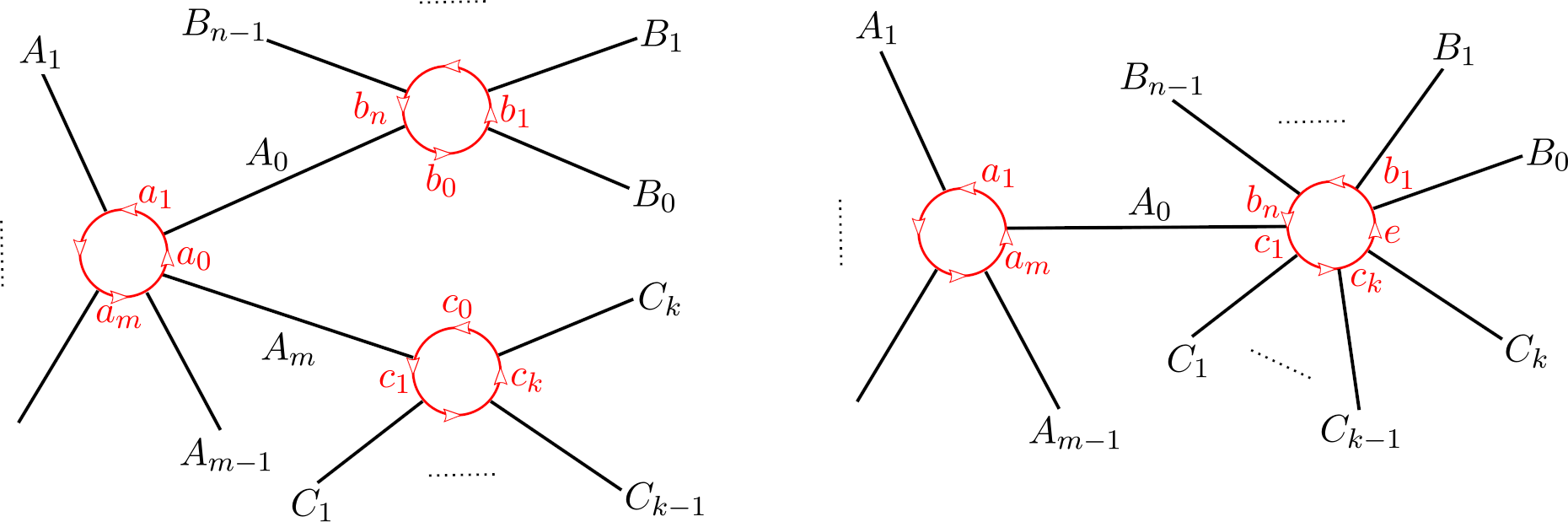}
\caption{Part of real blow-up for $G$ (left) and $G'$ (right). The pair of edges $A_0$ and $A_m$ of the graph $G$ are identified to yield the graph $G'$.}
\label{fig_edgeglue}
\end{figure}
The first step is to construct an $A_\infty$-functor  
\[
F:\mc A_G(S)\to \mc A_{G'}(S).
\]
Label the relevant edges and primitive boundary paths as in  Figure~\ref{fig_edgeglue}. 
We assume that the grading on $A_0$ and $A_m$ is chosen so that $|a_0|=0$.
$F$ is defined on objects by sending each edge in $G$ to the corresponding edge in $G'$. In particular, $F(A_0)=A_0$ and $F(A_m)=A_0$.
Define $F_1$ on the primitive boundary paths around $p,q,r$ by
\begin{gather*}
    F_1(a_0)=1_{A_0},\qquad F_1(a_1)=a_1,\qquad\ldots,\qquad F_1(a_m)=a_m \\
    F_1(b_0)=ec_k\cdots c_1,\qquad F_1(b_1)=b_1,\qquad\ldots,\qquad F_1(b_n)=b_n \\
    F_1(c_0)=b_n\cdots b_1e,\qquad F_1(c_1)=c_1,\qquad\ldots,\qquad F_1(c_k)=c_k
\end{gather*}
and sending each boundary path around a vertex  of $G$ other than $p,q,r$ to the corresponding boundary path for $G'$.
Also, $F_1$ is extended to non-primitive boundary paths by compatibility with concatenation.
There is also an $F_3$ term given by
\[
F_3(bb_0,a_0,c_0c)=F_1(b)eF_1(c)
\]
where $b$ is a boundary path starting at $B_0$ or $1_{B_0}$ and $c$ is a boundary path ending at $C_k$ or $1_{C_k}$.

\begin{lemma}
$F$ is an $A_\infty$-functor of curved $A_\infty$-categories over $\mathbf R$.
\end{lemma}

\begin{proof}
We need to check the $A_\infty$-functor equations~\eqref{ainfty_mor}.
We do this first for the corresponding categories before disk-corrections (terms with $t^{>0}$), which have only $\mk m_0$ and $\mk m_2$, then show compatibility with the deformation.

First look at the case of 0 inputs, then~\eqref{ainfty_mor} is just $F_1(\mk m_0(X))=\mk m_0(F(X))$ for all objects $X$.
Suppose $X=A_0$, then
\begin{align*}
    F_1\left(\mk m_0(A_0)\right)&=F_1(a_0a_m\cdots a_1+b_n\cdots b_1b_0) \\
    &=a_m\cdots a_1+b_n\cdots b_1ec_k\cdots c_1=\mk m_0(F(A_0))
\end{align*}
and similarly for other objects.

For the case of 1 input, there are no non-zero terms, since $\mk m_1=0$ and $F_2=0$.
The case of 2 inputs is more interesting. 
For boundary paths $x,y$ we have
\[
\mk m_2(F_1(x),F_1(y))=F_1(\mk m_2(x,y))
\]
unless $x$ or $y$ is $a_0$, in which case the right-hand side vanishes.
In those cases we have
\begin{align*}
    \mk m_2(F_1(a_0)),F_1(c_0c))&=(-1)^{|c_0c|}b_n\cdots b_1eF_1(c) =(-1)^{\|a_0\|+\|c_0c\|}F_3(\mk m_0(A_0),a_0,c_0c)
\end{align*}
and
\begin{align*}
    \mk m_2(F_1(bb_0),F_1(a_0))&=F_1(b)ec_k\cdots c_1 =F_3(bb_0,a_0,\mk m_0(A_m))
\end{align*}
using $|a_0|=0$.
For 3 inputs there are again no non-zero terms.
For 4 inputs, terms of the form $\mk m_2(F_3(x_3,x_2,x_1),x_0)$ cancel with terms of the form $F_3(x_3,x_2,\mk m_2(x_1,x_0))$ and terms of the form $\mk m_2(x_3,F_3(x_2,x_1,x_0))$ cancel with terms of the form $F_3(\mk m_2(x_3,x_2),x_1,x_0)$.
For $\geq 5$ inputs there are no non-zero terms.

To show compatibility of $F$ with disk corrections note that given a disk sequence $\alpha_k,\alpha_{k-1},\ldots,\alpha_1$ for $G$ there is a corresponding disk sequence $\beta_l,\beta_{l-1},\ldots,\beta_1$, $l\leq k$ for $G'$ where the original sequence is modified according to the rule
\[
\cdots b_0,a_0,c_0\cdots\mapsto \cdots e \cdots
\]
i.e. for any $i$ with $\alpha_i=a_0$, $\alpha_i$ is deleted, $b_0$ is cut from $\alpha_{i+1}$, $c_0$ is cut from $\alpha_{i-1}$, and the two ends are concatenated yielding a sequence of length two less.
This shows that the term $F_1(\mk m_k(\alpha_k,\ldots,\alpha_1))$ cancels with a term of the form $\mk m_l(F_*(\ldots),\ldots,F_*(\ldots))$ where each $*=1$ or $3$.
\end{proof}

Having constructed a functor $F:\mc A_{G}(S)\to\mc A_{G'}(S)$ we obtain a functor
\[
F:\mc F_G(S;\mc E)=\mathrm{Tw}(\mc E\otimes\mc A_{G}(S))\to\mathrm{Tw}(\mc E\otimes \mc A_{G'}(S))=\mc F_{G'}(S;\mc E)
\]
for which we use the same letter.

\begin{prop}
Suppose one of $q,r$, say $r$, is 2-valent.
The functor $F:\mc F_G(S;\mc E)\to \mc F_{G'}(S;\mc E)$ constructed above is a quasi-equivalence of curved $A_\infty$-categories over $\mathbf R$. 
\end{prop}

The restriction on the valency is not really necessary, as we will see later.

\begin{proof}
The strategy is to first show that the functor induces a quasi-equivalence on central fibers
\[
F_{\mathbf k}:\Gamma(\mc E_G;G)\cong V\left(\mc F_{G}(S;\mc E)_{\mathbf k}\right)\to V\left(\mc F_{G'}(S;\mc E)_{\mathbf k}\right)\cong \Gamma(\mc E_{G'};G')
\]
and then conclude by Proposition~\ref{prop_quasiequiv} that $F$ is itself a quasi-equivalence.
Since $r$ is 2-valent, the passage from $G$ to $G'$ can alternatively be thought of as sliding the edge on which $r$ lies from $p$ to $q$ along $A_0$.
Such a modification can in turn be realized by contracting and expanding $A_0$.
Since it is known that $\Gamma(\mc E_G;G)$ is invariant under edge contractions/expansions~\cite{dk_triangulated,hkk}, which just comes down to the fact the the homotopy pullback of $\mc A_m$ and $\mc A_n$ is $\mc A_{m+n-1}$, it follows that the source and target of $F_0$ are quasi-equivalent.
It remains to show that $F_0$ is compatible with this quasi-equivalence.

To see this, let $U\subset G$ be a contractible open neighbourhood of $p,q,r$ and $U'\subset G'$ be a contractible open neighbourhood of $p,q$.
Then both $\Gamma(\mc E_G;U)$ and $\Gamma(\mc E_{G'};U')$ are naturally identified with the category of functors from the path algebra of the $A_{m+n-1}$-quiver (with all arrow oriented to the right, say) to $\mc E$ so that the images of the $m+n-1$ vertices are the objects on the edges $A_{m-1},A_{m-2},\ldots,A_1,B_{n-1},\ldots,B_1,B_0$. 
As a special case of Proposition~\ref{prop_funmc}, this category is equivalent to $V(\mathrm{Tw}(\mc E\otimes Q_{m+n-1}))$ where $Q_{m+n-1}$ is the path algebra of the $A_{m+n-1}$-quiver but with relations that all consecutive pairs of arrows compose to zero.
Finally, there are functors from $V(\mc F_G(S;\mc E)_{\mathbf k})$ and $V(\mc F_{G'}(S;\mc E)_{\mathbf k})$ to $V(\mathrm{Tw}(\mc E\otimes Q_{m+n-1}))$ which extract from a Maurer--Cartan element only those summands involving the boundary paths $a_1,\ldots,a_m$, $b_1,\ldots,b_n$, and their compositions.
Since these are unchanged by $F$, we get a commutative diagram of functors
\[
\begin{tikzcd}
V(\mc F_G(S;\mc E)_{\mathbf k}) \arrow[r,"F_{\mathbf k}"]\arrow[d] & V(\mc F_{G'}(S;\mc E)_{\mathbf k})\arrow[dl] \\
V(\mathrm{Tw}(\mc E\otimes Q_{m+n-1}))
\end{tikzcd}
\]
which shows the compatibility of $F_{\mathbf k}$ with the identification we have a priori. 
\end{proof}

\begin{proof}[Proof of Proposition~\ref{prop_edgecontr}]
This follows from the previous proposition and the fact that an edge expansion can be realized in terms of subdivision and gluing edges, see Figure~\ref{fig_edgemovie}.
\end{proof}

\subsection{General fiber}
\label{subsec_generalfiber}

As before, we fix a compact surface with boundary $S$ with grading structure $\nu$, area form $\omega$, and coefficient category $\mc E$.
Also fix a finite subset $M\subset\partial S$ of the boundary and an affine structure on $S$. 
If $\nu$ is not just smooth but holomorphic then there is an affine structure coming from the flat metric $|\nu|$. We always assume this affine structure has been chosen if $\nu$ is holomorphic.
We no longer consider a fixed graph $G$, but the set $\GG$ of graphs in $S$ with the following properties.
\begin{df}
An embedded graph $G$ in $S$ is \textbf{compatible} if
\begin{enumerate}
\item $G_0\cap\partial S\subset M$,
\item edges of $G$ are linear with respect to the affine structure on $S$.
\end{enumerate}
\end{df}
The second condition is to ensure that the union of a pair of graphs is a graph. 
Alternatively, we could require edges to be semi-algebraic with respect to some algebraic structure on $S$.

We say that $G'\in\GG$ is a \textit{subdivision} of $G$ if it is obtained by subdividing some of the edges of $G$ by inserting new 2-valent vertices, or if $G=G'$.
We say that $G\in\GG$ is a \textit{subgraph} of $G'$ if it is obtained by deleting some of the edges and vertices of $G'$, or if $G=G'$.
Define a partial order on $\GG$ by $G\leq G'$ if some subdivision of $G$ is a subgraph of $G'$.
For a pair of graphs $G,G'\in\GG$, the union $G\cup G'\in\GG$ is the graph whose vertices are the vertices of $G$ and $G'$ together with the intersections points of edges, and whose edges are pieces of edges of $G$ or $G'$.
We then have $G,G'\leq G\cup G'$.

If $G\leq G'$ then there is a natural functor 
\[
F_{G',G}:\mc F_G(S;\mc E)\longrightarrow \mc F_{G'}(S;\mc E)
\]
and these satisfy
\[
F_{G'',G}=F_{G'',G'}\circ F_{G',G}
\]
for $G\leq G'\leq G''$.
In the case where $G'$ is a subdivision of $G$, this functor was defined in Subsection~\ref{subsec_refine}.
In the case where $G$ is a subgraph of $G'$, it is simply inclusion of a full subcategory. 
In general, it is the composition of the two functors.

Define
\[
\mc F(S,M;\mc E):=\varinjlim_{G\in\GG}\mc F_G(S;\mc E)
\]
where the direct limit is the usual one (not homotopical), and
\[
\mc F(S,M;\mc E)_{\mathbf K}:=\mc F(S,M;\mc E)\otimes_{\mathbf R}\mathbf K.
\]

A graph $G\in\GG$ is a \textbf{skeleton} of $(S,M)$ if $M\subset G_0$ and it is a deformation retract of $S$.

\begin{prop}
\label{prop_skelgen}
Suppose $G\in\GG$ contains a skeleton as a subgraph, then the canonical map
\[
\mc F_G(S;\mc E)_{\mathbf K}\longrightarrow\mc F(S,M;\mc E)_{\mathbf K}
\]
is a quasi-equivalence.
\end{prop}

The proof of the proposition will be based on several lemmas.
Suppose first $G\subset S$ is an embedded graph, $p$ a vertex of $G$ on $\partial S$, and $X$ an edge of $G$ which starts and ends at $p$ and so that both half-edges of $X$ are attached to $p$ before any other ones in their counter-clockwise order, and so that $X$ bounds a 1-gon of area $A$. (See Figure~\ref{fig_1gon}.)

\begin{figure}
\centering
\includegraphics[scale=1]{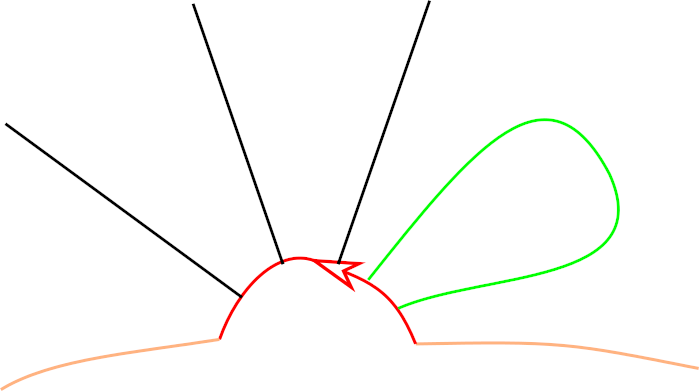}
\caption{Edge bounding a 1-gon in $S$, lifted to $\widehat{S}$.}
\label{fig_1gon}
\end{figure}

\begin{lemma}[1-gon removal]
Suppose $G\subset S$ and $X$ is a loop of $G$ as above.
Let $G':=G\setminus X$, i.e. the graph with $X$ removed, then the inclusion functor $\mc F_{G'}(S;\mc E)\to \mc F_G(S;\mc E)$ induces a quasi-equivalence 
\[
\mc F_{G'}(S;\mc E)_{\mathbf K} \to \mc F_G(S;\mc E)_{\mathbf K}
\]
of curved $A_\infty$-categories over $\mathbf K$.
\end{lemma}

\begin{proof}
The functor induces a quasi-equivalence on $\Hom$ already over $R$. It remain to show essential surjectivity. 

An object of $V(\mc F_G(S;\mc E))$ is a twisted complex of the form $((E\otimes X)\oplus B,\delta)$ for some $B\in V(\mc F_G(S;\mc E))$ and bounding cochain $\delta$.
By the condition on $X$, $\Hom(Y,X)=0$ for any edge $Y$ of $G'$, so $\delta$ is of the form
\[
\delta=\begin{pmatrix} \delta_{X} & 0 \\ \delta_{BX} & \delta_B \end{pmatrix}
\]
with respect to the direct sum decomposition, thus the extension of $(E\otimes X,\delta_X)$ by $(B,\delta_B)$.
We claim that $(E\otimes X,\delta_X)$ is a zero object in the homotopy category of $\mc F_G(S;\mc E)_{\mathbf K}$, from which the claim of the lemma follows.
To see this, let $\eta$ be the boundary path which starts and ends at the two endpoints of $X$, respectively. The 1-gon bounded by $X$ gives $\mk m_1(\eta)=t^{A}1_X$ thus
\[
\mk m_1(1_E\otimes t^{-A}\eta)=1_E\otimes 1_X
\]
thus $\mathrm{Ext}^{\bullet}(E\otimes X,E\otimes X)=0$.
\end{proof}

\begin{lemma}[edge removal]
\label{edgeremoval}
Let $G\subset S$ be an embedded graph so that all regions cut out by $G$ are simply connected.
Suppose $X$ is an edge of $G$ which bounds both an exterior region $R_{\mathrm{ext}}$ and an interior region $R_{\mathrm{int}}$, and let $G'=G\setminus X$.
Then then the inclusion functor $\mc F_{G'}(S;\mc E)\to \mc F_G(S;\mc E)$ induces a quasi-equivalence 
\[
\mc F_{G'}(S;\mc E)_{\mathbf K} \to \mc F_G(S;\mc E)_{\mathbf K}
\]
of curved $A_\infty$-categories over $\mathbf K$.
\end{lemma}

\begin{proof}
The idea is to reduce this to 1-gon removal by a series of edge contractions/expansions.
First, the graph is modified so that $R_{\mathrm{int}}$ is bounded by $X$ only, i.e. so that $X$ bounds a 1-gon. 
This is achieved by:
\begin{enumerate}
\item 
Making all vertices on the boundary of the interior region interior vertices by pushing-off any vertices on $\partial S$ while adding an edge connecting them back to the boundary.
\item
Removing any loops along the boundary of $R_{\mathrm{int}}$ by creating a parallel edge.
\item
Contracting all edges except $X$, which is possible by the previous two steps.
\end{enumerate} 
In the second stage of the procedure, the graph is modified so that $X$ follows immediately after $\partial S$ as one goes around the boundary of $R_{\mathrm{ext}}$ in counter-clockwise order.
This is achieved by the same three-step strategy as above but applied to those vertices and edges (if any) which come after $\partial S$ but before $X$ as one travels around the boundary of $R_{\mathrm{ext}}$ in counter-clockwise order.
\end{proof}

%%%%%%%%%%%%%%%%%%%%%%%%%%%%%%%%%%%%%%%%%%%%%%%%%%%%%%%%%%%%%%%%%%%%%%%%%%%%%%%

\section{Spectral networks and stability conditions}
\label{sec_general}

This Section contains definitions, results, and conjectures about the relation between spectral networks and stability conditions in the context of general base surfaces and fiber categories, while in later sections both will be significantly restricted.
Section~\ref{subsec_stab} recalls the definition and basic properties of stability conditions, while also fixing some conventions which differ from those in Bridgeland's original paper~\cite{bridgeland07}.
In Section~\ref{subsec_spectral} we define spectral networks, in our sense, in some generality. This relies on the work in the previous two sections of the paper.
Section~\ref{subsec_mainconj} contains the main conjecture about spectral networks and stability conditions at the level of generality at which we can make it precise.
Finally, Section~\ref{subsec_uniqueness} states and proves a general uniqueness result for spectral network representatives.

\subsection{Stability conditions}
\label{subsec_stab}

Stability conditions on triangulated categories were introduced in seminal work of Bridgeland~\cite{bridgeland07}, based on ideas from algebraic geometry and D-branes in string theory.
We recall the basic definitions.

Let $\mc C$ be a triangulated category (essentially small) and $\mathrm{cl}:K_0(\mc C)\to\Gamma$ a homomorphism to a  finite rank abelian group $\Gamma$. 
(In the examples considered in this paper, $K_0(\mc C)$ is finitely generated so we can take $\mathrm{cl}$ to be the identity map.)
A \textbf{stability condition} for $(\mc C,\Gamma,\mathrm{cl})$ is given by 
\begin{enumerate}
\item the \textit{central charge}: an additive map $Z:\Gamma\to\CC$, and
\item full additive subcategories $\mc C_\phi\subset\mc C$, $\phi\in\RR$, the \textit{semistable objects of phase $\phi$}
\end{enumerate}
such that
\begin{itemize}
\item $\mc C_{\phi+1}=\mc C_\phi[1]$,
\item $\Hom(E_2,E_1)=0$ for $E_i\in\mc C_{\phi_i}$, $\phi_1<\phi_2$.
\item For any $E\in\mc C$ there is a \textit{Harder--Narasimhan filtration} 
\[
\begin{tikzcd}[column sep=tiny]
0=E_0 \arrow{rr} & & E_1\arrow{rr}\arrow{dl} & & E_2\arrow{rr}\arrow{dl} & &\cdots \arrow{rr} &  & E_{n-1} \arrow{rr} & & E_n\cong E \arrow{dl} \\
& A_1 \arrow[dashed]{ul} & & A_2 \arrow[dashed]{ul} & & & &  & & A_n \arrow[dashed]{ul}
\end{tikzcd}
\]
with \textit{semistable factors} $0\neq A_i\in\mc C_{\phi_i}$, $\phi_1>\phi_2>\ldots>\phi_n$, and the triangles are exact.
\item $Z(E):=Z(\mathrm{cl}(E))\in\RR_{>0}e^{\pi i\phi}$ for $E\in\mc C_{\phi}$, $E\neq 0$.
\item The \textit{support property}: There is a norm $\|.\|$ on $\Gamma\otimes_{\ZZ}\RR$ and $C>0$ such that
\[
\|\gamma\|\leq C|Z(\gamma)|
\]
for any $\gamma\in\Gamma$ which is the class of a semistable object. 
\end{itemize}

\begin{remark}
In~\cite{bridgeland07} there is no $\Gamma$, and instead of the support property, which was introduced in~\cite{ks}, one requires a closely related finiteness condition.
\end{remark}

The stability condition is completely determined by $Z$ and the abelian subcategory $\mc C_{(0,1]}\subset \mc C$ of objects with semistable factors of phase contained in the interval $(0,1]$, which is the heart of a bounded t-structure.
This leads to an equivalent definition in terms of t-structures, see~\cite{bridgeland07} for details.

Let $\mathrm{Stab}(\mc C)=\mathrm{Stab}(\mc C,\Gamma,\mathrm{cl})$ be the set of stability condition for ($\mc C$, $\Gamma$, $\mathrm{cl}$). 
There is a natural (metric) topology on $\mathrm{Stab}(\mc C)$ which makes 
\[
\mathrm{Stab}(\mc C,\Gamma,\mathrm{cl})\to\Hom(\Gamma,\CC),\qquad \left((\mc C_\phi)_{\phi\in\RR},Z\right)\mapsto Z
\]
a local homeomorphism. 
Thus, $\mathrm{Stab}(\mc C)$ naturally has the structure of a complex manifold of dimension $\mathrm{rk}(\Gamma)$.

Suppose $(\mc C,\Gamma_{\mc C},\mathrm{cl}_{\mc C})$ and
$(\mc D,\Gamma_{\mc D},\mathrm{cl}_{\mc D})$ are triples as above.
If $F:\mc C\to \mc D$ is an equivalence of triangulated categories together with an isomorphism $f:\Gamma_{\mc C}\to \Gamma_{\mc D}$ compatible with the induced map $K_0(F):K_0(\mc C)\to K_0(\mc D)$, then there is an induced map 
\[
F_*:\mathrm{Stab}(\mc C,\Gamma_{\mc C},\mathrm{cl}_{\mc C})\longrightarrow\mathrm{Stab}(\mc D,\Gamma_{\mc D},\mathrm{cl}_{\mc D})
\]
sending $(Z,(\mc C_\phi)_{\phi\in\RR})$ to $(Z\circ f^{-1},\left(F(\mc C_\phi)_{\phi\in\RR})\right)$.
In particular, if $\mathrm{cl}:K_0(\mc C)\to \Gamma$ is constructed naturally from $\mc C$, as is usually the case in practice, then $\mathrm{Aut}(\mc C)$ acts on $\mathrm{Stab}(\mc C)$ on the left.

The universal cover of the group $\mathrm{GL}^+(2,\RR)$ also acts on $\mathrm{Stab}(\mc C)$, however we will mostly be interested in the action of the subgroup $\CC$ which is given by
\[
z\cdot \left(Z,(\mc C_\phi)_{\phi\in\RR}\right):=\left(e^{\pi i z}Z,(\mc C_{\phi-\mathrm{Re}(z)})_{\phi\in\RR}\right)
\]
so, in particular, $1\in\CC$ acts like the shift functor $[-1]$.

\subsection{Spectral networks}
\label{subsec_spectral}

In Section~\ref{sec_aside} we defined the Fukaya category with coefficients in a fixed triangulated DG-category $\mc E$. 
More generally, one expects that as coefficients of the Fukaya category one can take any \textit{(perverse) schober}~\cite{kapranov_schechtman}.
The full theory is still under development~\cite{bks, donovan19,harder_katzarkov,kss_infrared,christ,dkss21}. 
For much of this section we restrict to the following special class of schobers.

\begin{df}
Let $S$ be a smooth, oriented surface. 
A \textbf{schober without singular fibers} on $S$ is a local system $\mc E$ of triangulated DG-categories on the projectivized tangent bundle $\PP TS$ with monodromy $[1]$ along any fiber (for counterclockwise rotation).
\end{df}

If $\PP TS$ admits a section $\nu$ (e.g. coming from the horizontal foliation of a quadratic differential) then pullback along $\nu:S\to \PP TS$ gives a correspondence, depending on $\nu$, between schobers without singular fibers on $S$ and local systems of triangulated DG-categories on $S$.
In particular, if $\mc E$ is a triangulated DG-category, thought of as a constant local system on $S$, and $\nu$ a section of $\PP TS$ (which is the setup in Section~\ref{sec_aside}), then this data defines a schober without singular fibers on $S$.

\begin{remark}
More general schobers are allowed to have singularities on a discrete subset $D\subset S$, which means, roughly, that the local system is defined on $\PP T(S\setminus D)$ and there is the data of a spherical adjunction (see~\cite{spherical}) for each $p\in D$ between the ``nearby fiber'' and given ``vanishing cycles'' categories.
As we do not know the precise formulation of the main conjecture about stability conditions and spectral networks in this case, we restrict to the case without singular fibers.
\end{remark}

We expect that the definition of the Fukaya category with coefficients, $\mc F(S,M;\mc E)$, can be extended to the case where $\mc E$ is a schober with or without singular fibers.
The basic idea, at least in the case without singular fibers, is that objects are graphs $G$ with additional algebraic data as before, but instead of having an object in a fixed category $\mc E$ for each edge of $G$, one has, for each $p\in G$ which is not a vertex, an object $E_p$ in $\mc E_{p,l}$, where $l$ is the tangent line to $G$ at $p$, so that $E_p$ is essentially locally constant in $p$.

Given a complex structure on $S$, compatible with its orientation, let $T^{1,0}S$ be the holomorphic tangent bundle, $(T^{1,0}S)^{\otimes 2}$ its square as a holomorphic line bundle, and $0_S$ the zero section.
The quotient of $(T^{1,0}S)^{\otimes 2}\setminus 0_S$ by its $\RR_{>0}$ scaling action is naturally identified with $\PP TS$, and the quotient map is a fiberwise homotopy equivalence.
Thus, we may consider schobers without singular fibers, $\mc E$, as local systems on this $\CC^\times$-bundle instead.
The following definition is made from this point of view.

\begin{df}
A \textbf{holomorphic family of stability conditions} on $\mc E$ is choice of stability condition $\sigma_{p,q}$ on each fiber $\mc E_{p,q}$ such that
\begin{enumerate}
\item
$\sigma_{p,q}$ depends holomorphically on $(p,q)\in (T^{1,0}S)^{\otimes 2}\setminus 0_S$ (which makes sense since $\mc E_{p,q}$ is locally constant and $\mathrm{Stab}(\mc E_{p,q})$ is a complex manifold), and
\item
$\sigma_{p,q}$ depends equivariantly on $q$:
\begin{equation}
\label{stabfam_equiv}
\sigma_{p,e^{2\pi iz}q}=z\tau_*\sigma_{p,q}
\end{equation}
where $\tau:\mc E_{p,q}\to \mc E_{p,e^{2\pi iz}q}$ is the equivalence given by parallel transport along the path $e^{2\pi i tz}$, $t\in [0,1]$, and $z\in\CC$ acts on $\mathrm{Stab}$ in the usual way (so $1\in \CC$ acts like pushforward along the functor $[-1]$).
\end{enumerate}
\end{df}

As an example, consider the special case where the fibers $\mc E_{p,q}$ are (non-canonically) equivalent to $\mathrm{Perf}(\mathbf k)$, the category of finite-dimensional cochain complexes over $\mathbf k$.
Each $\mc E_{p,q}$ then has a unique, up to shift and isomorphism, indecomposable object $E_{p,q}$, thus $W(p,q):=(Z(E_{p,q}))^2$ is well-defined and satisfies $W(p,\lambda q)=\lambda W(p,q)$ for $\lambda\in \CC^\times$, thus extends (setting $W(p,0)=0$) to a non-vanishing holomorphic quadratic differential on $S$.

We now come to the definition of a spectral network in the case of schobers without singular fibers together with a holomorphic family of stability condition.
While we haven not precisely defined objects of $\mc F(S,M;\mc E)$ in the case where $\mc E$ has monodromy, what matters in the definition of a spectral network is only the underlying graph $G$ and the objects $E_p\in \mc E_{p,l}$ for $p\in G$ not a vertex.

\begin{df}
An object $E\in V(\mc F(S,M;\mc E))$ supported on a graph $G\subset S$ is a \textbf{spectral network} of phase $\phi\in\RR$ if $E_p$ is semistable of phase $\phi$ for any $p\in G$ which is not a vertex.
In this case, $E$ is said to be a \textbf{spectral network representative} of any object isomorphic to its image in $V\left(\mc F(S,M;\mc E)_{\mathbf K}\right)$.
\end{df}

Furthermore we define the central charge and mass for arbitrary objects, $X$, with underlying graph $G$ and objects $E_p$.
Note first that $Z(E_p)$ restricts to a $\CC$-valued 1-density on each edge of the graph.
Let
\begin{equation}
\label{graph_Z}
Z(X):=\sum_{\text{edges } e}\int_e Z(E_p)
\end{equation}
and
\begin{equation}
m(X):=\sum_{\text{edges } e}\int_e |Z(E_p)|
\end{equation}
where the sums are over the set of edges of $G$ and the integrals are along the interior of the given edge.
One has the \textit{BPS inequality} $|Z(X)|\leq m(X)$ which is saturated for spectral networks.

Let us consider a special case where both the local system of categories and the family of stability conditions are essentially constant.
We start with a triangulated DG-category $\mc E$ with a stability condition $\sigma$ and a nowhere vanishing holomorphic quadratic differential $\nu$ on $S$.
The pointwise dual of $\nu$ is a section $\nu^\vee$ of $(T^{1,0}S)^{\otimes 2}$ along which the local system of categories is the trivial one.
The family of stability conditions is characterized by $\sigma_{p,\nu^{\vee}(p)}=\sigma$ and $\CC$-equivariance~\eqref{stabfam_equiv}. 
We call this a \textbf{constant family of stability conditions} on $S$.
An object in the fiber category $\mc E_{p,q}$ of the local system can be written as $(E,\phi)$, where $E\in \mc E$ and $\phi\in\RR$ is such that $e^{2\pi i\phi}\nu^\vee=q$, and $(E[1],\phi)=(E,\phi+1)$.
The phase of an object in $\mc F(S,M;\mc E)$ at a point $p\in S$ in the supporting graph can then be written as a sum of contributions from the \textit{base} $(S,\nu)$ and the \textit{fiber} $(E,\sigma)$:
\begin{equation}
\label{sum_of_phases}
\phi_{\mathrm{total}}(E,\phi_{\mathrm{base}})=\phi_{\mathrm{base}}+\phi_{\sigma}(E)
\end{equation}
In particular, for a spectral network the edges of the graph must be straight lines with respect to the flat metric $|\nu|$ and the objects on the edges semistable such that the total phase as above is constant across the graph.

\begin{prop}
Assume the family of stability conditions is constant as above. Then if $X_1,X_2\in\mc F(S,M;\mc E)$ are spectral networks of phases $\phi_1$ and $\phi_2$, respectively, with $\phi_1<\phi_2$, then $\mathrm{Ext}^{\leq 0}(X_2,X_1)=0$
\end{prop}

We expect the statement to be true in general, but restrict to the case of constant family of stability conditions for simplicity. 

\begin{proof}
It suffices to show that $\mathrm{Ext}^{\leq 0}(X_2,X_1)=0$ in $\mc F(S,M;\mc E)_{\mathbf k}$, since the rank can only decrease when adding higher order corrections.
Subdividing as necessary, we can assume that $X_1$ and $X_2$ are defined on subgraphs of the same graph in $S$.
Furthermore, given the sheaf property of $V\left(\mc F(S,M;\mc E)_{\mathbf k}\right)$, it suffices to check the statement locally, i.e. for a single vertex. (Here it is essential that the statement is of the form $\mathrm{Ext}^{\leq i}=0$ not just $\mathrm{Ext}^0=0$, since the gluing can shift non-trivial extension classes up.)

Thus, suppose $Y_i$ is an edge of $X_i$ with grading $\alpha_i$ and semistable object $E_i$ of phase $\beta_i$, $i=1,2$, so that $Y_1$ and $Y_2$ meet at a vertex $p$.
Let $a$ be the shortest (less than a full circle) boundary path from $Y_2$ to $Y_1$ in the real blow-up of $p$ or $1_{Y_1}$ if $Y_1=Y_2$. 
For longer boundary paths, the degree will only be larger, so it suffices to consider this case.

After possibly shifting $\alpha_1$ and $E_1$ in opposite directions, which does not change the underlying object $X_1$ up to isomorphism, we can assume that $|a|=0$.
From the definition of $|a|$ it follows that  $\alpha_1-\alpha_2\geq 0$.
By assumption, $X_2$ has strictly bigger phase than $X_1$, so $\beta_2-\beta_1>\alpha_1-\alpha_2\geq 0$, thus $\mathrm{Ext}^{\leq 0}_{\mc E}(E_2,E_1)=0$, which holds for any stability condition by definition.
This shows that any morphism of the form $f\otimes a:E_2\otimes Y_2\to E_1\otimes Y_1$, which is non-zero in $\mathrm{Ext}(X_2,X_1)$ is in strictly positive degree.
\end{proof}

One consequence of the above proposition is that a non-zero object in $V(\mc F(S,M;\mc E)_{\mathbf K})$ cannot have a pair of spectral network representatives with distinct phases.
Another application is the following.

\begin{prop}
Under the same assumptions as in the previous proposition, and also that there is a stability condition on $H^0(\mc F(S,M;\mc E)_{\mathbf K})$ so that every semistable object of phase $\phi$ has a spectral network representative of phase $\phi$, no unstable object can have a spectral network representative of any phase.
\end{prop}

This proposition is useful when checking the main conjecture (Conjecture~\ref{main_conj} below) in specific examples --- see later sections of this paper.
In those examples, we do not define the stability condition in terms of spectral networks, but instead construct it directly, guessing what the correct one is.
If we can then show that all semistable objects have spectral network representatives of the correct phase it follows, using the above proposition, that the semistable objects are exactly those which admit a spectral network representative.

\begin{proof}
Suppose that $X\in V(\mc F(S,M;\mc E))$ is a spectral network of phase $\phi$ so that its image in $V(\mc F(S,M;\mc E)_{\mathbf K})$, also denoted by $X$, is unstable.
By assumption, $X$ has a Harder--Narasimhan filtration, thus fits into an exact triangle (the last triangle in the tower of extensions)
\[
Y\xrightarrow{x} X\xrightarrow{a} A\xrightarrow{y} Y[1]
\]
where $A$ is non-zero semistable of some phase $\psi$ and $Y$ is an extension of non-zero semistable objects of phases $>\psi$.
Since the triangle is part of a Harder--Narasimhan filtration, neither $x$ nor $a$ can be zero.
By the previous proposition and the presumed existence of spectral network representatives, this implies both $\phi>\psi$ and $\phi\leq \psi$, a contradiction.
\end{proof}

\subsection{The main conjecture}
\label{subsec_mainconj}

The following conjecture is part of a larger program suggested by Kontsevich (unpublished), c.f. the remarks in the introduction to this paper.
As before, $S$ is an oriented surface with area form and marked points $M\subset S$, and $\mc E$ is a schober of triangulated categories over $\mathbf k$ on $S$ without singular fibers.
We assume that the Fukaya category of $S$ with coefficients in $\mc E$, $\mc F(S,M;\mc E)$ is defined.
Furthermore, let $\sigma$ be a holomorphic family of stability conditions on $\mc E$, as defined in the previous subsection.
One should require, in the case of non-compact $S$, some convexity/completeness at infinity.
We will not attempt to formulate precise conditions, but note that the product case of quadratic differential (meromorphic or exponential type as in~\cite{hkk}) and constant stability condition should be a basic example satisfying these conditions.

\begin{conj}
\label{main_conj}
There is a stability condition on the homotopy category of the Fukaya category with coefficients $\mc F(S,M;\mc E)_{\mathbf K}$ with central charge given by \eqref{graph_Z} and semistable objects of phase $\phi$ those objects which have a spectral network representative of phase $\phi$.
\end{conj}

When $\mc E$ has fiber $\mathrm{Perf}(\mathbf k)$, the derived category of finite-dimensional chain complexes over $\mathbf k$, hence a holomorphic family of stability conditions on $\mc E$ is essentially a quadratic differential by the remarks in the previous subsection, the conjecture is proven in~\cite{hkk}. 
Note that in this case, spectral networks are just finite-length geodesics: saddle connections and geodesic loops. 
There is a similar phenomenon whenever the stability conditions on the fibers have central charge whose image is contained in a line in $\CC$, which will always be the case if the $K_0$-group of the fibers has rank one.
For instance, when the general fiber of $\mc E$ is the triangulated category generated by an objects $S$ with $\mathrm{Ext}^\bullet(S)=H^\bullet(S^2;\mathbf k)$ and $\mc E$ is allowed to have some singular fibers with vanishing cycles category $\mathrm{Perf}(\mathbf k)$ so that the monodromy of $\mc E$ around a singular fiber is the spherical twist along $S$.
Depending on whether $S$ is closed or not, the category $\mc F(S,M;\mc E)_{\mathbf K}$ should be quasi-equivalent to a quiver category as in~\cite{bs} or one of the categories defined \cite{h_3cyteich}. 
A form of Conjecture~\ref{main_conj} was then verified in these works, in the sense that it is shown that semistable objects correspond to finite length geodesics.

\begin{ex}
Let $S$ be a Riemann surface and $\widetilde S\subset T^*S$ holomorphic curve in the holomorphic cotangent bundle which is a branched covering over $S$ via the projection $T^*S\to S$.
Such $\widetilde S$ arise, for instance, as \textit{spectral curves} of Higgs bundles for the gauge group $\mathrm{SL}(n+1,\CC)$.
Suppose $p\in S$ is not a branch point and let 
\[
T^*_pS\cap\widetilde S=:\left\{z_0,\ldots,z_n\right\}
\]
be the fiber of $\widetilde S$ over $p$.
Given $q\in (T_S)^{\otimes 2}\setminus \{0\}$ there is a triangulated DG-category $\mc E_{p,q}$ with stability condition associated with the collection of $z_i$'s.
The category $\mc E_{p,q}$ is the (unique up to equivalence) 2CY category over $\mathbf k$ generated by an $A_n$ collection of spherical objects~\cite{thomas06,bridgeland_kleinian}.
Generically, the straight line from $z_i$ to $z_j$ does not pass through any other $z_k$'s and corresponds to a stable spherical object $S_{ij}$ (up to shift) with central charge $Z(S_{ij})=\pm(z_i-z_j)\sqrt{q}$ using the pairing of 1-forms and tangent vectors.

The categories $\mc E_{p,q}$ are the fibers of a schober $\mc E$ on $S$ with singular fibers at the branch points of $\widetilde S\to S$.
The stability conditions on the fibers combine to give a holomorphic family of stability conditions on $\mc E$.
(We have not defined what this means in the case of singular fibers, so strictly speaking the statement only makes sense on the complement of the branch locus.)
The global categories $\mc F(S;\mc E)_{\mathbf K}$ should recover the 3CY quiver categories considered in \cite{goncharov_webs,smith20}, at least for certain choices of $S$.
In the case $n=1$, the work~\cite{christ} recovers the 3CY quiver category from the schober using a definition of the Fukaya category with coefficients based on ribbon graphs and homotopy limits of DG-categories.

One can also formulate the conjectures about spectral networks in the sense of \cite{gmn} and stability conditions without the use of schobers, instead working directly with the open threefold fibered over $S$, c.f.~\cite{kerr_soibelman}.
\end{ex}

In the next section we will consider examples where the fiber category category is $\mc A_2$, the bounded derived category of the $A_2$-quiver, equipped with the standard (symmetric) stability condition. 
This is closely related to a special case of the above example where $n=2$ and $z_0,z_1,z_2$ form an equilateral triangle, except that the fiber category is not 2CY.
We verify Conjecture~\ref{main_conj} in the case when the base surface is a disk with three marked points on the boundary in Subsection~\ref{subsec_a2a2}, and the case of a disk with six marked points on the boundary in Section~\ref{sec_a5a2}.

\subsection{Uniqueness}
\label{subsec_uniqueness}

Spectral network representatives are typically not unique up to isomorphism, however we expect the following statement to hold, which implies in particular that the underlying graph is unique.

\begin{conj}
All spectral network representatives of a given semistable object in $V\left(\mc F(S,M;\mc E)_{\mathbf K}\right)$ have the same underlying support graph in $S$ with S-equivalent objects on the edges.
\end{conj}

Here, S-equivalence for semistable objects $X,Y$ of the same phase $\phi$ in the fiber category means that both have filtrations by semistable subobjects of the same phase which are equivalent in the sense of the Jordan--H\"older theorem.
We will prove the above conjecture in the special case of constant coefficients $\mc E$ and holomorphic family of stability conditions of product type, i.e. coming from a stability condition on $\mc E$ and a quadratic differential $\nu$ on $S$.
We expect the proof in more general cases to be similar.
In fact, we will prove a slightly stronger uniqueness statement which involves a certain abelian category of spectral networks of the same phase.

We consider the setup in which the category $\mc F(S,M;\mc E)$ has been rigorously defined (in Section~\ref{sec_aside}). 
Thus, $S$ is a Riemann surface with quadratic differential $\nu$, $M\subset \partial S$ a set of marked points, and $\mc E$ a triangulated DG-category over $\mathbf k$.
Furthermore, we fix a stability condition on $\mc E$ and a global phase $\phi\in\RR$.
Let $G$ be a graph in $S$ whose edges are straight lines, i.e. geodesics with respect to the flat metric $|\nu|$.

\begin{df}
Let $\mc F_G(S;\mc E)_\phi$ be the full subcategory of $V\left(\mc F_G(S;\mc E)_{\mathbf k}\right)$ of objects which are spectral networks of phase $\phi$, i.e. so that the object in $\mc E$ on an edge of $G$ with grading $\phi_1\in\RR$ is semistable of phase $\phi-\phi_1$.
\end{df}

Without loss of generality, choose the grading of the edges of $G$ in the range $[\phi-1,\phi)$, then for an object of $\mc F_G(S;\mc E)_\phi$ the semistable objects on the edges have phase in $(0,1]$, thus are by definition contained in the heart of the t-structure $\mc A\subset \mc E$ associated with the stability condition.
Let us fix an object of $\mc F_G(S;\mc E)_\phi$ and focus on a single vertex $p$ of $G$.
The horizontal foliation of $\nu$ partitions the set of half-edges at $p$ into two groups so that the phases of the objects on the half-edges, in clockwise order, are $\phi_1,\ldots,\phi_m,\psi_1,\ldots,\psi_n$ with
\[
0<\phi_1<\phi_2<\ldots<\phi_m\leq 1,\qquad 0<\psi_1<\psi_2<\ldots<\psi_n\leq 1
\]
and $m+n$ is the number of half-edges at $p$, see Figure~\ref{vertexrays}.
\begin{figure}
\centering
\begin{tikzpicture}
\draw[purple] (-3,0) -- (3,0);
\draw[thick] (0,0) -- (-1.8,1.3) node[anchor=south east] {$\phi_1$};
\draw[thick] (0,0) -- (-1.1,1.6) node[anchor=south] {$\phi_2$};
\node at (0,1) {$\cdots$};
\draw[thick] (0,0) -- (.9,1.6) node[anchor=south] {$\phi_{m-1}$};
\draw[thick] (0,0) -- (1.6,1.3) node[anchor=south west] {$\phi_m$};
\draw[thick] (0,0) -- (-1.5,-1.3) node[anchor=north east] {$\psi_n$};
\draw[thick] (0,0) -- (1.3,-1.6) node[anchor=north] {$\psi_2$};
\draw[thick] (0,0) -- (1.9,-1.3) node[anchor=north west] {$\psi_1$};
\node at (0,-1) {$\cdots$};
\end{tikzpicture}
\caption{Half-edges at a vertex of spectral network partitioned into two groups with phases of the objects on the edges in the range $(0,1]$, increasing in clockwise order.}
\label{vertexrays}
\end{figure}
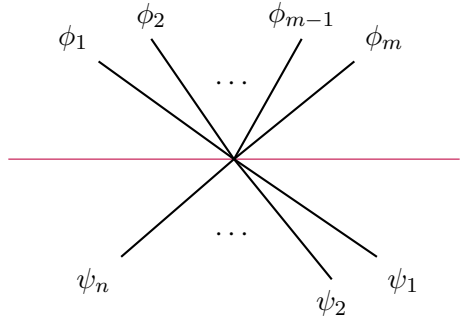
The stalk at $p$ of the sheaf of categories on $G$ whose global sections are $V\left(\mc F_G(S;\mc E)\right)$ is the category of representations of the quiver
\begin{equation}
\label{double_filt_quiver}
\bullet \longrightarrow \cdots \longrightarrow \bullet \longrightarrow \bullet \longleftarrow \bullet \longleftarrow \cdots \longleftarrow \bullet
\end{equation}
with $m-1$ arrows pointing one way and $n-1$ arrows pointing the other way.
The objects on the half-edges are then the cones over the morphisms in $\mc E$ assigned to the arrows, as well as the objects assigned to the first and last vertex.
If the object supported on $G$ is in fact in $\mc F_G(S;\mc E)_\phi$, i.e. a spectral network of phase $\phi$, then the object $E\in\mc E$ assigned to the sink vertex of \eqref{double_filt_quiver} is contained in the abelian category $\mc A=\mc E_{(0,1]}$ and the data of the representation of \eqref{double_filt_quiver} amounts to two \textit{anti-HN filtrations}, i.e. with semistable subquotients of strictly \textit{increasing} phase.

The above consideration for a single vertex lead to the following elementary description of the category $\mc F_G(S;\mc E)_\phi$ in terms of $\mc A$.
An object of $\mc F_G(S;\mc E)_\phi$ is given by an object $E_p\in \mc A$ for each vertex $p$ of $G$ together with a pair of anti-HN filtrations with subquotients of the correct phase as well as an isomorphism of subquotients for each edge of $G$.
Morphisms are required to preserve filtrations and intertwine the identifications of subquotients. 

\begin{prop}
\label{prop_fixedphase_abelian}
The category $\mc F_G(S;\mc E)_\phi$ is abelian.
\end{prop}

The key point is the following strictness property of maps between anti-HN filtrations.

\begin{lemma}
\label{lem_strictness}
Let $\mc C$ be a triangulated category with stability condition, $\mc A=\mc C_{(0,1]}\subset \mc C$ the associated heart.
Fix phases $0<\phi_1<\phi_2<\ldots<\phi_n\leq 1$ and let $A\in \mc A$ with filtration $0=F_0A\subset F_1A\subset \ldots \subset F_nA=A$ such that $F_iA/F_{i-1}A$ is semistable of phase $\phi_i$. 
Also let $B\in \mc A$ with filtration $F_iB$ of the same type.
Then any morphism $f:A\to B$ preserving the filtrations is strict, i.e.
\[
F_iB\cap \im(f)=f(F_iA)
\]
for $i=1,\ldots,n$.
\end{lemma}

\begin{proof}
The target $B$ of $f$ has a double filtration by subspaces $f(F_iA)\cap F_jB$ with associated subquotients
\[
G_{i,j}:=\left(f(F_iA)\cap F_jB\right)/\left((f(F_{i-1}A)\cap F_jB)+(f(F_iA)\cap F_jB)\right)
\]
and strictness of $f$ is equivalent to $G_{i,j}=0$ for $i>j$.
Suppose for contradiction that $f$ is not strict, then there are $r>s$ with $G_{r,s}\neq 0$ and $r-s>0$ minimal among such pairs.

We consider subquotients of $B$, arising from the double filtration, which are indicated schematically in Figure~\ref{fig_staircase}.
\begin{figure}
\centering
\begin{tikzpicture}
\fill[color=gray!30] (0,0) -- (0,.5) -- (.5,.5) -- (.5,1) -- (1,1) -- (1,1.5) -- (1.5,1.5) -- (1.5,2) -- (2,2) -- (2,2.5) -- (2.5,2.5) -- (2.5,3) -- (3,3) -- (3,3.5) -- (3.5,3.5) -- (3.5,4) -- (4,4) -- (4,3.5) -- (3.5,3.5) -- (3.5,3) -- (3,3) -- (3,2.5) -- (2.5,2.5) -- (2.5,2) -- (2,2) -- (2,1.5) -- (1.5,1.5) -- (1.5,1) -- (1,1) -- (1,.5) -- (.5,.5) -- (.5,0) -- (0,0);
\fill[color=gray!30] (1.5,0) -- (1.5,.5) -- (2,.5) -- (2,1) -- (2.5,1) -- (2.5,1.5) -- (3,1.5) -- (3,2) -- (3.5,2) -- (3.5,2.5) -- (4,2.5) -- (4,0) -- (1.5,0);
\fill[color=red,opacity=.5] (2,.5) rectangle +(2,.5);
\fill[color=green,opacity=.5] (2,0) rectangle +(.5,1);
\fill[color=yellow!80!black] (2,.5) rectangle +(.5,.5);
\draw[step=.5,gray,very thin] (0,0) grid (4,4);
\draw[color=green,thick] (2,0) rectangle +(.5,2);
\draw[color=red,thick] (1,.5) rectangle +(3,.5);
\draw (4.125,2) -- (4.25,2) -- (4.25,0);
\node[anchor=west] at (4.25,1.5) {$F_{r-1}B$};
\draw (4.625,1) -- (4.725,1) -- (4.725,0);
\node[anchor=west] at (4.725,.5) {$F_sB$};
\draw (2.5,-.125) -- (2.5,-.25) -- (0,-.25);
\node[anchor=north] at (2,-.25) {$f(F_rA)$};
\draw (1,-.625) -- (1,-.75) -- (0,-.75);
\node[anchor=north] at (.5,-.75) {$f(F_sA)$};
\end{tikzpicture}
\caption{The double filtration $f(F_iA)\cap F_jB$ and some of its subquotients. Boxes correspond to $G_{i,j}$'s. A white box indicates $G_{i,j}=0$}
\label{fig_staircase}
\end{figure}
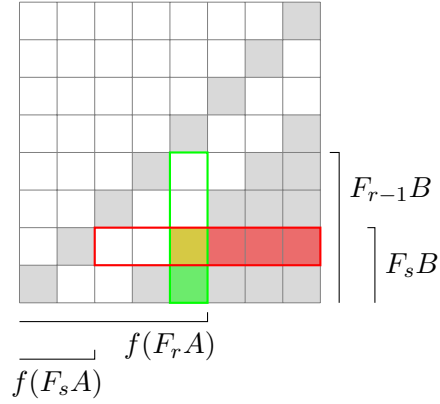
They are connected by a chain of maps as follows.
\begin{equation*}
\begin{tikzcd}
\left(F_rA\cap f^{-1}(F_{r-1}B)\right)/F_{r-1}A \arrow[d,two heads] & \\
\left(f(F_rA)\cap F_{r-1}B\right)/f(F_{r-1}A) \arrow[d,"\cong"] & F_sB/\left(F_{s-1}B+f(F_sA)\right) \arrow[d,"\cong"] \\
\left(f(F_rA)\cap F_sB\right)/\left(f(F_{r-1}A)\cap F_sB\right) \arrow[r] & F_sB/\left(F_{s-1}B+f(F_{r-1}A)\cap F_sB\right) \\
\end{tikzcd}
\end{equation*}
The object on the top left of the diagram is the kernel of the map $F_rA/F_{r-1}A\to F_rB/F_{r-1}B$, induced by $f$, between semistable objects of phase $\phi_r$, thus also semistable of phase $\phi_r$. (Recall that semistable objects of a given phase form an abelian category.) Similarly, the object on the top right of the diagram is the cokernel of the map $F_sA/F_{s-1}A\to F_sB/F_{s-1}B$, thus semistable of phase $\phi_s<\phi_r$.
The vertical map on the top-left is surjective.
The vertical map on the bottom-left and the vertical map on the right are isomorphisms since the relevant $G_{i,j}$'s vanish by our assumption on $(r,s)$.
Finally, the horizontal map at the bottom has image $G_{r,s}$, so is non-zero by assumption.
Following the diagram in counterclockwise direction thus gives a non-zero map from a semistable object to a semistable object of strictly smaller phase, a contradiction.
\end{proof}

\begin{proof}[Proof of Proposition~\ref{prop_fixedphase_abelian}]
Kernels in $\mc F_G(S;\mc E)_\phi$ are the usual kernels of linear maps together with the filtration obtained by intersection. 
Cokernels are the usual cokernels of linear maps with filtration obtained by taking the image under the quotient map.
For injective $f$, Lemma~\ref{lem_strictness} implies that the filtration on the source is equal to the one obtained from the target by taking the preimage under $f$, which implies that every mono is normal.
Similarly, for surjective $f$, Lemma~\ref{lem_strictness} implies that the filtration on the target is equal to the one obtained by taking the image under $f$, hence every epi is normal.
This shows that $\mc F_G(S;\mc E)_\phi$ is abelian.
\end{proof}

We borrow the following terminology from geometric invariant theory.
Suppose $\mc A$ is a finite--length abelian category, then $A,B\in \mc A$ are \textbf{S-equivalent} if there are filtrations 
\[
0=F_0A\subset F_1A\subset \ldots\subset F_nA=A,\qquad 0=F_iB\subset F_1B\subset \ldots\subset F_nB=B
\]
and a permutation $\sigma\in S_n$ such that
\[
F_iA/F_{i-1}A\cong F_{\sigma(i)}B/F_{\sigma(i)-1}B
\]
for $i=1,\ldots,n$.
Note that if $A$ and $B$ are S-equivalent, then $A=B$ is $K_0(\mc A)$. 
The converse follows from the fact that $K_0(\mc A)$ has a basis given by simple objects, as $\mc A$ is finite-length.

As before, we let $\mathbf R$ be the Novikov ring with residue field $\mathbf k$, which we assume to be algebraically closed, and field of fractions $\mathbf K$.

\begin{prop}
Let $\mc A$ be a curved $A_\infty$-category over $\mathbf R$ of the form $\mc A=\mathrm{Tw}(\mc A')$ such that $H^0\mc A_{\mathbf k}$ is finite-length admissible abelian.
If $X,Y\in \mc A$ are ismorphic in $\mc A_{\mathbf K}$, then $X_0,Y_0\in H^0\mc A_{\mathbf k}$ are S-equivalent.
\end{prop}

\begin{proof}
Since $H^0\mc A_{\mathbf k}$ is finite length, there is a filtration $0=F_0X_0\subset F_1X_0\subset\ldots F_nX_0=X_0$ with $S_i:=F_iX_0/F_{i-1}X_0$ simple.
Let $S:=\bigoplus_{i=1}^nS_i\in \mc A_{\mathbf k}$.
Using the assumption that $H^0\mc A_{\mathbf k}$ is admissible, we can find $\delta_0\in\mathrm{Hom}_{\mc A_{\mathbf k}}(S,S)$, strictly upper triangular, such that $X_0\cong (S,\delta_0)$ where $(S,\delta_0)$ is considered as one-sided twisted complex. 
By transfer of Maurer--Cartan elements (Lemma~\ref{lem_mctransport_nov}), we lift $(S,\delta_0)\in\mc A_{\mathbf k}$ to $(S,\delta)\in \mc A$ such that $X\cong (S,\delta)$.
Similarly, $Y\cong (T,\gamma)$, $T=\bigoplus_{i=1}^mT_i$.

Without loss of generality, we may assume that $\mc A$ is minimal in the sense that the differentials in $V(\mc A_\textbf{0})$ vanish, see~\cite{FOOO} where this is called the \textit{canonical model}.
We can also assume that if two of $S_1,\ldots,S_n,T_1,\ldots,T_m$ are isomorphic, then they are equal (otherwise replace objects in the list by isomorphic copies).
Let $U_1,\ldots,U_r$, $r\leq m+n$, be the pairwise non-isomorphic simple objects among the $S_i$'s and $T_i$'s.
As a consequence of the assumption that the $U_i$'s are simple, $\mathbf k$ is algebraically closed, and $H^0\mc A_{\mathbf k}$ is admissible it follows that
\[
\mathrm{Hom}_{\mc A}^0(U_i,U_j)=R^{\delta_{ij}},\qquad \mathrm{Hom}_{\mc A}^{<0}(U_i,U_j)=0
\]
where $R^{\delta_{ij}}=R$ if $i=j$ and $=0$ else.
In particular, $\Hom^0_{\mc A}(S\oplus T,S\oplus T)$ is a product of matrix algebras, and by strict unitality, the action on the bimodule $\Hom^1_{\mc A}(S\oplus T,S\oplus T)$ involves $\mk m_2$ only, i.e. is standard matrix multiplication.

By assumption, there is an isomorphism $f:(S,\delta)\to (T,\gamma)$.
Let $a_{ij}\in \Hom^0_{\mc A}(S_j,T_i)$ be its matrix coefficients.
Elementary row and column operations act on $\delta$, $\gamma$, and $f$, and can be used to diagonalize $(a_{ij})_{ij}$, i.e. so that $a_{ij}=\delta_{ij}t^{\lambda_i}$ for certain $\lambda_i\geq 0$.
These possibly permute the order of the $S_i$'s and $T_i$'s, so that $\delta_0$ and $\gamma_0$ are perhaps no longer strictly upper triangular, but this does not really matter.
In any case, we conclude that $m=n$ and $S_i\cong T_i$, so $X_0$ and $Y_0$ are S-equivalent.
\end{proof}

As a consequence, we obtain the following uniqueness statement for spectral network representatives.

\begin{theorem}
\label{thm_specnet_unique}
Let $X,X'\in V\left(\mc F(S,M;\mc E)\right)$ be spectral networks with $X_{\mathbf K}\cong X'_{\mathbf K}$, i.e. representing the same object in the Fukaya category $V\left(\mc F(S,M;\mc E)_{\mathbf K}\right)$. 
Then $X$ and $X'$ have the same support graph $G\subset S$ and $X_0$ and $X'_0$ are S-equivalent in $\mc F_G(S;\mc E)_\phi$.
\end{theorem}

%%%%%%%%%%%%%%%%%%%%%%%%%%%%%%%%%%%%%%%%%%%%%%%%%%%%%%%%%%%%%%%%%%%%%

\section{Examples of type $A_n\otimes A_2$}
\label{sec_ana2}

Following the general discussion in the previous section, we specialize in this section to cases where the base surface is a disk and/or the coefficient category is $\mc A_n$.
Much of this is preparation for Section~\ref{sec_a5a2}.
In Subsection~\ref{subsec_reptype} we discuss the representation type of the triangulated categories $\mc A_m\otimes \mc A_n$ and their interpretation as Fukaya categories with coefficients.
We fix a DG-model and stability condition for $\mc A_2$ in Subsection~\ref{subsec_coeffcat}.
Subsection~\ref{subsec_fukdisk} contains some preliminaries on the categories $\mc F(S,M;\mc E)$ in the case where $S$ is a disk with marked points $M$ on the boundary.
Finally, in Subsection~\ref{subsec_a2a2}, the main conjecture is checked in the case $A_2\otimes A_2$. 

\subsection{Representation type}
\label{subsec_reptype}

Some of the simplest examples of triangulated categories are $\mc A_n=D^b(A_n)$ where $A_n$ is the path algebra, over some field $\mathbf K$, of an $A_n$-type quiver, i.e. a quiver with underlying graph the $A_n$ Dynkin diagram.
As can be shown via reflection functors, the category $D^b(A_n)$ is independent of the orientation of the arrows.
By a combination of Gabriel's theorem and an argument for derived categories of hereditary abelian categories, it follows that $\mc A_n$ has only finitely many indecomposable objects up to isomorphism and shift.
The category $\mc A_n$ has a geometric interpretation as the Fukaya category of a disk with $n+1$ marked points on the boundary. Indecomposable objects then correspond to graded arcs in the disk, up to isotopy, connecting distinct marked points, see for example~\cite{hkk}.

More complicated examples of triangulated categories are obtained by taking tensor products.
Write $\mc A_m\otimes \mc A_n$ for the triangulated completion of the tensor product of DG-categories.
Then $\mc A_m\otimes \mc A_n\cong D^b(A_m\otimes A_n)$ where $A_m\otimes A_n$ is the tensor product of algebras and is the path algebra of a quiver with $(m-1)\times (n-1)$ squares and the relations that these squares commute. 
If $K$ is a Novikov field, then there is a quasi-equivalence
\[
\mc F(S,M;\mc A_n)_{\mathbf K}\cong \mc A_m\otimes \mc A_n
\]
where $S$ is a disk with $|M|=m+1$ marked points on the boundary, as will be shown in Subsection~\ref{subsec_fukdisk} below. 
They are thus basic examples of categories of the form $\mc F(S,M;\mc E)_{\mathbf K}$ where both the base $(S,M)$ and the fiber $\mc E$ are non-trivial, at least for $m,n\geq 2$.
It turns out that for small values of $m$ and $n$, $\mc A_m\otimes \mc A_n$ is significantly simpler than for general $m,n$, and the purpose of this subsection is to collect some facts in this direction.

A useful invariant of $\mc A_m\otimes \mc A_n$ is its Calabi--Yau dimension.
In general, a triangulated category $\mc C$ with Serre functor $S:\mc C\to \mc C$ is said to be \textbf{fractional Calabi--Yau} of dimension $p/q\in\QQ$ if $S^q\cong [p]$ for some $q\geq 1$, $p\in\ZZ$, see~\cite{keller_cy} for more details. 
We write this as $\dim_{\mathrm{fCY}}(\mc C)=p/q$.
For example 
\[
\dim_{\mathrm{fCY}}\left(\mc A_n\right)=1-\frac{2}{n+1}
\]
from which it follows by additivity that
\[
\dim_{\mathrm{fCY}}\left(\mc A_{n_1}\otimes\cdots\otimes \mc A_{n_k}\right)=k-\frac{2}{n_1+1}-\ldots-\frac{2}{n_k+1}.
\]
Without loss of generality, we assume that $n_i>1$ and $n_i\geq n_{i+1}$, then the above dimension is $<1$ for $\mc A_n$, $\mc A_2\otimes \mc A_2$, $\mc A_3\otimes \mc A_2$, and $\mc A_4\otimes \mc A_2$, equal to $1$ for $\mc A_2\otimes \mc A_2\otimes \mc A_2$, $\mc A_3\otimes \mc A_3$, and $\mc A_5\otimes \mc A_2$, and $>1$ in all other cases.

Let us look at the cases where $\dim_{\mathrm{fCY}}<1$ first.
It turns out, see~\cite{ladkani}, that there are derived Morita equivalences
\[
A_2\otimes A_2\sim D_4,\qquad A_3\otimes A_2\sim E_6,\qquad A_4\otimes A_2\sim E_8 
\]
where $D_4$ is the path algebra of some orientation of the $D_4$ Dynkin diagram, and similarly for $E_6$ and $E_8$.
By the already mentioned theorem of Gabriel it follows that all these cases are, like $A_n$, of finite representation type, i.e. there are only finitely many indecomposable objects up to isomorphism and shift in the derived category.
This makes it possible to verify Conjecture~\ref{main_conj} by hand, which will be demonstrated in the simplest case of $A_2\otimes A_2$ and for a particular stability condition in Subsection~\ref{subsec_a2a2}.

Next, we consider the cases where $\dim_{\mathrm{fCY}}=1$.
It turns out that not only is the derived category of tame representation type, i.e. there are at most one-parameter families of indecomposable objects, but it also has a geometric interpretation as the derived category of coherent sheaves on an orbifold $\PP^1$ which is a stacky quotient of an elliptic curve.
This equivalence was proven by Ueda in~\cite{ueda} and can be interpreted as an instance of Homological Mirror Symmetry.
\begin{center}
\begin{tabular}{c|c|c}
algebra & orbifold points on $\PP^1$  & elliptic curve quotient \\
\hline
$A_2\otimes A_2\otimes A_2$ & $\ZZ/3,\ZZ/3,\ZZ/3$ & $E/(\ZZ/3)$ \\
$A_3\otimes A_3$ & $\ZZ/4,\ZZ/4,\ZZ/2$ & $E/(\ZZ/4)$ \\
$A_5\otimes A_2$ & $\ZZ/6,\ZZ/3,\ZZ/2$ & $E/(\ZZ/6)$ 
\end{tabular}
\end{center}
In Section~\ref{sec_a5a2} we verify Conjecture~\ref{main_conj} in the case $A_5\otimes A_2$ where $\mc A_2$ plays the role of coefficient category and $\mc A_5$ is the Fukaya category of the base surface, a disk with six marked points on the boundary. 
We expect the other tame cases to be similar.
Furthermore, both the stability condition on the coefficient category and the geometry of the base surface are as symmetric as possible. 

For $\dim_{\mathrm{fCY}}>1$ the triangulated category has wild representation type, i.e. there are families of indecomposable objects of arbitrarily high dimension~\cite{Leszczynski1994}.
Hence, any direct verification of Conjecture~\ref{main_conj} by hand seems hopeless, and a different, more general approach is needed.

\subsection{Coefficient category}
\label{subsec_coeffcat}

The coefficient category in our main example is the bounded derived category of representations of the $A_2$ quiver over $\mathbf k$.
It will be convenient to use the following DG model. 
Start with the category $A_2$ with two objects $A,C$ with $\mathrm{End}(A)=\mathrm{End}(C)=\mathbf k$, $\Hom(A,C)=0$, and $\Hom(C,A)=\mathbf k[-1]$ (i.e. concentrated in degree 1).
Then let $\mc E$ be the category of one-sided twisted complexes over $A_2$.
Concretely, an object of $\mc E$ is of the form
\begin{equation}\label{a2_tw_cx}
\left((V\otimes C)\oplus (W\otimes A),\begin{bmatrix} d_V & 0 \\ f & d_W \end{bmatrix}\right)
\end{equation}
where $(V,d_V)$ and $(W,d_W)$ are $\ZZ$-graded finite-dimensional co-chain complexes over $\mathbf k$ and $f:V\to W$ is a chain map.
In particular we have the object $B:=(C\xrightarrow{1} A)$ which fits into an exact triangle $A\to B\to C\to A[1]$.

We choose the stability condition on $\mc E$ such that $A,B,C$ are semistable of phase $0,\frac{1}{3},\frac{2}{3}$, respectively, and unit length central charge.
Thus, the central charge of an object of the form \eqref{a2_tw_cx} is
\[
Z((V\otimes C)\oplus (W\otimes A))=\chi(V)e^{2\pi i/3}+\chi(W)
\]
where $\chi$ denotes the Euler characteristic, and semistable objects are of the form $E[m]^{\oplus n}$ where $E$ is one of $A,B,C$, $m\in\ZZ$, $n\in\ZZ_{\geq 0}$.
In fact any stability condition in the same $\CC$-orbit of $\mathrm{Stab}(\mc E)$ would work just as well --- we fix this one for concreteness.

\subsection{Fukaya category of the disk (with coefficients)}
\label{subsec_fukdisk}

We specialize to the case where the base surface $S$ is a disk with some number $n+1$ of marked points on the boundary.
More precisely, choose $\zeta\in\CC^*$ and let $S=S_\zeta\subset\CC$ be the closed disk of radius $|\zeta|$, $M=M_\zeta$ the orbit of $\zeta$ under the rotation by an angle of $2\pi/(n+1)$ around $0\in \CC$, $\nu=\zeta^2dz^2$ (which has ``horizontal'' foliation given by lines parallel to the one through $0$ and $\zeta$), and $\omega=dx\wedge dy$ the standard area form.
For the canonical choice $\zeta=1$, $M_\zeta$ is the set of $n+1$-st roots of unity, however it will be useful in our later arguments to have slightly more flexibility. 

Fix a coefficient DG-category $\mc E$.
We want to identify $\mc F(S_\zeta,M_\zeta;\mc E)_{\mathbf K}$ with the category of twisted complexes $\mathrm{Tw}((A_n\otimes\mc E)\otimes_{\mathbf k} \mathbf R)_{\mathbf K}$.
Assume first that $\zeta\in\RR_{>0}$.
Let $G\subset S$ be the graph with edges the segments between consecutive points in $M_\zeta$ except for the segment from $\zeta$ to $e^{2\pi i/(n+1)}\zeta$, i.e. all the sides except one of the regular $n+1$-gon with vertices $M_\zeta$.
The category $\mc A_G(S_\zeta)$ then has objects $X_1,\ldots,X_n$ and a basis of morphisms given by $x_i:X_{i-1}\to X_i$, $2\leq i\leq n$, and identity morphisms. 
We choose the (unique) grading on the edges of $G$ so that $|x_i|=1$ and the grading of $X_1$ is in the range $(0,1/2)$.
For example, for $n=5$ the grading on $X_i$ is $\frac{1}{3}-\frac{1}{3}(i-1)$.
By Proposition~\ref{prop_skelgen}, there is a quasi-equivalence
\[
\mc F_G(S_\zeta;\mc E)_{\mathbf K}\longrightarrow\mc F(S_\zeta,M_\zeta;\mc E)_{\mathbf K}
\]
and the category on the left is essentially by definition the category twisted complexes over $(A_n\otimes\mc E)\otimes_{\mathbf k}\mathbf R$, as there are no disk corrections.

For general $\zeta\in\CC^*$, the identification 
\[
\mc F(S_\zeta,M_\zeta;\mc E)\cong \mc F(S_{|\zeta|},M_{|\zeta|};\mc E)
\]
depends on a choice of $\mathrm{Arg}(\zeta)\in\RR$ which is used to identify gradings of graphs with endpoints in $M_\zeta$ with gradings of the corresponding rotated graph with endpoints in $M_{|\zeta|}$.

Fix for now $S=S_\zeta$ and $M=M_\zeta$.
We already know that for various $G\in\GG$ which are skeleta, i.e. \textit{trees} since $S$ is a disk, the categories $\mc F_G(S;\mc E)_{\mathbf K}$ are canonically quasi-equivalent.
For general $S$, there is no reason for this quasi-equivalence to be defined over $\mathbf k$, but in case of a disk, the equivalence has a simple description.

\begin{prop}
\label{prop_treecat}
Let $G,G'\in\GG$ be trees, then the following diagram of functors commutes up to natural isomorphism
\[
\begin{tikzcd}
\mc F_G(S;\mc E)_{\mathbf K} \arrow{dr} \arrow{rr} & & \mc F_{G'}(S;\mc E)_{\mathbf K}\arrow{dl} \\
& \mc F(S,M;\mc E)_{\mathbf K}
\end{tikzcd}
\]
where the horizontal arrow is the functor obtained by tensoring with $\mathbf K$ the canonical functor $\mc F_G(S;\mc E)_{\mathbf k}\to \mc F_{G'}(S;\mc E)_{\mathbf k}$ coming from edge contractions/expansions.
\end{prop}

\begin{proof}
First we can reduce to the case $\mc E=\mathrm{Perf}(\mathbf k)$, since the general case is obtained by tensoring everything by $\mc E$.
Then use the fact that autoequivalences of $\mc A_n$ very simple --- just rotations of the disk and shift --- so only the identity fixes the object which is an arc from one marked boundary point to the next.
Finally, everything is compatible with the restriction functor to $\mc E$ corresponding to any of the marked boundary points.
\end{proof}

Suppose $\zeta,\eta\in\CC^\times$ so that the straight line segment from $\zeta$ to $\eta$ does not intersect the $(n+1)$-gon corresponding to $\zeta$ and $M_\zeta$ is contained in the convex hull of $M_\eta$.
Then there is a functor
\begin{equation}
\label{fun_add_spokes}
\mc F(S_\zeta,M_\zeta;\mc E)\to \mc F(S_\eta,M_\eta;\mc E)
\end{equation}
which takes a graph $G$ and adds $n+1$ edges, from $e^{2\pi ik/(n+1)}\zeta$ to $e^{2\pi ik/(n+1)}\eta$ for $k=0,\ldots,n$.
The object of $\mc E$ on each new edge is the extension of the objects on the edges coming into the original marked boundary point. In fact, the functor is just the composition of $n+1$ edge expansion functors.

\begin{prop}
\label{prop_fun_spokes}
The functor \eqref{fun_add_spokes} is compatible with the identification of both categories, after tensoring with $\mathbf K$, with $\mathrm{Tw}((A_n\otimes\mc E)\otimes_{\mathbf k} \mathbf R)_{\mathbf K}$.
\end{prop}

\begin{proof}
Follows from Proposition~\ref{prop_treecat}.
\end{proof}

\subsection{Example: $A_2\otimes A_2$}
\label{subsec_a2a2}

We consider the simplest example where a spectral network with a vertex appears. 
Let $(S,M)$ be the unit disk with three marked points on the boundary the third roots of unity, and let $\mc E$ be the category $\mc A_2$ with the standard stability conditions as in Subsection~\ref{subsec_coeffcat}.
The Fukaya category $\mc F(S,M;\mc E)_{\mathbf K}$ is then quasi-equivalent to the derived category of the algebra $A_2\otimes A_2$ over $\mathbf K$.

We find the following 12 irreducible spectral networks with phase in $(0,1]$, depicted next to their central charge, see Figure~\ref{fig_a2a2Z}.
\begin{figure}
\centering
\begin{tikzpicture}
\draw[gray] (0,0) -- (5,0);
\draw[->] (0,0) -- (1.73,1);
\draw[gray] (2,.75) -- (2,1.25) -- (2.43,1);
\draw[thick] (2,.75) -- (2.43,1);
\node[below] at (2.21,.87) {$\scriptstyle A$};
\draw[gray] (3,.75) -- (3.43,1) -- (3,1.25);
\draw[thick] (3,.75) -- (3,1.25);
\node[left] at (3,1) {$\scriptstyle C$};
\draw[gray] (4,1.25) -- (4,.75) -- (4.43,1);
\draw[thick] (4,1.25) -- (4.43,1);
\node[above] at (4.21,1.12) {$\scriptstyle B$};
\draw[->] (0,0) -- (1.73,3);
\draw[gray] (2,3.25) -- (2.43,3.5) -- (2,3.75) -- cycle;
\draw[thick] (2.16,3.5) -- (2,3.25);
\node[below] at (1.9,3.25) {$\scriptstyle A$};
\draw[thick] (2.16,3.5) -- (2.43,3.5);
\node[right] at (2.43,3.5) {$\scriptstyle B$};
\draw[thick] (2.16,3.5) -- (2,3.75);
\node[above] at (1.9,3.75) {$\scriptstyle C$};
\draw[->] (0,0) -- (0,2);
\draw[gray] (-.13,2.25) -- (-.13,2.75) -- (.29,2.5);
\draw[thick] (-.13,2.25) -- (.29,2.5);
\node[below] at (.1,2.37) {$\scriptstyle B$};
\draw[gray] (-.13,3.75) -- (.29,3.5) -- (-.13,3.25);
\draw[thick] (-.13,3.75) -- (-.13,3.25);
\node[left] at (-.13,3.5) {$\scriptstyle A$};
\draw[gray] (-.13,4.75) -- (-.13,4.25) -- (.29,4.5);
\draw[thick] (-.13,4.75) -- (.29,4.5);
\node[above] at (.1,4.67) {$\scriptstyle C$};
\draw[->] (0,0) -- (-1.73,3);
\draw[gray] (-2.23,3.25) -- (-1.8,3.5) -- (-2.23,3.75) -- cycle;
\draw[thick] (-2.07,3.5) -- (-2.23,3.25);
\node[below] at (-2.33,3.25) {$\scriptstyle B$};
\draw[thick] (-2.07,3.5) -- (-1.8,3.5);
\node[right] at (-1.8,3.5) {$\scriptstyle C$};
\draw[thick] (-2.07,3.5) -- (-2.23,3.75);
\node[above] at (-2.33,3.75) {$\scriptstyle A$};
\draw[->] (0,0) -- (-1.73,1);
\draw[gray] (-4.43,1.25) -- (-4.43,.75) -- (-4,1);
\draw[thick] (-4.43,1.25) -- (-4,1);
\node[above] at (-4.21,1.12) {$\scriptstyle A$};
\draw[gray] (-3.43,.75) -- (-3,1) -- (-3.43,1.25);
\draw[thick] (-3.43,.75) -- (-3.43,1.25);
\node[left] at (-3.43,1) {$\scriptstyle B$};
\draw[gray] (-2.43,.75) -- (-2.43,1.25) -- (-2,1);
\draw[thick] (-2.43,.75) -- (-2,1);
\node[below] at (-2.21,.82) {$\scriptstyle C$};
\draw[->] (0,0) -- (-3.46,0);
\draw[gray] (-4.43,.25) -- (-4,0) -- (-4.43,-.25) -- cycle;
\draw[thick] (-4.29,0) -- (-4.43,.25);
\node[above] at (-4.53,.25) {$\scriptstyle B$};
\draw[thick] (-4.29,0) -- (-4,0);
\node[right] at (-4,0) {$\scriptstyle A$};
\draw[thick] (-4.29,0) -- (-4.43,-.25);
\node[below] at (-4.53,-.25) {$\scriptstyle C$};
\end{tikzpicture}
\caption{Irreducible spectral networks with phase in $(0,1]$ and their central charges.}
\label{fig_a2a2Z}
\end{figure}
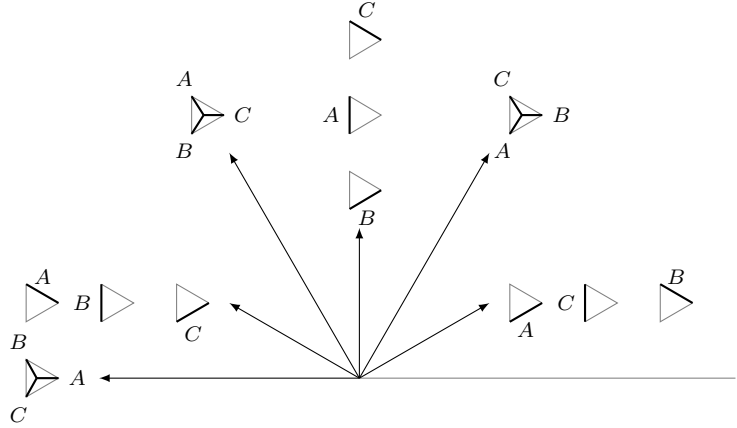
Looking at Figure~\ref{fig_a2a2Z}, we see that all central charges of spectral networks are of the form $me^{\pi i/6}-n\sqrt{3}$ for $m,n\in\ZZ_{\geq 0}$, where $e^{\pi i/6}$ is the central charge of the spectral networks representing the objects $A\to 0$, $0\to C$, and $B\to B$, and $-\sqrt{3}$ is the central charge of the spectral network representing the object $A[1]\to B[1]$.
Thus, we expect these four objects to be precisely the simple ones in the heart of the bounded t-structure associated with the desired stability condition.
The non-zero morphisms between the four objects are, besides multiples of the identity, all in degree one, with a basis corresponding to the three (long) arrows in the following quiver.
\[
\begin{tikzcd}
B\to B & & \\
 & A[1]\to B[1] \arrow[ul]\arrow[dl]\arrow[r] & A\to 0\\
0 \to C & & 
\end{tikzcd}
\]
In this way, one re-discovers the derived equivalence between $A_2\otimes A_2$ and $D_4$.
The point of this discussion is that it is now clear how to construct a suitable stability condition on $D^b(A_2\otimes A_2)$, i.e. one whose semistable objects are precisely those admitting a spectral network representative.
Namely, one uses the derived equivalence with $D^b(D_4)$ and its  natural bounded t-structure, together with the central charge as above.
This verifies the main conjecture in this example.

The same strategy works for $A_3\otimes A_2$ and $A_4\otimes A_2$, and also for other stability conditions on the coefficient category and/or less symmetric configurations of marked points in the base surface.
The reason is that the heart of the bounded t-structure of the desired stability condition is automatically finite-length, as the category has only finitely many indecomposable objects up isomorphism and shift.
Already for $A_5\otimes A_2$ it turns out the relevant abelian subcategory is not finite-length, so a different strategy is needed, see Section~\ref{sec_a5a2}.

%%%%%%%%%%%%%%%%%%%%%%%%%%%%%%%%%%%%%%%%%%%%%%%%%%%%%%%%%%%%%%%%%%%%%

\section{Example: $A_5\otimes A_2$}
\label{sec_a5a2}

This section contains a proof of our main result, the verification of the conjecture about spectral networks and stability conditions in the case where the base surface is a disk with marked points at the vertices of a regular hexagon and the coefficient category is $\mc A_2$ with the standard stability condition.
The first subsection contains an outline of the remainder of the section.
Some open problems and questions are formulated in Subsection~\ref{subsec_summary}.

\subsection{Outline}
\label{subsec_outline}

We consider the case where $S\subset C$ is a closed disk centered at the origin, $M\subset \partial S$ are the vertices of a regular hexagon, and $\mc E$ is the category $\mc A_2$ with the standard stability conditions as in Subsection~\ref{subsec_coeffcat}.
The Fukaya category $\mc F(S,M;\mc E)_{\mathbf K}$ is then quasi-equivalent to the derived category of the algebra $A_5\otimes A_2$ over $\mathbf K$, see Subsection~\ref{subsec_fukdisk}.
The goal of this section is to verify Conjecture~\ref{main_conj} in this case:

\begin{theorem}
\label{thm_a5a2}
For $S,M,\mc E$ as above, there is a stability condition on $\mc F(S,M;\mc E)_{\mathbf K}$ with central charge~\eqref{graph_Z} so that all semistable objects have spectral network representatives.
\end{theorem}

\begin{figure}
\centering
\includegraphics[scale=1]{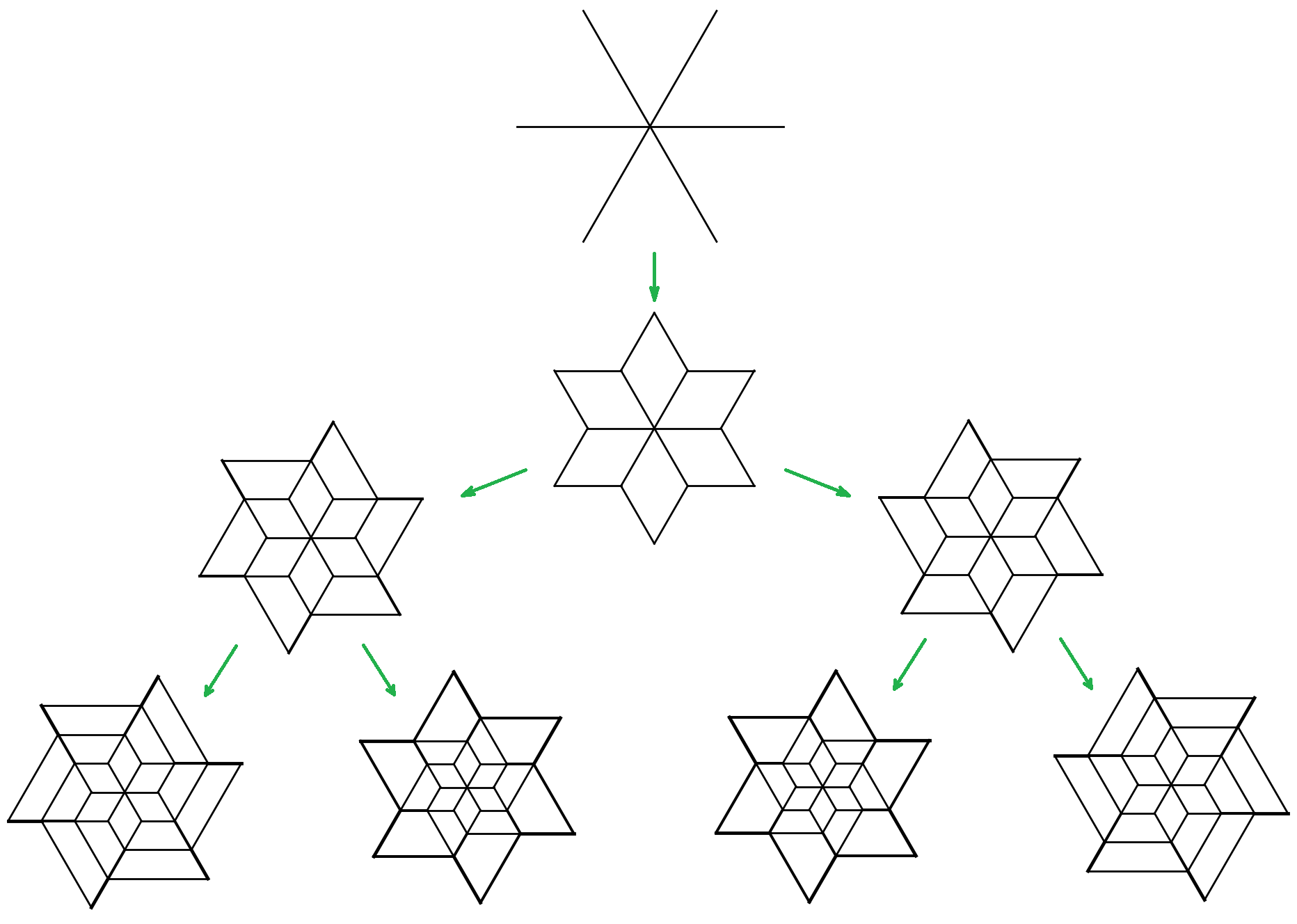}
\caption{Recursive construction of spectral networks.}
\label{fig_recursion}
\end{figure}

Let us outline the strategy of the proof.
\begin{enumerate}
    \item 
    The category $\mc A_5\otimes \mc A_2$ is identified with the derived category of coherent sheaves on the stacky quotient $E/(\ZZ/6)$ of an elliptic curve.
    The quotient is a $\PP^1$ with orbifold points with stabilizer groups $\ZZ/6$, $\ZZ/3$, and $\ZZ/2$.
    (Section~\ref{subsec_bside})
    \item
    There is an obvious stability condition on $D^b(E/(\ZZ/6))$ coming from slope stability of coherent sheaves. (Section~\ref{subsec_bside})
    \item
    Similar to the case of elliptic curves, there is a braid group action on the category $D^b(E/(\ZZ/6))$. The action is generated by twists along spherical adjunctions.
    Moreover, the action is transitive on slopes of semistable objects, hence the non-zero categories of semistable objects of fixed slopes are all equivalent.
    (Section~\ref{subsec_bside})
    \item
    The stability condition on $\mc F(S,M;\mc E)_{\mathbf K}$ is obtained by transfer along the equivalence with $D^b(E/(\ZZ/6))$.
    To verify the main conjecture, it remains to show that each stable objects has a spectral network representative.
    \item
    We show that stable objects of phase 0, which are the simple torsion sheaves on $E/(\ZZ/6)$, have spectral network representatives.
    (Section~\ref{subsec_phase0})
    \item
    Given a spectral network, one obtains two new ones by adding additional edges at the endpoints (see Figure~\ref{fig_recursion}). This operation lifts the twist along spherical adjunctions to the category $\mc F(S,M;\mc E)$.
    (Section~\ref{subsec_induction})
    \item Objects of phase contained in the sector $[0, 1/3)$ come from objects of phase $0$ by successive application
    of a twist $T_2$ or a dual-twist $T_1^{-1}$. By induction on the number of (dual-)twists in the sequence,
    the previous two steps combined show that all semistable objects with phase in $[0,1/3)$ have spectral network representatives.
    \item
    Applying autoequivalences of the coefficient category $\mc A_2$, which shift the phase by some integer multiple of $1/3$, one concludes that semistable objects of any phase have spectral network representatives. 
    This completes the proof.
\end{enumerate}

\subsection{B-side constructions}
\label{subsec_bside}

Recall from Subsection~\ref{subsec_reptype} that there is an equivalence of triangulated categories $D^b(A_5\otimes A_2)\cong D^b(E/(\ZZ/6))$ where $E$ is an elliptic curve with $\ZZ/6$-action and the quotient is considered as an orbifold/stacky $\PP^1$ with orbifold points $\ZZ/6,\ZZ/3,\ZZ/2$, a result of Ueda~\cite{ueda}.
The plan of this subsection is to fix an explicit equivalence between the two categories and to discuss the natural stability condition on $D^b(E/(\ZZ/6))$ coming from slope stability of coherent sheaves.

Fix an algebraically closed field $\mathbf K$ of characteristic zero --- we will take this to be a Novikov field later.
A concrete elliptic curve with $\ZZ/6$-action is the Fermat elliptic curve
\begin{equation*}
E:=\{x^3+y^3+z^3=0\}\subset\mathbb{P}^2.
\end{equation*}
The subgroup of $\mathrm{Aut}(E)$ which fixes the point $q:=(1,-1,0)\in E$ is cyclic of order 6, generated by $\sigma:E\to E$ with
\begin{equation*}
    \sigma(x,y,z):=(y,x,\omega z)
\end{equation*}
where $\omega$ is a primitive third root of unity.
This action of $\mathbb{Z}/6$ on $E$ has a single fixed point, $q$, an orbit of length 2 given by $(1,-\omega^{\pm 1},0)$, an orbit of length 3 given by $(1,1,-2^{1/3}\omega^k)$, $k=0,1,2$, and all other orbits are free.
Thus, the quotient orbifold 
\begin{equation*}
C:=E\big/(\mathbb Z/6)
\end{equation*}
is $\mathbb{P}^1$ with three stacky points with stabilizer groups $\mathbb Z/6,\mathbb Z/3,\mathbb Z/2$.
Make the identification $C\cong\mathbb{P}^1$ so that these points lie at $\mathbf{0,1,\infty}$ respectively.

The action of $\mathbb Z/6$ on $E$ naturally extends to an action on the bounded derived category of coherent sheaves on $E$, $D^b(E)$, and the quotient (orbit) category in the sense of~\cite{keller_orbit} is equivalent to $D^b(C)$, where $C$ is considered as a stack:
\begin{equation}\label{DbC_orbitcat}
    D^b(E)\big/(\mathbb Z/6)\cong D^b(C)
\end{equation}
see \cite{elagin}.
Thus, the category $D^b(C)$ looks a lot (up to a finite group action) like the derived category of an elliptic curve.
In fact, it would suffice for our purposes to \textit{define} $D^b(C)$ via~\eqref{DbC_orbitcat}.
Another approach which circumvents the general theory of stacks is to embed $C$ or more general $\PP^1$'s with orbifold points into a weighted projective space, see the influential work of Geigle--Lenzing~\cite{gl}.

\subsubsection*{Torsion sheaves and line bundles}

Let us review some basic facts about torsion sheaves and line bundles on $C$, as well as fix notation needed later.
We refer to~\cite{gl} for more details.

A simple sheaf $S\in\mathrm{Coh}(C)$ is necessarily supported at a single point $p\in C$. 
If $p$ is not an orbifold point, then $S=S_p$ is uniquely determined by $p$ up to isomorphism.
On the other hand, if $p$ is a $\mathbb Z/n$ orbifold point, then there are $n$ isomorphism classes of simple sheaves supported at $p$ corresponding to the characters of $\mathbb Z/n$.
More precisely, let $S_{p,1}$ be the skyscraper sheaf at $p$ with fiber the cotangent space $T^*_pC$ and its $\mathbb Z/n$ action, and let $S_{p,i}:=\left(S_{p,1}\right)^{\otimes i}$, $i\in\mathbb Z/n$ be its tensor powers.
Then $S_{p,0}$ is the skyscraper sheaf at $p$ with trivial action of $\mathbb Z/n$.
By Serre duality,
\begin{equation*}
    \mathrm{Ext}^1(S_{p,1},\mathcal O_C)\cong\mathrm{Hom}(\mathcal O_C,S_{p,0})\cong{\mathbf K}
\end{equation*}
and thus
\begin{equation*}
    \mathrm{Ext}^1(S_{p,i+1},S_{p,i})\cong{\mathbf K}
\end{equation*}
and these are the only non-zero $\mathrm{Hom}$'s and $\mathrm{Ext}^1$'s for pairs of non-isomorphic sheaves among $\mathcal O_C$, $S_{p,0}$, \ldots, $S_{p,n-1}$.

The Picard group $\mathrm{Pic}(C)$ is generated by $\mathcal O(p)$, $p\in\{\mathbf{0,1,\infty}\}$ and has relations
\begin{equation*}
    \mathcal O(6\cdot \mathbf 0)=\mathcal O(3\cdot \mathbf 1)=\mathcal O(2\cdot\mathbf \infty).
\end{equation*}
The dualizing sheaf is the line bundle
\begin{equation*}
    \omega_C=\mathcal O(5\cdot\mathbf 0-1\cdot\mathbf 1-1\cdot\mathbf\infty)
\end{equation*}
which generates the torsion subgroup in $\mathrm{Pic}(C)$.
The line bundles are related to the skyscraper sheaves via exact sequences
\begin{equation*}
    0\longrightarrow \mathcal O((k-1)\cdot p) \longrightarrow \mathcal O(k\cdot p) \longrightarrow S_{p,k} \longrightarrow 0.    
\end{equation*}
The direct sum of line bundles
\begin{equation*}
    \mathcal O\oplus\mathcal O(1\cdot\mathbf 0)\oplus\mathcal O(2\cdot\mathbf 0)\oplus\cdots \oplus\mathcal O(6\cdot\mathbf 0)\oplus\mathcal O(1\cdot\mathbf 1)\oplus\mathcal O(2\cdot\mathbf 1)\oplus\mathcal O(1\cdot\mathbf \infty)
\end{equation*}
is a formal Beilinson--type generator of $D^b(C)$ with corresponding quiver
\begin{equation}
\label{beilinson10}
\begin{tikzcd}[column sep=small]
 & \mathcal O(1\cdot\mathbf 0) \arrow[r] & \mathcal O(2\cdot\mathbf 0) \arrow[r] & \mathcal O(3\cdot\mathbf 0) \arrow[r] & \mathcal O(4\cdot\mathbf 0) \arrow[r] & \mathcal O(5\cdot\mathbf 0) \arrow[dr,start anchor=east,bend left] \\
\mathcal O \arrow[ur,end anchor=west,bend left]\arrow[rr]\arrow[drrr,end anchor=west,bend right=10] & & \mathcal O(1\cdot\mathbf 1) \arrow[rr] & & \mathcal O(2\cdot\mathbf 1) \arrow[rr] & & \mathcal O(6\cdot\mathbf 0) \\
 & & & \mathcal O(1\cdot\mathbf \infty) \arrow[urrr,start anchor=east,bend right=10]
\end{tikzcd}
\end{equation}
with relations that all three paths from the left-most vertex to the right-most vertex are equal.

\subsubsection*{Braid group action}

We begin with some generalities.
Let $\mathcal C$ and $\mathcal D$ be triangulated DG-categories, linear over a field $\mathbf K$, and $F:\mathcal D\to\mathcal C$ and $G:\mathcal C\to\mathcal D$ a pair of adjoint functors between them with counit $\varepsilon:FG\to 1_{\mathcal C}$ and unit $\eta:1_{\mathcal D}\to GF$.
Following Anno--Logvinenko~\cite{spherical}, the adjunction is called \textbf{spherical} if the endofunctors
\begin{equation*}
    T=\mathrm{Cone}(\varepsilon):\mathcal C\to\mathcal C,\qquad
    S=\mathrm{Cone}(\eta):\mathcal D\to\mathcal D
\end{equation*}
are autoequivalences, the \textit{twist} and \textit{cotwist}, respectively.
There is also a notion of \textit{dual twist} $T'$ which fits into an exact triangle 
\[
T'\longrightarrow 1_{\mc C}\longrightarrow FG'\longrightarrow T'[1]
\]
i.e. $T'$ is the cocone over the unit. For spherical adjunctions, $T'=T^{-1}$, see~\cite{spherical}.

Twists by spherical objects arise as a special case of this.
Namely, suppose that $\mathrm{Ext}^{\bullet}(X,Y)$ is  finite-dimensional for any pair of objects, and that $\mathcal C$ is a Calabi--Yau category of dimension $d$, i.e. $[d]$ is the Serre functor of $\mathcal C$.
Let $S\in\mathcal C$ be a spherical object, meaning that $\mathrm{Ext}^{\bullet}(S,S)=H^\bullet(S^d;\mathbf K)=\mathbf K\oplus\mathbf K[-d]$.
Under these assumptions, the hom-tensor adjunction with 
\begin{gather*}
    F:\mathrm{Perf}(\mathbf K)\to\mathcal C,\quad F(X):=X\otimes S \\
    G:\mathcal C\to\mathrm{Perf}(\mathbf K),\quad G(X)=\mathrm{Hom}^{\bullet}(S,X).
\end{gather*}
is a spherical adjunction and the autoequivalence $T_S:=\mathrm{Cone}(\varepsilon)$ is the \textbf{spherical twist} along $S$.
More explicitly, $T_S(X)$ fits into an exact triangle
\begin{equation*}
    \mathrm{Hom}^{\bullet}(S,X)\otimes S\to X\to T_S(X)\to\mathrm{Hom}^{\bullet+1}(S,X)\otimes S.
\end{equation*}
Its inverse, the dual twist, fits into an exact triangle
\[
\mathrm{Hom}^{\bullet-1}(X,S)\otimes S\to  {T_S}^{-1}(X)\to X\to \mathrm{Hom}^{\bullet}(X,S)\otimes S.
\]

We return to our example of the Fermat elliptic curve, $E$, with its $\mathbb Z/6$-action and fixed-point $q\in E$.
The category $D^b(E)$ has an abundance of spherical objects. 
We focus on the structure sheaf, $\mathcal O_E$, and the skyscraper sheaf at the fixed-point, $\mathcal O_q$.
The corresponding twist functors 
\begin{equation*}
    T_1:=T_{\mathcal O_q}=\mathcal O(q)\otimes\underline{\quad},\qquad T_2:=T_{\mathcal O_E}
\end{equation*}
satisfy the relations
\begin{equation*}
    T_1T_2T_1=T_2T_1T_2,\qquad (T_1T_2)^6=[2]
\end{equation*}
(see Seidel--Thomas~\cite{seidel_thomas}) and thus give an action of the group
\begin{equation*}
    \widetilde{SL(2,\mathbb Z)}:=\left\langle a,b,t\mid aba=bab,(ab)^6=t^2,t\text{ central}\right\rangle
\end{equation*}
which is a central extension of $SL(2,\mathbb Z)$,
on $D^b(E)$.

Since both $\mc O_q$ and $\mc O_E$ are fixed under the $\ZZ/6$-action, the $\widetilde{SL(2,\mathbb Z)}$-action commutes with the $\ZZ/6$-action.
Thus we obtain an induced $\widetilde{SL(2,\mathbb Z)}$-action on the orbit category $D^b(C)=D^b(E)/(\ZZ/6)$.
It turns out that the autoequivalences of $D^b(C)$ corresponding to $T_1$ and $T_2$ are not twists by spherical objects, but arise from spherical adjunctions between $D^b(C)$ and the category
\[
\mathrm{Perf}(\mathbf K)\big/(\ZZ/6)\cong\mathrm{Perf}\left(\mathbf K[\mathbb Z/6]\right)\cong\mathrm{Perf}(\mathbf K)^{6}
\]
which is the derived category of a set of six points (as a triangulated category).
The situation is summarized in the following diagram.
\[
\begin{tikzcd}[column sep=large,row sep=large]
\mathrm{Perf}(\mathbf K) \arrow[shift left,r,"\underline{\enspace}\otimes S"] \arrow[shift left,d] & D^b(E) \arrow[shift left,l,"\Hom^\bullet{(S,\underline{\enspace})}"] \arrow[shift left,d,"\pi_*"] \arrow[r,"T_S"] & D^b(E) \arrow[shift left,d,"\pi_*"] \\
\mathrm{Perf}(\mathbf K[\ZZ/6]) \arrow[shift left,u]\arrow[shift left,r] & D^b(C) \arrow[shift left,l]\arrow[shift left,u,"\pi^*"]\arrow[r] & D^b(C) \arrow[shift left,u,"\pi^*"] 
\end{tikzcd}
\]
Here, the second row is obtained from the first by passing to the $\ZZ/6$ orbit categories and the vertical arrows are induction/restriction functors.
The functor $\mathrm{Perf}(\mathbf K[\ZZ/6])\to D^b(C)$ sends a character $\chi:\ZZ/6\to \mathbf K^\times$ to the corresponding equivariant object $(S,\chi)\in D^b(C)$, which makes sense since $S$ is assumed to be fixed.
The corresponding twist autoequivalence is explicitly given by
\begin{equation}\label{twist6}
    T_{S_1,\ldots,S_6}(X):= \mathrm{Cone}\left(\bigoplus_{i\in\ZZ/6}\Hom^\bullet(S_i,X)\otimes S_i\to X\right)
\end{equation}
where 
\[
\left\{S_1,\ldots,S_6\right\}:=\left\{(S,\chi)\mid \chi:\ZZ/6\to\mathbf K^\times\right\}
\]
is the set of equivariant objects corresponding to $S$.
Note that the spherical adjunction is completely determined (up to equivalence) by the list of six objects, just as a spherical twist is determined by a single spherical object.
In fact, since the $S_i$'s are the indecomposable summands of $S_{1\ldots 6}:=S_1\oplus\cdots\oplus S_6$, the twist is determined by the single object $S_{1\ldots 6}$. 
We will occasionally refer to $T_{S_1,\ldots,S_6}$ as the \textit{twist by $S_{1\ldots 6}$} for brevity. 

In the particular case $S=\mathcal O_q$, the six different $\mathbb Z/6$-equivariant structures give objects
\begin{equation*}
    S_{0,0}, S_{0,1}, \ldots, S_{0,5}\in D^b(C)
\end{equation*}
i.e. the simple sheaves at the $\mathbb Z/6$ orbifold point, while the six equivariant structures on $\mathcal O_E$ give objects
\begin{equation*}
    O_C,\omega_C^{-1},\ldots,\omega_C^{-5}\in D^b(C)
\end{equation*}
where $\omega_C$ is the canonical bundle.
We write $T_1,T_2\in\mathrm{Aut}(D^b(C))$ for the corresponding autoequivalences.

\subsubsection*{Stability conditions}

While the categories $D^b(E)$ and $D^b(C)$ are in many respects quite similar, the space of stability conditions of $D^b(C)$ is significantly more complicated than that of $D^b(E)$.
For $D^b(E)$ one takes $\Gamma:=K_{\mathrm{num}}(D^b(E))\cong \ZZ^2$, the numerical Grothendieck group, which records rank and degree, and $\mathrm{cl}$ the quotient map from $K_0(D^b(E))$.
Thus, the space of stability conditions has $\dim_{\CC}=2$ and it was shown by Bridgeland~\cite{bridgeland07} that the action of $\widetilde{GL^+(2,\mathbb R)}$ is free and transitive.
All stability conditions have the same set of semistable objects (just with relabelled phases), which are shifts of torsion sheaves and slope-semistable vector bundles.

On the other hand 
\[
K_0(D^b(C))=K_{\mathrm{num}}(D^b(C))\cong \ZZ^{10}
\]
as follows from the existence of the Beilinson exceptional collection~\eqref{beilinson10}.
In particular, taking $\Gamma=K_0(D^b(C))$, we have
\[
\dim_{\CC}\mathrm{Stab}(D^b(C))=10.
\]
Furthermore, since $\ZZ/6$ acts trivially on $\mathrm{Stab}(D^b(E))$, there is a map
\begin{equation}
\label{stab_induction}
\mathrm{Stab}(D^b(E))\longrightarrow\mathrm{Stab}(D^b(C))
\end{equation}
which is a special case of a result of Macri--Mehrotra--Stellari~\cite{mms09}.
The map is defined so that, given a stability condition on $D^b(E)$, and object in $D^b(C)$ is semistable of a particular phase if the corresponding object in $D^b(E)$, obtained by forgetting the equivariant structure, is semistable of that same phase.

Here, we will not attempt to compute all of $\mathrm{Stab}(D^b(C))$, which seems to be a formidable task, but instead focus on a single stability condition which corresponds to our A-side geometry (regular hexagon in the base, equilateral triangle in the fiber) and is in the image of the map~\eqref{stab_induction}. 
\begin{df}
The \textbf{standard stability condition} on $D^b(C)$ has heart $\mathrm{Coh}(C)[1]\subset D^b(C)$ and central charge
\begin{equation}\label{Z_std}
    Z(E):=\mathrm{deg}(E)-e^{\pi i/3}\mathrm{rk}(E).
\end{equation}
\end{df}
Here $\deg$ is normalized so that $\deg S_{\mathbf{0},k}=1$, $\deg S_{\mathbf{1},k}=2$, and $\deg S_{\mathbf{\infty},k}=3$.
We introduce the shift $[1]$ above for convenience so that torsion sheaves have phase $\phi=0$ instead of $\phi=1$.
From now on, unless explicitly stated otherwise, we assume that $D^b(C)$ is equipped with the standard stability condition.

\begin{prop}
\label{prop63}
Any semistable object of $D^b(C)$ is obtained as the image of a semistable object of phase 0 (torsion sheaf) under the action of $\widetilde{SL(2,\ZZ)}$ on $D^b(C)$.
\end{prop}

\begin{proof}
The corresponding statement is well-known for $D^b(E)$, where it follows from the fact that any indecomposable object is semistable, thus any autoequivalence preserves semistable objects, only changing their phases.
Since the standard stability condition on $D^b(C)$ is induced from the stability condition on $D^b(E)$ with heart $\mathrm{Coh}(E)[1]$ and central charge defined as in~\eqref{Z_std}, and the $\widetilde{SL(2,\ZZ)}$-action descends to $D^b(C)$, the statement also holds for $D^b(C)$.
Alternatively, one simply uses the fact that indecomposable implies slope semistable also in $D^b(C)$ as shown in~\cite{gl}.
\end{proof}

We will need a slightly different version of the above Proposition where we only act by the submonoid in $\mathrm{Aut}(D^b(C))$ generated by $T_1^{-1}$ and $T_2$ (defined above).

\begin{prop}
\label{prop_t1t2slopes}
Any semistable object of $D^b(C)$ with phase in the interval $[0,1/3)$ is obtained from a semistable object of phase $0$ by applying the functors $T_1^{-1}$ and $T_2$ some number of times.
\end{prop}

\begin{proof}
We see from~\eqref{Z_std} that the possible phases of semistable objects are of the form
\[
\phi_{m,n}=\frac{1}{\pi}\mathrm{Arg}\left(m+e^{\pi i/3}n\right)
\]
for $(m,n)\neq (0,0)$, which we can choose to be primitive, and all branches of $\mathrm{Arg}$.
The case $\phi_{m,n}\in [0,1/3)$ corresponds to $(m,n)\in\ZZ_{>0}\times\ZZ_{\geq 0}$.
The dual twist $T_1^{-1}$ sends objects with $(\mathrm{deg},-\mathrm{rk})=(m,n)$ to $(m+n,n)$ and the twist $T_2$ sends them to $(m,n+m)$. 
(This is easy to see for the corresponding twists on $D^b(E)$ and thus follows for $D^b(C)=D^b(E/(\ZZ/6))$ by construction.)
By Euclid's algorithm we get, starting from $(1,0)$, any primitive $(m,n)$ with $m>0$, $n\geq 0$, thus all $\phi_{m,n}$ in the range $[0,1/3)$.
\end{proof}

\subsubsection*{Equivalence with $\mc A_5\otimes \mc A_2$}

We need to fix an equivalence of the two categories in order to transfer stability conditions from $D^b(C)$ to $D^b(A_5\otimes A_2)$.
This will give (an outline of) an alternative proof of the theorem of Ueda mentioned in Subsection~\ref{subsec_reptype}.
It suffices to give a generator in $D^b(A_5\otimes A_2)$ whose endomorphism algebra is the path algebra of the Beilinson-type quiver for $C$, in particular concentrated in degree 0, thus formal.

The $A_2$ quiver has three indecomposable representations, $A,B,C$ which fit into a short exact sequence $0\to A\to B\to C\to 0$.
It will be convenient to describe representations of $A_2\otimes A_5$ by listing five objects of $D^b(A_2)$ and maps between them.

\begin{center}
\begin{tabular}{|c|c|}
    \hline
    $\mathrm{Coh}(C)$ & $A_2\otimes A_5$ \\ 
    \hline\hline
    $S_{\mathbf 0,0}$ & $B\to B\to B\to B\to B$ \\ \hline
    $S_{\mathbf 0,1}$ & $0\to 0\to 0\to 0\to A[1]$ \\ \hline
    $S_{\mathbf 0,2}$ & $0\to 0\to 0\to C\to 0$ \\ \hline
    $S_{\mathbf 0,3}$ & $0\to 0\to B\to 0\to 0$ \\ \hline
    $S_{\mathbf 0,4}$ & $0\to A\to 0\to 0\to 0$ \\ \hline
    $S_{\mathbf 0,5}$ & $C[-1]\to 0\to 0\to 0\to 0$ \\ \hline\hline
    $S_{\mathbf 1,0}$ & $0\to 0\to C\to C\to C$ \\ \hline
    $S_{\mathbf 1,1}$ & $0\to B\to B\to B\to 0$ \\ \hline
    $S_{\mathbf 1,2}$ & $A\to A\to A\to 0\to 0$ \\ \hline\hline
    $S_{\mathbf \infty,0}$ & $0\to B\to B\to C\to C$ \\ \hline
    $S_{\mathbf \infty,1}$ & $A\to A\to B\to B\to 0$ \\ \hline
\end{tabular}
\end{center}

\begin{center}
\begin{tabular}{|c|c|}
    \hline
    $\mathrm{Coh}(C)$ & $A_2\otimes A_5$ \\ 
    \hline \hline
    $\mathcal O$ & $0 \to 0 \to 0 \to 0 \to B$ \\ \hline
    $\mathcal O(1\cdot\mathbf 0)$ & $0 \to 0 \to 0 \to 0 \to C$ \\ \hline
    $\mathcal O(2\cdot\mathbf 0)$ & $0 \to 0 \to 0 \to C \to C$ \\ \hline
    $\mathcal O(3\cdot\mathbf 0)$ & $0 \to 0 \to B \to C \to C$ \\ \hline
    $\mathcal O(4\cdot\mathbf 0)$ & $0 \to A \to B \to C \to C$ \\ \hline
    $\mathcal O(5\cdot\mathbf 0)$ & $C[-1] \to A \to B \to C \to C$ \\ \hline
    $\mathcal O(1\cdot\mathbf 1)$ & $0 \to B \to B \to B \to B$ \\ \hline
    $\mathcal O(2\cdot\mathbf 1)$ & $A \to A\oplus B \to A\oplus B \to B \to B$ \\ \hline
    $\mathcal O(1\cdot\mathbf \infty)$ & $A \to A \to B \to B \to B$ \\ \hline
    $\mathcal O(6\cdot\mathbf 0)$ & $A \to A\oplus B \to B\oplus B \to B\oplus C \to B\oplus C$ \\ \hline
    
\end{tabular}
\end{center}

The 1-parameter family of skyscraper sheaves $S_p$ corresponds to a family of representations of the form
\[
A\to A\oplus B\to B\oplus B\to B\oplus C\to C.
\]
To see the parameter note that the image of the maps from each copy of $A$ and the kernel of the maps to each copy of $C$ determine a quadruple of lines in $B\oplus B$, so we get four points on a projective line, which have a single invariant --- their cross-ratio.
For concreteness, we can take the representations of the form
\renewcommand{\arraystretch}{.5}
\[
\begin{tikzcd}[column sep=huge,row sep=huge,ampersand replacement=\&]
0 \arrow[r]\arrow[d] \& \mathbf K \arrow[r,"\begin{bmatrix} 0 \\ 1 \end{bmatrix}"]\arrow[d,"\begin{bmatrix} 0 \\ 1 \end{bmatrix}"] \& \mathbf K^2 \arrow[r,"{\begin{bmatrix} 1 & 0 \\ 0 & 1 \end{bmatrix}}"]\arrow[d,"{\begin{bmatrix} 1 & 0 \\ 0 & 1 \end{bmatrix}}"] \& \mathbf K^2 \arrow[r,"{\begin{bmatrix} 0 & 1 \end{bmatrix}}"]\arrow[d,"{\begin{bmatrix} 1 & -1 \end{bmatrix}}"] \& \mathbf K \arrow[d] \\
\mathbf K \arrow[r,"\begin{bmatrix} x \\ y \end{bmatrix}"'] \& \mathbf K^2 \arrow[r,"{\begin{bmatrix} 1 & 0 \\ 0 & 1 \end{bmatrix}}"'] \& \mathbf K^2 \arrow[r,"{\begin{bmatrix} 1 & -1 \end{bmatrix}}"'] \& \mathbf K \arrow[r] \& 0 \\
\end{tikzcd}
\]
where $p=x/y\in\mathbb{P}^1_{\mathbf{K}}$.
For $p\in\{0,1,\infty\}$ the above representation corresponds to an extension of the simple torsion sheaves at the orbifold point of order $6,3,2$ respectively. (This agrees with our previous convention for the coordinate on the orbifold.)

\subsection{Spectral networks of phase 0}
\label{subsec_phase0}

From now on we assume $n=5$, i.e. $|M_\zeta|=6$, and that $\mc E=\mc A_2$.
Equip the triangulated category
\[
\mc F\left(S_\zeta,M_\zeta;\mc E\right)\cong \mc A_5\otimes \mc A_2\cong D^b\left(E/(\ZZ/6)\right)
\]
with the stability condition constructed on the B-side.
Recall that for general $\zeta\in\CC^\times$, the equivalence between categories, and thus the stability condition, depend on a choice of $\mathrm{Arg}(\zeta)$. For $\zeta\in\RR_{>0}$ we choose $\mathrm{Arg}(\zeta)=0$.

\begin{prop}
Every semistable object of phase $0$ in $\mc F(S_\zeta,M_\zeta;\mc E)$ has a spectral network representative.
\end{prop}

\begin{proof}
Without loss of generality, we can assume $\zeta\in\RR_{>0}$ --- the general case follows by rotation.
It suffices to show that stable objects have spectral network representatives.
The stable objects of phase 0 in $D^b\left(E/(\ZZ/6)\right)$ are simple torsion sheaves and the corresponding objects in $D^b\left(A_5\otimes A_2\right)$ were described in the previous subsection.

The spectral networks for the $6+3+2$ simple torsion sheaves at the orbifold points are as follows.
\[
\tikz{\hexbg (0,-1) \node[below] at (0,-2) {$B\to B\to B\to B\to B$}; \draw[thick] (2,0) to (1,1.73); \node[anchor=north east] at (1.5,.87) {$B$};} \qquad
\tikz{\hexbg (0,-1) \node[below] at (0,-2) {$0\to 0\to 0\to 0\to A[1]$}; \draw[thick] (-1,1.73) to (1,1.73); \node[anchor=north] at (0,1.73) {$A$};} \qquad
\tikz{\hexbg (0,-1) \node[below] at (0,-2) {$0\to 0\to 0\to C\to 0$}; \draw[thick] (-2,0) to (-1,1.73); \node[anchor=north west] at (-1.5,.87) {$C$};} 
\]
\[
\tikz{\hexbg (0,-1) \node[below] at (0,-2) {$0\to 0\to B\to 0\to 0$}; \draw[thick] (-2,0) to (-1,-1.73); \node[anchor=south west] at (-1.5,-.87) {$B$};} \qquad
\tikz{\hexbg (0,-1) \node[below] at (0,-2) {$0\to A\to 0\to 0\to 0$}; \draw[thick] (-1,-1.73) to (1,-1.73); \node[anchor=south] at (0,-1.73) {$A$};} \qquad
\tikz{\hexbg (0,-1) \node[below] at (0,-2) {$C[-1]\to 0\to 0\to 0\to 0$}; \draw[thick] (2,0) to (1,-1.73); \node[anchor=south east] at (1.5,-.87) {$C$};} 
\]
\[
\tikz{\hexbg (0,-1) \node[below] at (0,-2) {$0\to 0\to C\to C\to C$}; \draw[thick] (1,1.73) to (-1,-1.73); \node[anchor=south east] at (0,0) {$C$};} \qquad
\tikz{\hexbg (0,-1) \node[below] at (0,-2) {$0\to B\to B\to B\to 0$}; \draw[thick] (1,-1.73) to (-1,1.73); \node[anchor=north east] at (0,0) {$B$};} \qquad
\tikz{\hexbg (0,-1) \node[below] at (0,-2) {$A\to A\to A\to 0\to 0$}; \draw[thick] (-2,0) to (2,0); \node[anchor=north] at (0,0) {$A$};} 
\]
\[
\tikz{\hexbg (0,-1) \node[below] at (0,-2) {$0\to B\to B\to C\to C$}; \draw[thick] (1,1.73) to (0,0) to (1,-1.73); \draw[thick] (0,0) to (-2,0); \node[anchor=north west] at (.5,.87) {$C$}; \node[anchor=north east] at (.5,-.87) {$B$}; \node[anchor=south] at (-1,0) {$A$};} \qquad
\tikz{\hexbg (0,-1) \node[below] at (0,-2) {$A\to A\to B\to B\to 0$}; \draw[thick] (-1,1.73) to (0,0) to (-1,-1.73); \draw[thick] (2,0) to (0,0); \node[anchor=north] at (1,0) {$A$}; \node[anchor=south west] at (-.5,.87) {$B$}; \node[anchor=south east] at (-.5,-.87) {$C$};}
\]
In each case one must show that the spectral network object is isomorphic to the object of $D^b(A_5\otimes A_2)$ written below it.
This follows easily from Proposition~\ref{prop_treecat}, as all these spectral networks are supported on single edges or trees.

Next, we look at simple torsion sheaves away from the orbifold points.
Recall from Subsection~\ref{subsec_bside} that the object of $D^b(A_5\otimes A_2)$ corresponding to $\mc O_p$ is 
\begin{equation}
\label{genpt_a5a2}
\begin{tikzcd}[column sep=huge,row sep=huge,ampersand replacement=\&]
0 \arrow[r]\arrow[d] \& \mathbf K \arrow[r,"\begin{bmatrix} 0 \\ 1 \end{bmatrix}"]\arrow[d,"\begin{bmatrix} 0 \\ 1 \end{bmatrix}"] \& \mathbf K^2 \arrow[r,"{\begin{bmatrix} 1 & 0 \\ 0 & 1 \end{bmatrix}}"]\arrow[d,"{\begin{bmatrix} 1 & 0 \\ 0 & 1 \end{bmatrix}}"] \& \mathbf K^2 \arrow[r,"{\begin{bmatrix} 0 & 1 \end{bmatrix}}"]\arrow[d,"{\begin{bmatrix} 1 & -1 \end{bmatrix}}"] \& \mathbf K \arrow[d] \\
\mathbf K \arrow[r,"\begin{bmatrix} x \\ y \end{bmatrix}"'] \& \mathbf K^2 \arrow[r,"{\begin{bmatrix} 1 & 0 \\ 0 & 1 \end{bmatrix}}"'] \& \mathbf K^2 \arrow[r,"{\begin{bmatrix} 1 & -1 \end{bmatrix}}"'] \& \mathbf K \arrow[r] \& 0 \\
\end{tikzcd}
\end{equation}
where $p=x/y\in\mathbb{P}^1_{\mathbf{K}}\setminus \{0,1,\infty\}$.
Let $\alpha>0$ be the area of the hexagon with vertices $M_\zeta$.
We divide $\PP_{\mathbf K}^1\setminus \{0,1,\infty\}=\mathbf K^\times\setminus \{1\}$ into three regions.
\begin{align*}
R_1&:=\left\{a\in \mathbf K^\times\setminus \{1\}\mid \mathrm{val}(a)\in (-\infty,0]\right\} \\
R_2&:=\left\{a\in \mathbf K^\times\setminus \{1\}\mid \mathrm{val}(a)\in (0,\alpha)\right\} \\
R_3&:=\left\{a\in \mathbf K^\times\setminus \{1\}\mid \mathrm{val}(a)\in [\alpha,+\infty)\right\} 
\end{align*}
It turns out that the topology of the corresponding spectral network is different for each of the three regions, so we will consider them separately.

\underline{\textit{Region} $R_1$}.
For $x/y\in R_1$ we can normalize to $x=1$, then $y\in \mathbf R$, so the representation \eqref{genpt_a5a2} is defined over $\mathbf R$ and thus corresponds to some object supported on $G$ where $G$ is any skeleton for $(S_1,M_1)$.
In particular we can choose $G$ with a single 6-valent vertex.
We claim that the objects on the spokes (edges of the tree), which depend only on the reduction of \eqref{genpt_a5a2} modulo $t^{>0}$, are as indicated in the following figure.
\[
\tikz{\hexbg (0,-1) \node[below] at (0,-2) {$\mc O_p$, $p\in R_1$}; \draw[thick] (1,1.73) to (-1,-1.73); \draw[thick] (-1,1.73) to (1,-1.73); \draw[thick] (2,0) to (-2,0); \node[anchor=north west] at (.5,.87) {$C$}; \node[anchor=north east] at (.5,-.87) {$B$}; \node[anchor=south] at (-1,0) {$A$}; \node[anchor=north] at (1,0) {$A$}; \node[anchor=south west] at (-.5,.87) {$B$}; \node[anchor=south east] at (-.5,-.87) {$C$};} 
\]
In particular, they are semistable of the correct phase, which implies that we get a spectral network.

We look at spokes in counter-clockwise order, starting with the one pointing to the right.
The object on the first spoke is the first object in the sequence \eqref{genpt_a5a2}, which is $A$.
The object on the second spoke is the cone over the reduction mod $t^{>0}$ of the first map in the sequence, which is $B=(C\xrightarrow{1} A)$, as the image in $\mathbf k^2$ of the first horizontal arrow in the lower row is different from the image of the second vertical arrow.
Similar, looking at the cones over the remaining three maps and the fifth object in the sequence we get the objects $C,A,B,C$ as claimed.

\underline{\textit{Region} $R_2$}.
For $x/y\in R_2$ we claim that the object \eqref{genpt_a5a2} has a spectral network of the following form.
\[
\tikz{\hexbg (0,-1) \node[below] at (0,-2) {$\mc O_p$, $p\in R_2$}; \draw[thick] (1,0) to (.5,.87) to (-.5,.87) to (-1,0) to (-.5,-.87) to (.5,-.87) to cycle; \draw[thick] (1,0) -- (2,0); \draw[thick] (.5,.87) -- (1,1.73); \draw[thick] (-.5,.87) -- (-1,1.73); \draw[thick] (-1,0) -- (-2,0); \draw[thick] (-.5,-.87) -- (-1,-1.73); \draw[thick] (.5,-.87) -- (1,-1.73); \node[anchor=north] at (0,.87) {$A$}; \node[anchor=south] at (0,-.87) {$A$}; \node[anchor=east] at (.75,-.43) {$C$}; \node[anchor=west] at (-.75,.43) {$C$}; \node[anchor=east] at (.75,.43) {$B$}; \node[anchor=west] at (-.75,-.43) {$B$}; 
\node[anchor=west] at (.75,1.3) {$C$}; \node[anchor=east] at (.75,-1.3) {$B$}; \node[anchor=south] at (-1.5,0) {$A$}; \node[anchor=north] at (1.5,0) {$A$}; \node[anchor=west] at (-.75,1.3) {$B$}; \node[anchor=east] at (-.75,-1.3) {$C$};} 
\]
This reduces to the proof in the case $x/y\in R_3$ (see below) by rescaling the size of the outer hexagon so that its area becomes $\mathrm{val}(x/y)$ and applying Proposition~\ref{prop_fun_spokes}.

\underline{\textit{Region} $R_3$}.
For $x/y\in R_3$ we normalize to $y=1$, then $x/t^{\alpha}\in\mathbf R$.
We claim that the object \eqref{genpt_a5a2} is isomorphic to the object
\[
\tikz{\hexbg (0,-1) \node[below] at (0,-2) {$\mc O_p$, $p\in R_3$}; \draw[thick] (2,0) to (1,1.73) to (-1,1.73) to (-2,0) to (-1,-1.73) to (1,-1.73) to cycle; \node[anchor=north] at (0,1.73) {$A$}; \node[anchor=south] at (0,-1.73) {$A$}; \node[anchor=south east] at (1.5,-.87) {$C$}; \node[anchor=north west] at (-1.5,.87) {$C$}; \node[anchor=north east] at (1.5,.87) {$B$}; \node[anchor=south west] at (-1.5,-.87) {$B$};}
\]
with \textit{total monodromy} $x/t^\alpha$. 
This means the following: The Maurer--Cartan element for a graph with objects on edges as above is any $\mathbf R$-linear combination of the six boundary paths corresponding to the six vertices of the hexagon. The product of the six scalars is by definition the total monodromy.

To prove the claim, we describe an explicit isomorphism in $\mc F(S_\zeta,M_\zeta;\mc E)$.
\[
\begin{tikzcd}[column sep=huge,row sep=huge,ampersand replacement=\&]
C \arrow[d,"1"] \arrow[r,"1"] \& C[-1] \arrow[dr,"x/t^{\alpha}"] \arrow[ddd,"1",orange] \arrow[ddr,swap,pos=.4,"x-1",orange]\& \& C \arrow[r,"1"]\arrow[d,"1"]  \arrow[dd,"t^\alpha{\begin{bmatrix} x-1 \\ 0 \end{bmatrix}}",orange,bend left] \& C \arrow[dr,"1"] \arrow[dd,"t^\alpha{\begin{bmatrix} x-1 \\ 0 \end{bmatrix}}",orange] \\
A \arrow[ddr,"1",orange] \& \& A \arrow[r,"1"] \arrow[dd,swap,pos=.5,"t^\alpha{\begin{bmatrix} 1 \\ 1 \end{bmatrix}}",orange,bend right] \& A \arrow[dd,swap,pos=.65,"t^\alpha{\begin{bmatrix} 1 \\ 1 \end{bmatrix}}",orange,bend right] \& \& A \arrow[lllll,bend left=10,"1"] \\
 \&  \& C \arrow[r,"\begin{bmatrix} 0 \\ 1 \end{bmatrix}"]\arrow[d,"\begin{bmatrix} 0 \\ 1 \end{bmatrix}"] \& C^2 \arrow[r,"I_2"]\arrow[d,"I_2"] \& C^2 \arrow[r,"{\begin{bmatrix} 0 & 1 \end{bmatrix}}"]\arrow[d,"{\begin{bmatrix} 1 & -1 \end{bmatrix}}"] \& C  \\
 \& A \arrow[r,"\begin{bmatrix} x \\ y \end{bmatrix}"] \& A^2 \arrow[r,"I_2"] \& A^2 \arrow[r,"{\begin{bmatrix} 1 & -1 \end{bmatrix}}"] \& A \&  \\
\end{tikzcd}
\]
When checking that the morphism is closed, $\mk m_6$ terms appear: In the fourth column $t^\alpha{\begin{bmatrix} x \\ 1 \end{bmatrix}}$ from $C$ to $A^2$ and in the fifth column $t^\alpha(x-1)$ from $C$ to $A$.
To check that the morphism is an isomorphism one constructs an explicit inverse in a similar way.
\end{proof}

\begin{remark}
Let us comment on the above proof.
Note that while all the categories $\mc F(S_\zeta,M_\zeta;\mc E)_{\mathbf K}$ are canonically identified for $\zeta\in\RR_{>0}$, the spectral network representative of a given object can depend on the area of the hexagon.
Suppose $p\in R_2$ for a given area $\alpha$, so the spectral network representative of $\mc O_p$ is a hexagon with six spokes. 
If we shrink $\alpha$, then the spokes shrink with it, and eventually $p\in R_3$ so the spectral network representative is just the outer hexagon without spokes.
\end{remark}

\subsection{Induction Step}
\label{subsec_induction}

As before, for $\zeta\in\CC^\times$ we write $M_\zeta$ for the set of vertices of a regular hexagon centered at the origin in $\CC$ and with $\zeta\in M_\zeta$, and $S_\zeta$ for the closed disk with $M_\zeta$ on its boundary.
We equip $S_\zeta$ with the quadratic differential $\nu_\zeta=\zeta^2dz^2$.
The coefficient category is $\mc E=\mc A_2$ with the stability condition as in Subsection~\ref{subsec_coeffcat} throughout.

Suppose $G$ is a graph with endpoints on $M_{\zeta}$, then we obtain the category over $\mR$
and its extension of scalars to $\mK$
$$
\Ff _G(S_\zeta; \Ee ) \rightarrow \Ff _G(S_\zeta; \Ee )_{\mK}.
$$
We'll assume that $G$ stays within the convex hull of $M_{\zeta}$. 

If $\eta$ is another point, we'll say it is {\em outside $\zeta$} if 1) the straight line segment 
$\rho$ from $\zeta$ to $\eta$ only intersects the convex hull of $M_{\zeta}$ at the point $\zeta$, and 2) the convex hull of $M_\eta$ contains $M_\zeta$. 
In this case, the rotational translates of $\rho$ will be disjoint and also stay outside the convex hull of $M_{\zeta}$.
Recall from Subsection~\ref{subsec_fukdisk} that there is a functor
\begin{equation}
\epsilon_{\rho} :\mc F_G(S_\zeta;\mc E)\to \mc F_{G^+}(S_\eta;\mc E)
\end{equation}
where $G^+$ is a graph containing $G$ and the six rotational translates of $\rho$.

On the other hand we have the B-side category $\Pp _{\mK} := D^b(C_{\mK}) \cong D^b(A_5\otimes A_2)$
defined over the field $\mK$. 
The following proposition is a special case of Proposition~\ref{prop_fun_spokes}.

\begin{prop}
\label{functorH}
Suppose $G$ is a connected graph  meeting all the vertices of  $M_{\zeta}$. 
Then there is an equivalence of categories
$$
H_{\zeta}:\Pp _{\mK} \stackrel{\cong}{\longrightarrow} \Ff _G(S; \Ee )_{\mK}
$$
for any choice of $\mathrm{Arg}(\zeta)$.
These satisfy the following compatibility constraint. Suppose $\eta$ is another point outside of $\zeta$ in the above sense,
with $\rho$ the line segment from $\zeta$ to $\eta$ and $\epsilon_{\rho}$ the extension of objects functor to 
a graph $G^+$ that is connected graph with endpoints on $M_{\eta}$ containing $G$ and the rotational translates of $\rho$.
Also choose the unique $\mathrm{Arg}(\eta)$ with $|\mathrm{Arg}(\eta)-\mathrm{Arg}(\zeta)|<\pi$.
Then the diagram 
$$
\begin{array}{ccc}
\Pp _{\mK} & \stackrel{H_{\zeta}}{\longrightarrow} & 
\Ff _G(S; \Ee )_{\mK} \\
{\scriptstyle = }\downarrow & & \downarrow {\scriptstyle \epsilon_{\rho}} \\
\Pp _{\mK} & \stackrel{H_{\eta}}{\longrightarrow} & 
\Ff _{G^+ }(S; \Ee )_{\mK} 
\end{array}
$$
commutes.  
\end{prop}

\begin{coro}
\label{secondpart}
Suppose we are given an object $B \in \Pp _{\mK}$, and let $B' = T_{Z}(B)$ be the result of applying the twist
along an object  $Z=Z_1\oplus\cdots \oplus Z_6$ as in Subsection~\ref{subsec_bside}. 
We have the object $Z_{\zeta,\mK}:= H_{\zeta}(Z)$.
Let $A_{\zeta , \mK}:= H_{\zeta}(B)$ and $A'_{\zeta , \mK}:= H_{\zeta}(B')$.
Then
$$
A'_{\zeta , \mK} = T_{Z_{\zeta,\mK}}(A_{\zeta , \mK} ).
$$
\end{coro}
\begin{proof}
This is just because of the functoriality of the twist operation. 
\end{proof}

Now suppose $\eta$ is outside of $\zeta$, and  $\eta '= e^{\pm i\pi /3}\eta$. 
We suppose that $\eta '$ is also outside of $\zeta$ in the previous sense. 

Let $\sigma$ and $\sigma '$ be the straight path segments joining $\eta$ to $\zeta$ and $\zeta$ to $\eta '$, respectively. 
Let $\Sigma$ be the union of $\sigma$ and $\sigma '$, it is a path joining $\eta$ to $\eta '$. 

Suppose given a graph $G$ whose endpoints are on $M_{\zeta}$. Assume that $G$ is contained in the convex hull of $M_{\zeta}$. 
Let $G\langle \Sigma  \rangle$ denote the graph obtained by adding the edges $e^{i\pi n/ 3} \sigma$  and $e^{i\pi n/ 3} \sigma '$ to $G$, to get a graph whose endpoints are on $M_{\eta}$. 
The new edges don't intersect $G$ except on $M_{\zeta}$.

\begin{lemma}
\label{twochoices}
Suppose $G$ is a spectral network graph for some phase $\phi$. We are interested in the possible choices of the pair $\{ \eta , \eta '\}$ such that $\sigma$ and $\sigma '$ are spectral network edges of the same phase $\phi$. 
\newline
(1)\, If the edges of $G$ are parallel to the edges of the
hexagon $M_{\zeta}$, then there is a single choice for $\{ \eta , \eta '\}$ with this property. 
\newline
(2)\, If the edges of $G$ are not parallel to the edges of the
hexagon $M_{\zeta}$, then are two distinct choices for $\{ \eta , \eta '\}$ with this property. 
\end{lemma}
\begin{proof}
Given a fixed phase $\phi$, let $\{ \phi \}$ denote the set of directions in the complex plane $S$ that are 
associated to phases $\phi + k / 3$ for $k\in \zz$. If $\zeta$ is a nonzero complex number, we say that
$\{ \phi \}$ is {\em special} with respect to $M_{\zeta}$ if the edge directions of the hexagon spanned by
$M_{\zeta}$ are in $\{ \phi \}$, {\em non-special} otherwise.

If $\{ \phi \}$ is special with respect to $M_{\zeta}$, then at each point $y\in M_{\zeta}$ there are
three {\em outgoing rays} and three {\em inward rays} in the directions of $\{ \phi \}$ at $y$.

If
$\{ \phi \}$ is non-special with respect to $M_{\zeta}$, then at each point $y\in M_{\zeta}$ there are
four {\em outgoing rays} and two {\em inward rays} in the directions of $\{ \phi \}$ at $y$. 
Figure \ref{inwardoutward-fig} gives the picture in the non-special case (we leave it to the reader to draw the special case). 
The special case is part (1) of the lemma, the non-special case is part (2).

\begin{figure}
\centering
\setlength{\unitlength}{.5mm}
\begin{picture}(200,120)

\put(100,75){\circle*{3}}
\put(40,40){\circle*{3}}
\put(160,40){\circle*{3}}

\linethickness{.2mm}
\qbezier(100,75)(100,75)(40,40)
\qbezier(100,75)(100,75)(160,40)

\qbezier(40,40)(40,10)(40,10)
\qbezier(160,40)(160,10)(160,10)

\qbezier(100,75)(100,75)(80,71)
\qbezier(100,75)(100,75)(86,91)
\qbezier(100,75)(100,75)(106,93)
\qbezier(100,75)(100,75)(120,79)
\qbezier(100,75)(100,75)(114,59)
\qbezier(100,75)(100,75)(94,57)

\qbezier(160,40)(160,40)(180,36)
\qbezier(160,40)(160,40)(174,56)
\qbezier(160,40)(160,40)(154,58)
\qbezier(160,40)(160,40)(140,44)
\qbezier(160,40)(160,40)(144,24)
\qbezier(160,40)(160,40)(166,22)

\qbezier(40,40)(40,40)(60,36)
\qbezier(40,40)(40,40)(54,56)
\qbezier(40,40)(40,40)(34,58)
\qbezier(40,40)(40,40)(20,44)
\qbezier(40,40)(40,40)(24,24)
\qbezier(40,40)(40,40)(46,22)

%%%%%%%%%%%%%%%%%%% notations
\put(99,63){$\zeta$}

%%%%%%%%%%%%%%%%%%%

\end{picture}
\caption{\small Two inward and four outward rays in $\{ \phi \}$ at each point of $M_{\zeta}$}
\label{inwardoutward-fig}
\end{figure}
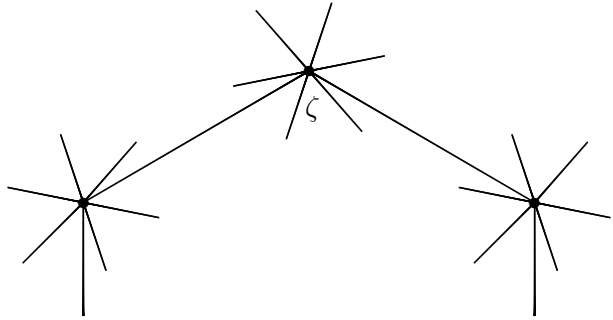

A {\em directional choice} for $M_{\zeta}$ consists of choosing, uniformly at each point in $M_{\zeta}$, either the leftmost or rightmost outgoing ray $R$ in the non-special case, and always the leftmost outgoing ray $R$ in the special case.
In each case, there is unique outgoing ray $\widetilde{R}$ making an angle of $2\pi /3$ with $R$.  
These are drawn in Figure \ref{directional-fig}, for the left directional choice.

\begin{figure}
\centering
\setlength{\unitlength}{.5mm}
\begin{picture}(200,120)

\put(100,75){\circle*{3}}
\put(40,40){\circle*{3}}
\put(160,40){\circle*{3}}

\linethickness{.2mm}
\qbezier(100,75)(100,75)(160,40)
\qbezier(100,75)(100,75)(40,40)

\qbezier(40,40)(40,10)(40,10)
\qbezier(160,40)(160,10)(160,10)

\linethickness{.5mm}
\qbezier(100,75)(100,75)(80,71)
\linethickness{.2mm}
\qbezier(100,75)(100,75)(86,91)
\linethickness{.5mm}
\qbezier(100,75)(100,75)(106,93)
\linethickness{.2mm}
\qbezier(100,75)(100,75)(120,79)
\qbezier(100,75)(100,75)(114,59)
\qbezier(100,75)(100,75)(94,57)

\qbezier(40,40)(40,40)(20,36)
\linethickness{.5mm}
\qbezier(40,40)(40,40)(26,56)
\linethickness{.2mm}
\qbezier(40,40)(40,40)(46,58)
\qbezier(40,40)(40,40)(60,44)
\qbezier(40,40)(40,40)(56,24)
\linethickness{.5mm}
\qbezier(40,40)(40,40)(34,22)
\linethickness{.2mm}

\qbezier(160,40)(160,40)(140,36)
\linethickness{.5mm}
\qbezier(160,40)(160,40)(146,56)
\linethickness{.2mm}
\qbezier(160,40)(160,40)(166,58)
\linethickness{.5mm}
\qbezier(160,40)(160,40)(180,44)
\linethickness{.2mm}
\qbezier(160,40)(160,40)(176,24)
\qbezier(160,40)(160,40)(154,22)

%%%%%%%%%%%%%%%%%%% notations
\put(99,64){$\zeta$}

\put(72,69){$R$}

\put(105,96){$\widetilde{R}$}

%\qbezier(108,55)(120,60)(124,69)
%\qbezier(124,69)(124,69)(127,66)
%\qbezier(124,69)(124,69)(118,67)

%\put(116,52){$e^{i\theta}$}

%%%%%%%%%%%%%%%%%%%

\end{picture}
\caption{\small Directional choice}
\label{directional-fig}
\end{figure}
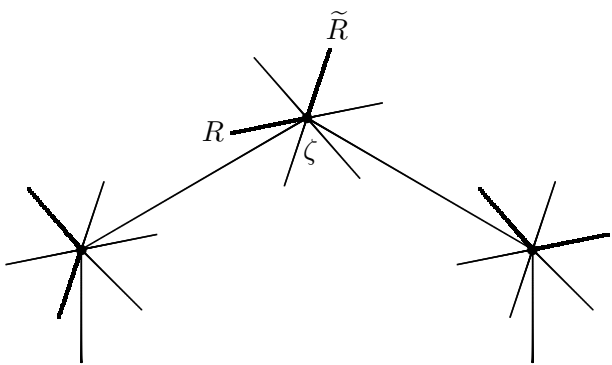

Given a directional choice, consider the ray $\widetilde{R}$ starting at $\zeta$, as well as the ray $R$ at $e^{-i\pi / 3} \zeta$ for the left directional choice or at $e^{i\pi / 3} \zeta$ for the right directional choice. 
These meet at a point $\eta$. 

Let $\eta ' := e^{i\pi / 3} \eta$ for the left directional choice and $\eta ' := e^{-i\pi / 3} \eta$ for the right directional choice. 
The ray $R$ starting at $\zeta$ meets $\eta '$. 
Let $\sigma$ be the path segment joining $\zeta$ to $\eta$ and $\sigma '$ the path segment joining $\zeta$ to $\eta '$. 
See Figure \ref{determining-fig}. 

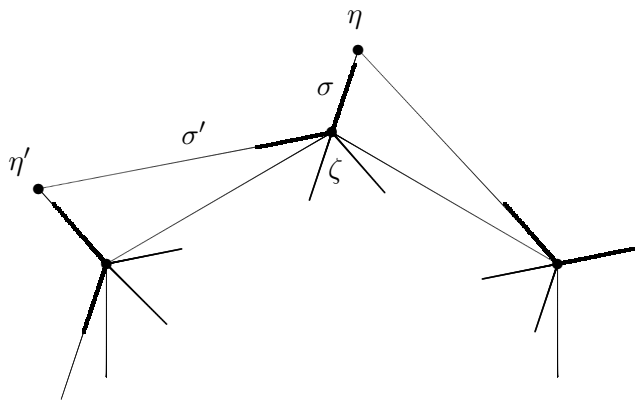
\begin{figure}
\centering
\setlength{\unitlength}{.5mm}
\begin{picture}(200,120)

\put(100,75){\circle*{3}}
\put(160,40){\circle*{3}}
\put(40,40){\circle*{3}}

\linethickness{.1mm}
\qbezier(100,75)(100,75)(160,40)
\qbezier(100,75)(100,75)(40,40)

\qbezier(40,40)(40,10)(40,10)
\qbezier(160,40)(160,10)(160,10)

\linethickness{.5mm}
\qbezier(100,75)(100,75)(80,71)
\qbezier(100,75)(100,75)(106,93)
\linethickness{.2mm}
\qbezier(100,75)(100,75)(114,59)
\qbezier(100,75)(100,75)(94,57)

\linethickness{.5mm}
\qbezier(40,40)(40,40)(26,56)
\linethickness{.2mm}
\qbezier(40,40)(40,40)(60,44)
\qbezier(40,40)(40,40)(56,24)
\linethickness{.5mm}
\qbezier(40,40)(40,40)(34,22)
\linethickness{.2mm}

\qbezier(160,40)(160,40)(140,36)
\linethickness{.5mm}
\qbezier(160,40)(160,40)(146,56)
\qbezier(160,40)(160,40)(180,44)
\linethickness{.2mm}
\qbezier(160,40)(160,40)(154,22)

\put(22,60){\circle*{3}}

\linethickness{.1mm}
\qbezier(100,75)(100,75)(22,60)
\qbezier(40,40)(40,40)(22,60)

\put(107,97){\circle*{3}}

\linethickness{.1mm}
\qbezier(100,75)(100,75)(107,97)
\qbezier(160,40)(160,40)(107,97)

\qbezier(40,40)(40,40)(28,4)

%%%%%%%%%%%%%%%%%%% notations
\put(99,63){$\zeta$}

\put(14,65){$\eta '$}

\put(104,104){$\eta $}

\put(60,72){$\sigma '$}
\put(96,85){$\sigma $}

%\put(170,45){$e^{-i\theta}\sigma'$}

%%%%%%%%%%%%%%%%%%%

\end{picture}
\caption{\small Determining $\eta$}
\label{determining-fig}
\end{figure}

Given the directional choice, the pair $\{\eta,\eta '\}$ is uniquely determined. 
In the non-special case (2) there are two possible directional choices, giving two possibilities for $\{ \eta , \eta '\}$. 
In the special case (1) there is only one possible choice of a pair of outgoing rays separated by $2\pi / 3$ since there
are only three outgoing rays, yielding the single choice in (1). This completes the proof. 
\end{proof}

Note from the picture that for $\zeta=m+e^{-i\pi/6}n$ we get $\eta=m+e^{-i\pi/6}(n+m)$ for the left directional choice and $\eta=(m+n)+e^{-i\pi/6}n$ for the right directional choice.
Thus, starting with $\zeta=1$ one stays in the sector with $\frac{1}{\pi}\mathrm{Arg}(\zeta)\in (-1/3,0]$ for any sequence of directional choices.
Given our conventions for the quadratic differential $\nu_\zeta$, this means that starting from a spectral network of phase 0 we obtain spectral networks (at least their underlying graphs) with phases in $[0,1/3)$.

We assume now that $\{ \eta , \eta '\}$ is the left of the two directional choices given by Lemma \ref{twochoices}, as in our pictures so far. 
The case of right directional choice is similar except that twists need to be replaced by dual twists.

We choose $\sigma$ for the path $\rho$ in Proposition \ref{functorH}, so we use the extension functor $\epsilon _{\sigma}$, and are now going to apply Corollary \ref{secondpart} to the new graph $G\langle \Sigma \rangle$ with endpoints
on $M_{\eta}$. For this, we will be choosing an object $Z = \bigoplus _{k=0}^5 Z_k$ in $\Pp$. 
Let $Z_{\eta , \mK} = \bigoplus Z_{k,\eta , \mK}$ denote the resulting object in  $\Pp _{\mK}$
with $Z_{k,\eta , \mK} = H_{\eta} (Z_k)$. 

We are going to choose objects 
$$
Z_{k,\eta} \in \Ff _{G\langle \Sigma  \rangle }(S_\eta; \Ee )
$$
lifting the $Z_{k,\eta , \mK}$, and such that $Z_{k,\eta}$ is supported on $e^{i\pi k/3}\Sigma$. This yields the lift 
$$
Z_{\eta} = \bigoplus _{k=0}^5 Z_{k,\eta}.
$$
The spherical twist functor along $Z_{\eta}$ provides a lift of the spherical twist along $Z_{\eta , \mK}$.

\begin{lemma}
\label{sphericaltwistdiagram}  
These fit into the diagram
$$
\begin{array}{ccc}
\Ff _{G\langle \Sigma  \rangle }(S_\eta; \Ee ) & \stackrel{T_{Z_{\eta}}}{\longrightarrow} & 
\Ff _{G\langle \Sigma  \rangle }(S_\eta; \Ee ) 
\\
\downarrow & & \downarrow \\
\Ff _{G\langle \Sigma  \rangle }(S_\eta; \Ee )_{\mK}  & \stackrel{T_{Z_{\eta,\mK}}}{\longrightarrow} & 
\Ff _{G\langle \Sigma  \rangle }(S_\eta; \Ee )_{\mK} 
\\
{\scriptstyle H_{\eta}} \uparrow & & \uparrow {\scriptstyle H_{\eta}} \\
\Pp _{\mK} & \stackrel{T_{Z}}{\longrightarrow} & \Pp _{\mK}  .
\end{array}
$$
\end{lemma}

Choose the point $\eta$ as in Lemma \ref{twochoices} above. 
Let $\delta$ be the difference
between the phases of the main segments
for $M_{\eta}$ and $M_{\zeta}$. The sign is normalized by saying that the spectral network condition of phase
$\phi + \delta$ relative to $\nu_\eta$ should be the same as the spectral network condition of phase $\phi$ relative
to $\nu_\zeta$. 

\begin{theorem}
\label{twistsn}
Suppose $Z_{k,\eta}$ is chosen so that its restriction to the edge $e^{i\pi k/3}\sigma '$ is a spectral network object
of phase $\phi +\delta + 1$ relative to $\nu_\eta$  and rank $1$. 
If $A\in \Ff _G(S_\zeta; \Ee )$ is a spectral network object of phase $\phi$ with boundary on $M_{\zeta}$ 
then 
$$
T_{Z_{\eta}}(\epsilon _{\sigma} A) \in \Ff _G(S_\eta; \Ee )
$$
is a spectral network object of phase $\phi + \delta$ with boundary on $M_{\eta}$. 
\end{theorem}

\begin{coro}
\label{inductiveconclusion}
Suppose we know that every semistable object of $\Ff (S_\zeta,M_{\zeta}; \Ee ) _{\mK}$ of phase $\phi$ has a 
spectral network representative. 
Then every semistable object 
of $\Ff (S_\eta,M_{\eta}; \Ee ) _{\mK}$ of phase 
$\phi + \delta $ has a 
spectral network representative. 
\end{coro}

\begin{proof}
Suppose given a semistable object $A'_{\mK}$ of $\Ff (S_\eta,M_{\eta}; \Ee ) _{\mK}$ of phase $\phi + \delta$. 
By the equivalence of Proposition \ref{functorH}, we have $A'_{\mK}=H_{\eta}(B')$ for $B'$ a semistable object of $\Pp _{\mK}$
of phase $\phi + \delta$. We have seen in Proposition~\ref{prop63} that $B'=T_Z(B)$ for a semistable object $B$ of phase
$\phi$, such that $H_{\zeta}(B)$ is a semistable object in 
$\Ff (S_\zeta,M_{\zeta}; \Ee ) _{\mK}$ of the same phase $\phi$.

By the inductive hypothesis, there is a spectral network object $A$ in $\Ff (S_\zeta,M_{\zeta}; \Ee )$ whose projection to the 
category over $\mK$ is $A_{\mK}=H_{\zeta}(B)$. Consider the object $A'':= \epsilon _{\sigma}(A)$
whose projection to the category over $\mK$ is
$$
A''_{\mK} = \epsilon _{\sigma}(A_{\mK}) = \epsilon _{\sigma}H_{\zeta}(B) = H_{\eta}(B),
$$
the last equivalence being by the commutative diagram in
Proposition \ref{functorH}. By Lemma \ref{sphericaltwistdiagram} we have
$$
T_{Z_{\eta, \mK}}(A''_{\mK}) = 
T_{Z_{\eta, \mK}}( H_{\eta}(B)) = H_{\eta}(T_Z(B))=H_{\eta}(B').
$$
We set $A':= T_{Z_{\eta}}(A'') = T_{Z_{\eta}}(\epsilon _{\sigma} A)$. The functor of projection from the category over 
$\mR$ to the category over $\mK$ is compatible with spherical twists, so the projection of $A'$ over $\mK$
which we'll provisionally denote by $(A')_{\mK}$ is equal to the original $A'_{\mK}$:
$$
(A')_{\mK} = T_{Z_{\eta, \mK}}(A''_{\mK})=H_{\eta}(B') =A'_{\mK}.
$$
On the other hand, by Theorem \ref{twistsn}, $A'$ is a spectral network object of phase $\phi$. This completes the proof
that our semistable object $A'_{\mK}$ admits a spectral network representative of phase $\phi$. 
\end{proof}

\bigskip

The rest of this section is devoted to the proof of Theorem \ref{twistsn}.

\bigskip

Recall from \eqref{sum_of_phases} that an object $A\in \Ff _G(S_\eta; \Ee )$ is a {\em spectral network object} of phase $\phi$ if, when we extend scalars to $\mbk$ and then restrict to any edge, we obtain an object of $\Ee$ that is semistable of the appropriate phase combining $\phi$ with the phase of the edge.

Suppose now that $\{ \phi \}$ is fixed. We assume case (2) of Lemma \ref{twochoices},
and furthermore we fix one of the directional choices; the other leads to similar pictures that
are left to the reader.  We'll draw pictures of the objects discussed below following the drawings that
were started in Lemma \ref{twochoices}. For convenience
various notations are summed up in Figure \ref{notations-fig}.

\begin{figure}
\centering
\setlength{\unitlength}{.5mm}
\begin{picture}(200,120)

\put(100,75){\circle*{3}}
\put(160,40){\circle*{3}}
\put(40,40){\circle*{3}}

\linethickness{.1mm}
\qbezier(100,75)(130,61)(160,40)
\qbezier(100,75)(70,61)(40,40)

\qbezier(40,40)(40,10)(40,10)
\qbezier(160,40)(160,10)(160,10)

\linethickness{.2mm}
\qbezier(100,75)(100,75)(114,59)
\qbezier(100,75)(100,75)(94,57)

\qbezier(40,40)(40,40)(60,44)
\qbezier(40,40)(40,40)(56,24)

\qbezier(160,40)(160,40)(140,36)
\qbezier(160,40)(160,40)(154,22)

\put(22,60){\circle*{3}}

\qbezier(100,75)(100,75)(22,60)
\qbezier(40,40)(40,40)(22,60)

\put(107,97){\circle*{3}}

\qbezier(100,75)(100,75)(107,97)
\qbezier(160,40)(160,40)(107,97)

\qbezier(40,40)(40,40)(28,4)
\qbezier(160,40)(160,40)(180,44)

%%%%%%%%%%%%%%%%%%% notations
\put(99,63){$\zeta$}

\put(14,65){$\eta '$}

\put(104,104){$\eta $}

\put(60,72){$\sigma '$}
\put(96,85){$\sigma $}

%\qbezier(108,55)(120,60)(124,69)
%\qbezier(124,69)(124,69)(127,66)
%\qbezier(124,69)(124,69)(118,67)

%\put(116,52){$e^{i\theta}$}

%\put(48,70){$e^{i\theta}\sigma$}

%\put(170,45){$e^{-i\theta}\sigma'$}
%%%%%%%%%%%%%%%%%%%

\end{picture}
\caption{\small A summary of the notations}
\label{notations-fig}
\end{figure}
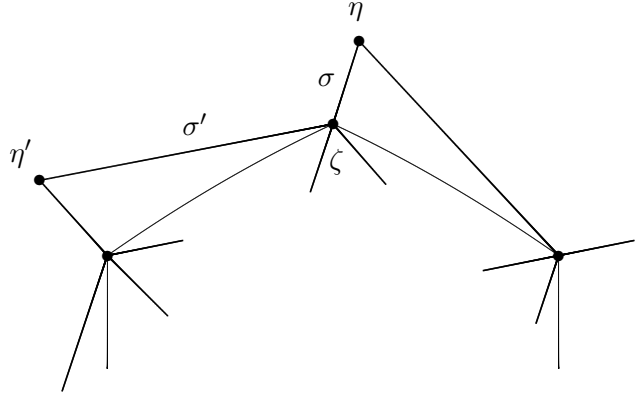

The path $\sigma$ and its rotational translates are used to carry the object $A_{\zeta}$, that we are assuming
is a spectral network object with respect to $M_{\zeta}$, to an object 
$\epsilon _{\sigma}A_{\zeta}$ as shown in Figure \ref{epsilonA-fig}. 
This will retain the spectral network property on the interior of the hexagon spanned by
$M_{\zeta}$, but it will not have that property along the edge $\sigma$. This is going to be corrected by the spherical
twist. 

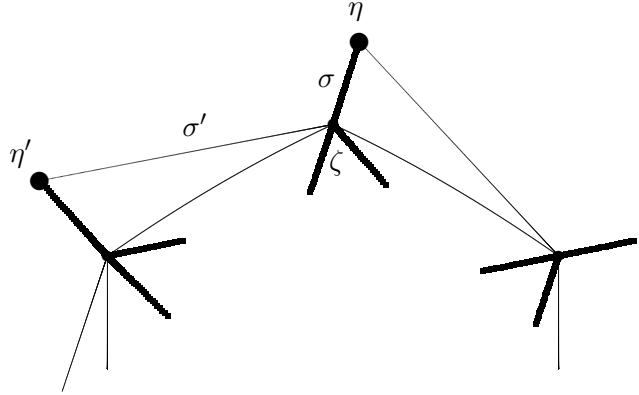
\begin{figure}
\centering
\setlength{\unitlength}{.5mm}
\begin{picture}(200,120)

\put(100,75){\circle*{3}}
\put(160,40){\circle*{3}}
\put(40,40){\circle*{3}}

\linethickness{.1mm}
\qbezier(100,75)(130,61)(160,40)
\qbezier(100,75)(70,61)(40,40)

\qbezier(40,40)(40,10)(40,10)
\qbezier(160,40)(160,10)(160,10)

\linethickness{.7mm}
\qbezier(100,75)(100,75)(114,59)
\qbezier(100,75)(100,75)(94,57)

\qbezier(40,40)(40,40)(60,44)
\qbezier(40,40)(40,40)(56,24)

\qbezier(160,40)(160,40)(140,36)
\qbezier(160,40)(160,40)(154,22)

\put(22,60){\circle*{5}}

\linethickness{.1mm}
\qbezier(100,75)(100,75)(22,60)
\linethickness{.7mm}
\qbezier(40,40)(40,40)(22,60)
\linethickness{.1mm}

\put(107,97){\circle*{5}}

\linethickness{.7mm}
\qbezier(100,75)(100,75)(107,97)
\linethickness{.1mm}
\qbezier(160,40)(160,40)(107,97)

\qbezier(40,40)(40,40)(28,4)
\linethickness{.7mm}
\qbezier(160,40)(160,40)(180,44)
\linethickness{.1mm}

%%%%%%%%%%%%%%%%%%% notations
\put(99,63){$\zeta$}

\put(14,65){$\eta '$}

\put(104,104){$\eta $}

\put(60,72){$\sigma '$}
\put(96,85){$\sigma $}

%\put(48,70){$e^{i\theta}\sigma$}
%%%%%%%%%%%%%%%%%%%

\end{picture}
\caption{\small The object $\epsilon _{\sigma} A_{\zeta}$ includes $\sigma$ but not $\sigma '$ }
\label{epsilonA-fig}
\end{figure}

We have the path $\Sigma = \Sigma (\eta ,\eta ' )$ 
between $\eta$ and $\eta '$ as follows.  From $\eta$ go along $\sigma$ to $\zeta$, then along $\sigma '$ to $\eta '$. 
This path is shown in Figure \ref{spherical-fig}. 

\begin{figure}
\centering
\setlength{\unitlength}{.5mm}
\begin{picture}(200,120)

\put(100,75){\circle*{3}}
\put(160,40){\circle*{3}}
\put(40,40){\circle*{3}}

\linethickness{.1mm}
\qbezier(100,75)(130,61)(160,40)
\qbezier(100,75)(70,61)(40,40)

\qbezier(40,40)(40,10)(40,10)
\qbezier(160,40)(160,10)(160,10)

\linethickness{.1mm}
\qbezier(100,75)(100,75)(114,59)
\qbezier(100,75)(100,75)(94,57)

\qbezier(40,40)(40,40)(60,44)
\qbezier(40,40)(40,40)(56,24)

\qbezier(160,40)(160,40)(140,36)
\qbezier(160,40)(160,40)(154,22)

\put(22,60){\circle*{5}}

\linethickness{.7mm}
\qbezier(100,75)(100,75)(22,60)
\linethickness{.1mm}
\qbezier(40,40)(40,40)(22,60)

\put(107,97){\circle*{5}}

\linethickness{.7mm}
\qbezier(100,75)(100,75)(107,97)
\linethickness{.1mm}
\qbezier(160,40)(160,40)(107,97)

\qbezier(40,40)(40,40)(28,4)
\qbezier(160,40)(160,40)(180,44)

%%%%%%%%%%%%%%%%%%% notations

\put(14,65){$\eta '$}

\put(104,104){$\eta $}

\put(60,72){$\sigma '$}
\put(96,85){$\sigma $}

%%%%%%%%%%%%%%%%%%%

\end{picture}
\caption{\small The path $\Sigma (\eta ,\eta ' )$ supporting the object $Z_{0,\zeta}$ }
\label{spherical-fig}
\end{figure}
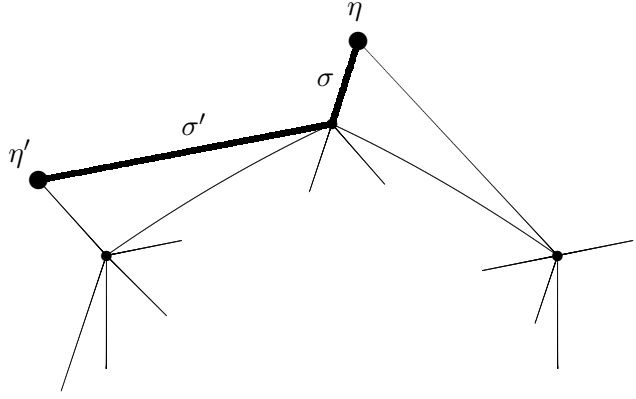

We will then choose our lifted object $Z_{\eta}$ to be a direct sum of $Z_{i,\eta }$ where these are rotations of the
basic $Z_{0,\eta }$ that is supported on the path $\Sigma (\eta ,\eta ' )$. We make the hypothesis that
$Z_{0,\eta }$ is a spectral network object of phase $\phi$  on $\sigma '$, it follows that it is 
a spectral network object of phase
$\phi + 1 / 3$  on $\sigma$. 

Recall from \eqref{twist6} that the twist $T_Z$ is defined by exact triangle
$$
\bigoplus_{i=1}^6\Hom(Z_i,X)\otimes Z_i \rightarrow X \rightarrow T_Z(X) \rightarrow \bigoplus_{i=1}^6\Hom(Z_i,X)\otimes Z_i[1].
$$
We note that for the other directional choice, one needs to use instead the dual twist given by
$$
\bigoplus_{i=1}^6\Hom(X,Z_i)\otimes Z_i[-1] \rightarrow T_Z(X) \rightarrow X \rightarrow \bigoplus_{i=1}^6\Hom(X,Z_i)\otimes Z_i.
$$

In the situation of our picture we'll be looking at the twist about the object $Z_{\eta}$ which is the direct sum of the $Z_{i,\eta }$.  
Applied to our extended object we obtain $T_{Z_{\eta}} (\epsilon _{\sigma} A_{\zeta}) $ whose support is
shown in Figure \ref{twist-fig}. 

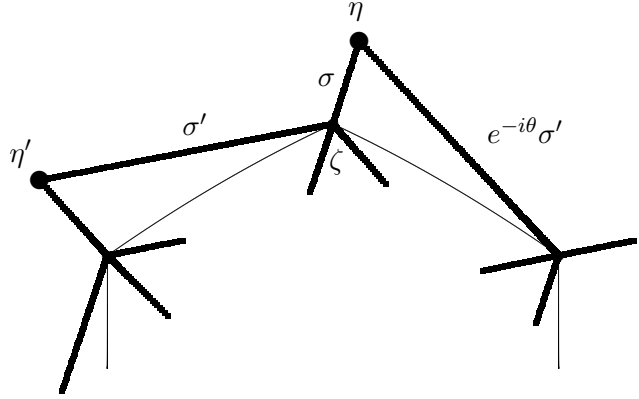
\begin{figure}
\centering
\setlength{\unitlength}{.5mm}
\begin{picture}(200,120)

\put(100,75){\circle*{3}}
\put(160,40){\circle*{3}}
\put(40,40){\circle*{3}}

\linethickness{.1mm}
\qbezier(100,75)(130,61)(160,40)
\qbezier(100,75)(70,61)(40,40)

\qbezier(40,40)(40,10)(40,10)
\qbezier(160,40)(160,10)(160,10)

\linethickness{.7mm}
\qbezier(100,75)(100,75)(114,59)
\qbezier(100,75)(100,75)(94,57)

\qbezier(40,40)(40,40)(60,44)
\qbezier(40,40)(40,40)(56,24)

\qbezier(160,40)(160,40)(140,36)
\qbezier(160,40)(160,40)(154,22)

\put(22,60){\circle*{5}}

\qbezier(100,75)(100,75)(22,60)
\qbezier(40,40)(40,40)(22,60)

\put(107,97){\circle*{5}}

\qbezier(100,75)(100,75)(107,97)
\qbezier(160,40)(160,40)(107,97)

\qbezier(40,40)(40,40)(28,4)
\qbezier(160,40)(160,40)(180,44)

%%%%%%%%%%%%%%%%%%% notations
\put(99,63){$\zeta$}

\put(14,65){$\eta '$}

\put(104,104){$\eta $}

\put(60,72){$\sigma '$}
\put(96,85){$\sigma $}

\put(141,70){$e^{-i\theta}\sigma '$}

%%%%%%%%%%%%%%%%%%%

\end{picture}
\caption{\small The new twisted object $A'_{\eta} := T_{Z_{\eta}} (\epsilon _{\sigma} A_{\zeta}) $ is going to be a spectral network. At 
$\zeta$ it combines a collision of $\sigma$, $\sigma '$ and one of the inward edges, together with a straight piece
consisting of $\sigma $ and the other inward edge } 
\label{twist-fig}
\end{figure}

After making a directional  choice and letting $\sigma$ and the choice of liftings $Z_{i,\eta }$ be fixed 
as above, suppose that 
$A_{\zeta}$ is a spectral network object in $\Ff _G(S_\eta; \Ee )$, of phase $\phi$, lifting $A_{\zeta , \mK}$. 
We obtain the extension to an object 
$$
\epsilon _{\sigma} A_{\zeta} \in \Ff _{G\langle \Sigma \rangle }(S_\eta; \Ee ).
$$

The objective is to show that the spherical twist
$$
A'_{\eta} := T_{Z_{\eta}} (\epsilon _{\sigma} A_{\zeta}) 
$$
is a spectral network object lifting 
$$
A'_{\eta , \mK } = T_{Z_{\eta , \mK}} (A_{\eta , \mK}) = H_{\eta}(B').
$$

The fact that the spherical twist $A'_{\eta}$ is an object lifting $A'_{\eta , \mK }$ is due to the compatibility
of spherical twists along functors, applied to the extension of scalars functor from the category over $\mR$ to the
category over $\mK$.

We would like to show that $A'_{\eta}$ satisfies the spectral network property. 
This is true everywhere on the support of the original object $A_{\zeta}$. It is also true from our choice
of objects $Z_{i,\eta}$ on the rotational translates of the edges $\sigma '$, since the extended object 
$\epsilon _{\sigma} A_{\zeta}$ is zero on these edges. See below for more details on the choice
of $Z_{i,\eta}$  that is designed to insure the spectral network property on $\sigma$. 
For the proof, what will remain is to justify why the spectral network property holds on the rotational
translates of $\sigma$. It suffices to consider just $\sigma$ itself.

If we include the real blow-up interval at the point $\eta$ into the picture, and allow an ``unzipping'' perturbation
(using Lemma~\ref{edgeremoval}), the
picture of the two objects $\epsilon _{\sigma} A_{\zeta}$ and $Z_{0,\zeta}$ is shown in Figure~\ref{circle-fig}. 

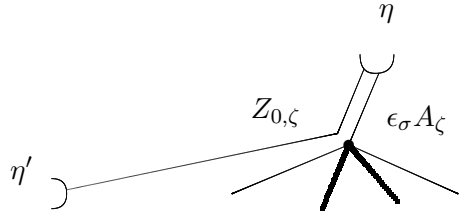
\begin{figure}
\centering
\setlength{\unitlength}{.5mm}
\begin{picture}(200,120)

\put(101,74){\circle*{3}}

\linethickness{.1mm}
\qbezier(101,74)(70,61)(70,61)
\qbezier(101,74)(130,61)(130,61)

\linethickness{.6mm}
\qbezier(101,74)(101,74)(114,59)
\qbezier(101,74)(101,74)(94,57)

%\put(22,61){\circle{8}}
\linethickness{.1mm}
\qbezier(22,65)(26,65)(26,61)
\qbezier(26,61)(26,57)(22,57)

\qbezier(98,77)(98,77)(26,62)

%\put(108,98){\circle{9}}
\qbezier(113,98)(113,93)(108,93)
\qbezier(108,93)(104,93)(104,98)

\qbezier(98,77)(98,77)(105,94)

\qbezier(101,74)(101,74)(109,93)

%%%%%%%%%%%%%%%%%%% notations

\put(11,65){$\eta '$}

\put(109,108){$\eta$}

\put(75,81){$Z_{0,\zeta}$}

\put(111,79){$\epsilon _{\sigma} A_{\zeta}$}

%%%%%%%%%%%%%%%%%%%

\end{picture}
\caption{\small The objects $\epsilon _{\sigma} A_{\zeta}$ and $Z_{0,\zeta}$ meeting at the real blow-up of $\eta$ } 
\label{circle-fig}
\end{figure}

We need to discuss relative edge orientations and shifts of objects. We'll draw arrows on the edges to indicate
orientations, or just say towards which points the orientations go. Recall that to specify a grading we really need
to specify a lifting of the unit tangent vector to the universal cover of $S^1$. In the pictures, the directions of the 
inward pointing edges, and the directions of the outward pointing edges, will both fall into the same half-circle of 
$S^1$. This condition determines the half-circle up to an angle that is $< \pi$, hence it determines a simply
connected segment of the circle containing all the directions. Thus, while an additional integer is needed to specify
the lifted orientations, the relative positions of the lifted orientations are well defined. 

In this situation, if we want to place an object along a single edge segment that contains an incoming and an
outgoing edge, then the coefficients should be the same object of the fiber category. The Maurer--Cartan element
at the vertex is the identity map, considered as a map of degree $1$ between appropriate shifts of the objects.
On the other hand, suppose we would like to do that for two inward edges, as shown in Figure \ref{through-fig}. 

\begin{figure}
\centering
\setlength{\unitlength}{.5mm}
\begin{picture}(200,80)

\put(100,40){\circle*{3}}

\linethickness{.2mm}
\qbezier(100,40)(100,40)(100,70)

\linethickness{.7mm}
\qbezier(100,40)(120,28)(140,16)
\qbezier(100,40)(80,28)(60,16)
\linethickness{.2mm}

\qbezier(100,70)(100,70)(106,64)
\qbezier(100,70)(100,70)(94,64)

\qbezier(120,28)(120,28)(127,30)
\qbezier(120,28)(120,28)(121,21)

\qbezier(80,28)(80,28)(73,30)
\qbezier(80,28)(80,28)(79,21)

\put(98,73){$0$}
\put(51,5){$W$}
\put(140,5){$W'$}

%%%%%%%%%%%%%%%%%%%

\end{picture}
\caption{\small To go through the vertex using these edge orientations we should have $W'=W[-1]$ } 
\label{through-fig}
\end{figure}
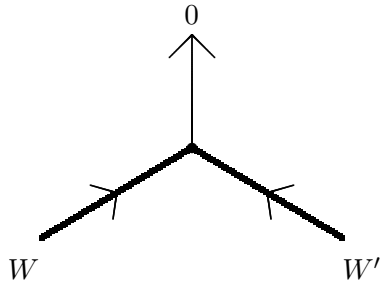

In this case, the Maurer--Cartan element should be the identity viewed as a degree one map
$$
W \stackrel{+1}{\rightarrow} W' = W[-1].
$$

Let us also consider the situation of a collision, as shown in Figure \ref{collision-fig}. The Maurer--Cartan maps
are morphisms of degree $0$ from the left to the upper edge and from the upper to the right edge, while
it is a morphism of degree $+1$ from the right to the left bottom edges. 

\begin{figure}
\centering
\setlength{\unitlength}{.5mm}
\begin{picture}(200,80)

\put(100,40){\circle*{3}}

\linethickness{.2mm}
\qbezier(100,40)(100,40)(100,70)

\linethickness{.7mm}
\qbezier(100,40)(120,28)(140,16)
\qbezier(100,40)(80,28)(60,16)
\linethickness{.2mm}

\qbezier(100,70)(100,70)(106,64)
\qbezier(100,70)(100,70)(94,64)

\qbezier(120,28)(120,28)(127,30)
\qbezier(120,28)(120,28)(121,21)

\qbezier(80,28)(80,28)(73,30)
\qbezier(80,28)(80,28)(79,21)

\put(98,73){$B$}
\put(51,8){$C$}
\put(143,8){$A$}

%%%%%%%%%%%%%%%%%%%

\end{picture}
\caption{\small In a collision we have an exact triangle $A\rightarrow B \rightarrow C \rightarrow A[1]$ } 
\label{collision-fig}
\end{figure}
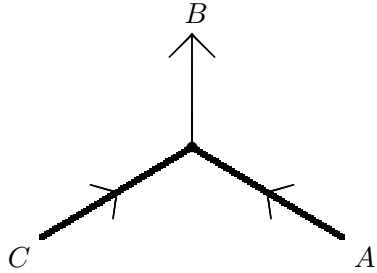

In our picture for the proof, we choose edge orientations of the lower edges in the upward direction towards $\zeta$;
of $\sigma '$ also in the upward direction towards $\eta '$, and of $\sigma$ to the right towards $\eta$,
as shown in Figure \ref{orient-fig}. The above convention is then used to fix the relative lifted orientations. 

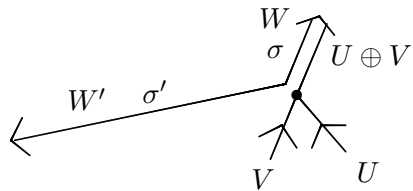
\begin{figure}
\centering
\setlength{\unitlength}{.5mm}
\begin{picture}(200,120)

\put(101,74){\circle*{3}}

\linethickness{.2mm}

\qbezier(101,74)(101,74)(114,59)
\qbezier(101,74)(101,74)(94,57)

\qbezier(98,77)(98,77)(26,62)

\qbezier(98,77)(98,77)(105,94)

\qbezier(101,74)(101,74)(109,93)

\qbezier(97,66)(97,66)(91,63)
\qbezier(97,66)(97,66)(101,60)

\qbezier(107.5,66)(107.5,66)(105,60)
\qbezier(107.5,66)(107.5,66)(114,66)

\qbezier(25,62)(25,62)(28,68)
\qbezier(25,62))(25,62))(30,58)

\qbezier(107,95)(107,95)(101,93)
\qbezier(107,95)(107,95)(111,90)

\put(60,72){$\sigma '$}

\put(93,84){$\sigma $}

\put(91,92){$W$}
\put(40,70){$W '$}

\put(117,51){$U$}
\put(89,49){$V$}

\put(110,82){$U\oplus V$}

\end{picture}
\caption{\small Edge orientations and objects } 
\label{orient-fig}
\end{figure}

We note that $\sigma'$ and $\sigma$ have both outward orientations, so here we are in the situation described above
(up to a symmetry). In particular, if we would like to create an object $Z_{0,\zeta}$ along this single edge using
coefficient objects $W'$ on $\sigma '$ and $W$ on $\sigma$, then the Maurer-Cartan element at the junction 
should be an identification $W' = W[-1]$.

Using these  orientation conventions, 
let's denote by  $U$ and $V$ the coefficient objects in $A_{\zeta}$ along, respectively, the left and right inward edges at $\zeta$.

The spectral network condition means that these are semistable objects of $\Ee = \mc A_2$, of phases differing by $1 / 3$. 
More precisely, we can write
\begin{equation}
\label{expressUV}
U = B^{\oplus m}, \;\;\;\; V = C^{\oplus n}
\end{equation}
in terms of a triple of generating objects $A,B,C$ for $A_2$ that fit into an exact triangle
$$
A \rightarrow B \rightarrow C \rightarrow A[1] .
$$
This choice of notations is motivated by the picture that in the end we'll have a collision between $B$ on the $U$ edge, $A$ on $\sigma '$, and $C$ on $\sigma $.

Now, we will be placing a rank $1$ object of $\Ee$ along the edge $\sigma '$ as coefficient in $Z_{0,\zeta}$, and this is supposed
to be semistable of a certain phase related to the phase of $A_{\zeta}$.

Recall the exact triangle for the twist $T_Z(X)$:
$$
\bigoplus_{i=1}^6\Hom(Z_i,X)\otimes Z_i \rightarrow X \rightarrow T_Z(X) \rightarrow \bigoplus_{i=1}^6\Hom(Z_i,X)\otimes Z_i[1].
$$
Along $\sigma '$, our object $X=\epsilon _{\sigma} A_{\zeta}$ has no coefficient, so if we want $T_Z(X)$ to be semistable of phase $\phi$, we need to have $Z[1]$ semistable of that phase on $\sigma '$.

The object of the good phase on $\sigma'$ is $A$, so the coefficient object of $Z[1]$ along $\sigma '$ should be $A$. 
That is to say, the
coefficient object $W'$ of $Z$ along $\sigma '$ should be $A[-1]$. 
From the above considerations, this gives $W= A$ for coefficient of $Z_{0,\zeta}$ along $\sigma$. 

We note that $U$ and $V$ don't have any mutual extensions, so the coefficient of $\epsilon _{\sigma} A_{\zeta}$ 
along $\sigma$ is just the direct sum $U\oplus V$. By looking at the picture, the morphisms between 
$Z_{0,\zeta}$ and $\epsilon _{\sigma} A_{\zeta}$  are wrapping morphisms at the real blow-up of $\eta$,
and these are just going to be morphisms in the category $\Ee$.

\begin{lemma}
\label{coefftwist}
The coefficient of the twist
$$
T_{Z_{\eta}} (\epsilon _{\sigma} A_{\zeta})
$$
along $\sigma$ is just the twist in the fiber category $\Ee = \mc A_2$
$$
{\rm coeff} = T_A(U\oplus V).
$$
\end{lemma}
\begin{proof}
This follows from the fact that the functor $\mc F_G (S;\mc E)\to \mc E$ which takes an object on the graph to the object of $\mc E$ on a fixed chosen edge is exact.
\end{proof}

We would like to show that this has the spectral network property along the edge $\sigma$. In view of the phase calculations, this amounts to saying that we would like this coefficient object in $\Ee$ to be a direct sum of objects of the form $C$. 
In view of the expressions of $U$ and $V$ in \eqref{expressUV}, this property is provided by the following lemma and this will complete the proof of the theorem in the case of the directional choice we have made for the pictures above.

\begin{lemma}
Consider, as before, the category $\Ee = \mc A_2$ with objects $A,B,C$ fitting into an exact triangle
$$
A \rightarrow B \rightarrow C \rightarrow A[1] .
$$
Let $T_A$ be the twist functor around the object $A$. 
Then 
$$
T_A(A) = 0, \;\;\;\; T_A(B) = C,  \;\;\;\; T_A(C) = C.
$$
\end{lemma}
\begin{proof}
By definition, $T_A(X)$ fits into the exact triangle as the  cone
$$
\Hom(A,X)\otimes A\rightarrow X \rightarrow T_A(X)
$$
over the evaluation map.
For $X=A$, we have $\Hom(A,A) =\mathbf k\cdot  1$, so $T_A(A)$ is the cone on the identity map $A\rightarrow A$, thus it is zero. 

For $X=B$, we also have $\Hom(A,B)=\mathbf k$ and the map 
$$
A\cong \Hom(A,B)\otimes A \rightarrow B
$$
is tautologically the standard map. Its  cone is $C$. 

For $X=C$ we note that $\Hom(A,C)=0$, thus, the twisting terms vanish and $T_A(C)$ is isomorphic to $C$. 
\end{proof}

As described above, in view of the expressions \eqref{expressUV} and
Lemma \ref{coefftwist}, this lemma completes the proof of Theorem \ref{twistsn}.

\subsection{Summary of $A_5\otimes A_2$ and further questions}
\label{subsec_summary}

We summarize the proof of Theorem \ref{thm_a5a2}, pointing out that the steps outlined in Subsection \ref{subsec_outline} are now filled in. 

\begin{itemize}
    \item 
From the B-side considerations of Subsection \ref{subsec_bside}, the category $\mc A_5\otimes \mc A_2$ over $\mK$ has a stability condition with the property of Proposition \ref{prop63}, that the action of $ \widetilde{SL(2,\mathbb Z)}$ permutes the subcategories of objects of different slopes. The generators for this action are the spherical twist functors. The action on the set of phases is transitive.

\item
Applying symmetries of the coefficient category --- geometrically a rotation by $\pi/3$ --- reduces the problem to finding spectral network representatives of semistable objects with phase in $[0,1/3)$.

\item
The proof for semistable objects of phase $\phi\in[0,1/3)$ is by induction on the number of spherical twists and dual twists necessary to get from phase $0$ to phase $\phi$. 

\item
The initiation of the recursion is given by the results of Subsection \ref{subsec_phase0}, giving the statement of Theorem \ref{thm_a5a2} for semistable objects of phase $0$, namely they have spectral network representatives. 

\item
For the induction step, we need to show that if Theorem  \ref{thm_a5a2} is known for a phase $\phi$, and if a spherical twist moves from $\phi$ to $\phi + \delta$, then Theorem  \ref{thm_a5a2} holds for $\phi + \delta$. This statement is given by Corollary \ref{inductiveconclusion} that depends on  Theorem \ref{twistsn}. The proof of that theorem has been completed in the previous subsection. 

\item
We note that the inductive process follows the picture in Figure \ref{fig_recursion}, indeed starting from each $\phi$ (aside from the initial one) there are two phases $\phi + \delta$ and $\phi + \delta '$ obtained from $\phi$ by the two generating spherical twists. In the pictures in the preceding subsection that choice corresponds to the choice of $\eta$ provided by Lemma \ref{twochoices}. 
\end{itemize}

Finally, let us formulate some open problems concerning the case of $A_5\otimes A_2$.
Notably, we have verified the main conjecture only in the case where the base is a regular hexagon and the stability condition on the coefficient category is also the most symmetric one.

\begin{prob}
Verify the main conjecture in the case where $(S,M)$ is a disk with six marked points and $\mc E=\mc A_2$ for more general configurations of marked points.
\end{prob}

One can hope that this is still within reach of the methods of this paper. This would mean that all spectral networks are still obtained by starting from the ones for some finite set of special phases and applying spherical twists (but not following the entire binary tree).

\begin{q}
What does the set of phases of semistable objects (mod 1) look like for non-regular hexagons? Is it finite, dense in some subset of $S^1$, or dense in all of $S^1$?
\end{q}

A probably quite difficult problem is to understand the entire 10-dimensional space of stability conditions of $\mc A_5\otimes \mc A_2$.

\begin{prob}
Determine $\mathrm{Stab}(\mc A_5\otimes \mc A_2)$ as a complex manifold with local homeomorphism to $\CC^{10}$ and, if possible, its wall-and-chamber structure.
Characterize those stability conditions which are of ``product type'', i.e. arise from a pair of stability conditions on $\mc A_5$ and $\mc A_2$.
\end{prob}

There are two known sources of stability conditions on $\mc A_5\otimes \mc A_2$: 1) from the identification with $D^b(E/(\ZZ/6))$ and slope stability, 2) from various finite-dimensional algebras $A$ with $D^b(A)\cong \mc A_5\otimes \mc A_2$ --- these each yield an open subset of $\mathrm{Stab}$. Besides $A_5\otimes A_2$, with various orientations on $A_5$, we have the algebra \eqref{beilinson10}, as well as many others obtained by mutation.
There is no overlap between the two, since for 1) the set of phases is dense, while for 2) it always has a gap around $0$.

\bibliographystyle{plain}
\bibliography{stabspec}

\Addresses

\end{document}